\pdfoutput=1

\documentclass[preprint,12pt,A4]{elsarticle}
\usepackage[top=2.5cm,bottom=2.5cm,left=2.5cm,right=2.5cm]{geometry}

\usepackage{amssymb}
\usepackage{latexsym}
\usepackage{graphicx}
\usepackage{amsmath,amsthm}
\usepackage{amssymb}
\usepackage{epsfig}
\usepackage{float,lscape}
\usepackage{rotating}
\usepackage{graphicx}
\usepackage{sidecap}
\usepackage{multirow}
\usepackage{float}
\usepackage{lipsum}
\usepackage{rotating}
\usepackage{algorithm}
\usepackage{algpseudocode}
\usepackage{lineno}
\usepackage{float,lscape}
\usepackage{rotating}
\algnewcommand{\algorithmicgoto}{\textbf{go to}}%
\algnewcommand{\Goto}[1]{\algorithmicgoto~\ref{#1}}
\newtheorem{thm}{Theorem}[section]

\theoremstyle{definition}
\newtheorem{defn}{Definition}[section]

\theoremstyle{remark}
\newtheorem{rem}{Remark}[section]

\journal{}

\begin{document}

\begin{frontmatter}


\author[]{Rabia Hameed \corref{cor1}}
\ead{Corresponding author: rabiahameedrazi@hotmail.com}
\address[]{Department of Mathematics, The Government Sadiq College Women University Bahawalpur \fnref{label3}}
\title{Variations in 2D and 3D models by a new family of subdivision schemes and algorithms for its analysis}




\begin{abstract}
A new family of combined subdivision schemes with one tension parameter is proposed by the interpolatory and approximating subdivision schemes. The displacement vectors between the points of interpolatory and approximating subdivision schemes provide the flexibility in designing the limit curves and surfaces. Therefore, the limit curves generated by the proposed subdivision schemes variate in between or around the approximating and interpolatory curves. We also design few analytical algorithms to study the properties of the proposed schemes theoretically. The efficiency of these algorithms are analyzed by calculating their time complexity. The graphical representations and graphical properties of the proposed schemes are also analyzed.
\end{abstract}


\end{frontmatter}


\section{Introduction}

A common problem in computer aided geometric design consists of drawing smooth curves and surfaces that either interpolate or approximate a given shape. Subdivision schemes are one of the most successful approaches in this area. Subdivision scheme are used to construct smooth curves and surfaces from a given set of control points through iterative refinements. Due to its clarity and simplicity, subdivision schemes have been esteemed in many fields such as image processing, computer graphics and computer animation. Generally, subdivision schemes can be categorized as interpolatory and approximating subdivision schemes. Interpolatory schemes get better shape control while approximating schemes have better smoothness. The most important interpolatory $4$-point binary subdivision scheme was proposed by Dyn et al. \cite{Dyn1} which gives the curves with $C^{1}$ continuity. This scheme was extended to the $6$-point binary interpolatory scheme by Weissman \cite{Weissman}. The scheme proposed in \cite{Weissman} gives the limiting curves with $C^{2}$ continuity. Most of the approximating schemes were developed from splines. Two of the most famous approximating schemes are Chaikin's subdivision scheme \cite{Chaikin} and cubic B-spline subdivision scheme \cite{Cohen}, which actually generate uniform quadratic and cubic B-spline curves with $C^{1}$ continuity and $C^{2}$ continuity, respectively. There also exists a class of parametric subdivision schemes, called combined subdivision schemes. Combined subdivision schemes can be regarded either as a member of the approximating schemes or of the interpolatory schemes according to the specific values assumed by the parameters. Pan et al. \cite{Pan} presented a combined ternary $3$-point relaxed subdivision scheme with three tension parameters. Novara and Romani \cite{Novara} extended their technique to construct a $3$-point relaxed ternary combined subdivision scheme with three tension parameters to generate curves up to $C^{3}$ continuity. Recently, Zhang et al. \cite{Zhang} also used similar technique to construct a $4$-point relaxed ternary scheme with four tension parameters to generate the curves up to $C^{4}$ smoothness. The nice performance of combined subdivision schemes motivates us to the direction of constructing the binary combined subdivision scheme. Our contribution in this manner is presented below.
\subsection{Motivation and contribution}
In this article, we define a novel technique for constructing combined subdivision schemes by using interpolatory and approximating subdivision schemes with optimal performance in curves and surfaces fitting. The optimal performance of the family of interpolatory schemes of \cite{Dubuc} in reproducing polynomials and good performance of the family of approximating B-spline schemes in generating curves with good continuity motivates us to unify them into a single family of schemes.

The $(2n+2)$-point interpolatory subdivision schemes defined in \cite{Dubuc} can reproduce polynomials up to degree $2n+1$ and generate the limiting curves with $C^{1}$. Whereas, the $(2n+2)$-point relaxed B-spline approximating schemes are $C^{4n}$-continuous, while the degree of polynomial reproduction of these schemes is only one. Since both type of schemes are primal schemes, hence each of them is the combination of one vertex and one edge rules. The vertex rule of interpolatory schemes keep the initial vertices at every level of subdivision, while the vertex rule of the approximating schemes update the initial vertices at each level of subdivision. Similarly, the edges rule of each scheme subdivides every edge of the given polygon into two edges at each level of subdivision. The limit curves generated by both type of schemes after 4 subdivision levels are shown in Figure \ref{motivation}. In this article, we construct a family of combined subdivision schemes with one parameter which controls the shape of the limit curve.

We construct the new family of schemes by the family of interpolatory schemes defined in \cite{Dubuc} and the family of approximating B-spline schemes. Therefore, the proposed family of schemes carry the properties of both types of schemes and give optimal numerical and mathematical results. From Figures \ref{motivation1}-\ref{motivation2}, we can see the better performance of the proposed schemes than that of their parent schemes (schemes which are used in their construction). Figure \ref{motivation1} show that if we choose the value of tension parameter for the proposed $4$-point scheme between -2.5 and -2 than this scheme is good for fitting noisy data. Figure \ref{motivation2}(c) shows that the curves fitted by the proposed $6$-point scheme do not give the artifacts as presented by the $6$-point scheme \cite{Dubuc} in Figure \ref{motivation2}(a). This figure also shows that the curve fitted by the proposed scheme preserve the shape of the initial polygon if we take value of tension parameter between -0.5 and 0, while the curve fitted by the $6$-point B-spline scheme does not have this ability as shown in Figure \ref{motivation2}(b).

The remainder of this article is organized as: Section 2 deals with some \emph{}basic definition and basic results. Section 3 provides the framework for constructing a family of combined binary subdivision schemes. In Section 4, we provide the analytical properties of the proposed family of schemes. Section 5 gives the graphical behaviors and graphical properties of the proposed family of schemes. Concluding remarks are presented in Section 6.
\section{Preliminaries}
In this section, we present some basics notations and definition relating to the results which we have used in this article.

A general compact form of linear, uniform and stationary binary univariate subdivision scheme $S_{a}$ which maps a polygon $P^{k}=\{P^{k}_{i}, i \in \mathbb{Z}\}$ to a refined polygon $P^{k+1}=\{P^{k+1}_{i}, i \in \mathbb{Z}\}$ is defined as
\begin{eqnarray}\label{subdivisionscheme}
  P_{i}^{k+1} &=& \sum \limits_{j \in \mathbb{Z}}a_{i -2 j} P_{j}^{k}, \quad i \in \mathbb{Z}.
\end{eqnarray}
The symbol of the above subdivision scheme is given by the Laurent polynomial
\begin{eqnarray}\label{polynomial}
  a(z) &=& \sum \limits_{i \in \mathbb{Z}}a_{i}z^{i}, \quad z \in \mathbb{C}\setminus \{0\},
\end{eqnarray}
where $a= \{a_{i}, i \in \mathbb{Z}\}$ is called the mask of the subdivision scheme. The necessary condition for the binary subdivision scheme (\ref{subdivisionscheme}) to be convergent is that its mask satisfies the basic sum rule
\begin{eqnarray}\label{sumrule}
\sum\limits_{j \in \mathbb{Z}}a_{i-2j}&=&1, \,\ i=0,1.
\end{eqnarray}
Which is equivalent to the following relation
\begin{eqnarray}\label{sumrule1}
&& a(1)=2 \,\ \mbox{and} \,\  a(-1)=0.
\end{eqnarray}
\begin{defn}
An algorithm is a procedure for solving a mathematical problem in a finite number of steps that frequently involves repetition of an operation.
\end{defn}
\begin{defn}
In computer science, the time complexity is the computational complexity that describes the amount of time it takes to run an algorithm. It is denoted by $T(n)$, where $n$ is the input of the given algorithm. The time complexity indicates the number of times the operations are executed in an algorithm. The number of instructions executed by a program is affected by the size of the input and how their elements are arranged.
\end{defn}
\begin{defn}
 Step Count Method consider each and every step to the given algorithm for calculating the final time complexity of the given algorithm. So each and every step in given algorithm contributes to the final complexity of given algorithm.
\end{defn}
\begin{defn}
Worst-case time complexity denotes the longest running time $\mathcal{O}(n)$ of an algorithm given any input of size $n$. It gives an upper bound on time requirements.
\end{defn}
\begin{defn}
Best-case time complexity denotes the shortest running time $\Omega(n)$ of an algorithm given any input of size $n$. It gives a lower bound on time requirements.
\end{defn}
\section{Family of the combined primal schemes}
In this section, we discuss the construction of the family of $(2n+2)$-point binary relaxed schemes. These schemes are the binary schemes, so are the combination of two refinement rules. One of these rules, consists of the affine combination of $(2n+1)$-point of level $k$ to insert a new point at level $k+1$, is used to update the vertex of the given polygon hence is called the vertex rule. The other refinement rule, consists of the affine combination of $(2n+2)$-points of level $k$ to insert a new point at level $k+1$, is used for subdividing each edge of the given polygon so is called the edge rule. Construction of the family of schemes is discussed in coming subsections in detail.
\begin{figure}[h!] 
\begin{center}
\begin{tabular}{c}
\epsfig{file=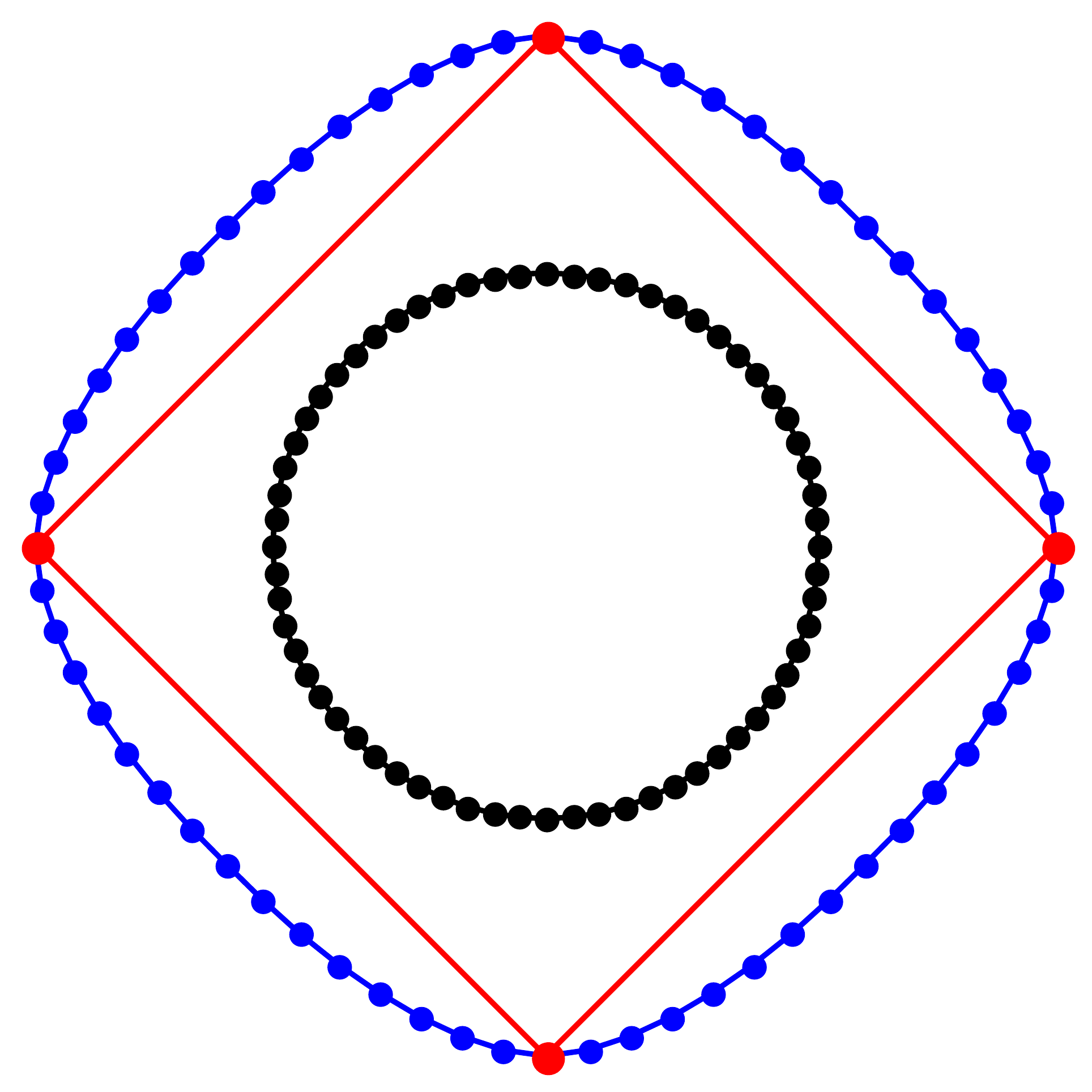, width=3.5 in}
\end{tabular}
\end{center}
 \caption[]{\label{motivation}\emph{Red bullets and red lines represent initial points and initial polygon respectively. Blue bullets and blue curve are the points and limiting curve generated by the $4$-point scheme of \cite{Dubuc} respectively after 4 subdivision levels. Whereas black bullets and black curve are the points and limiting curve generated by the $4$-point relaxed B-spline scheme respectively after 4 subdivision levels.}}
\end{figure}

\begin{figure}[h!] 
\begin{center}
\begin{tabular}{ccc}
\epsfig{file=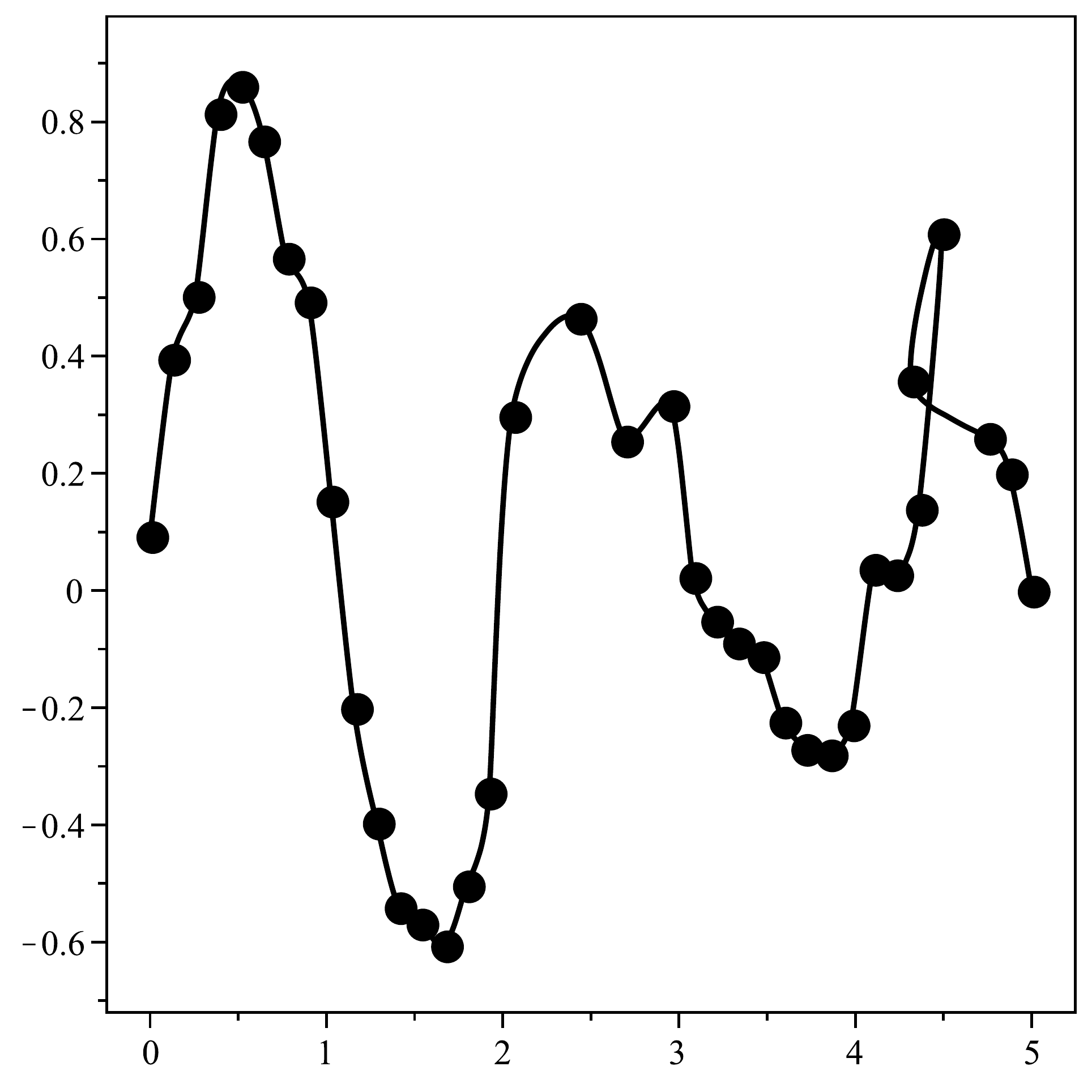, width=2.0 in} & \epsfig{file=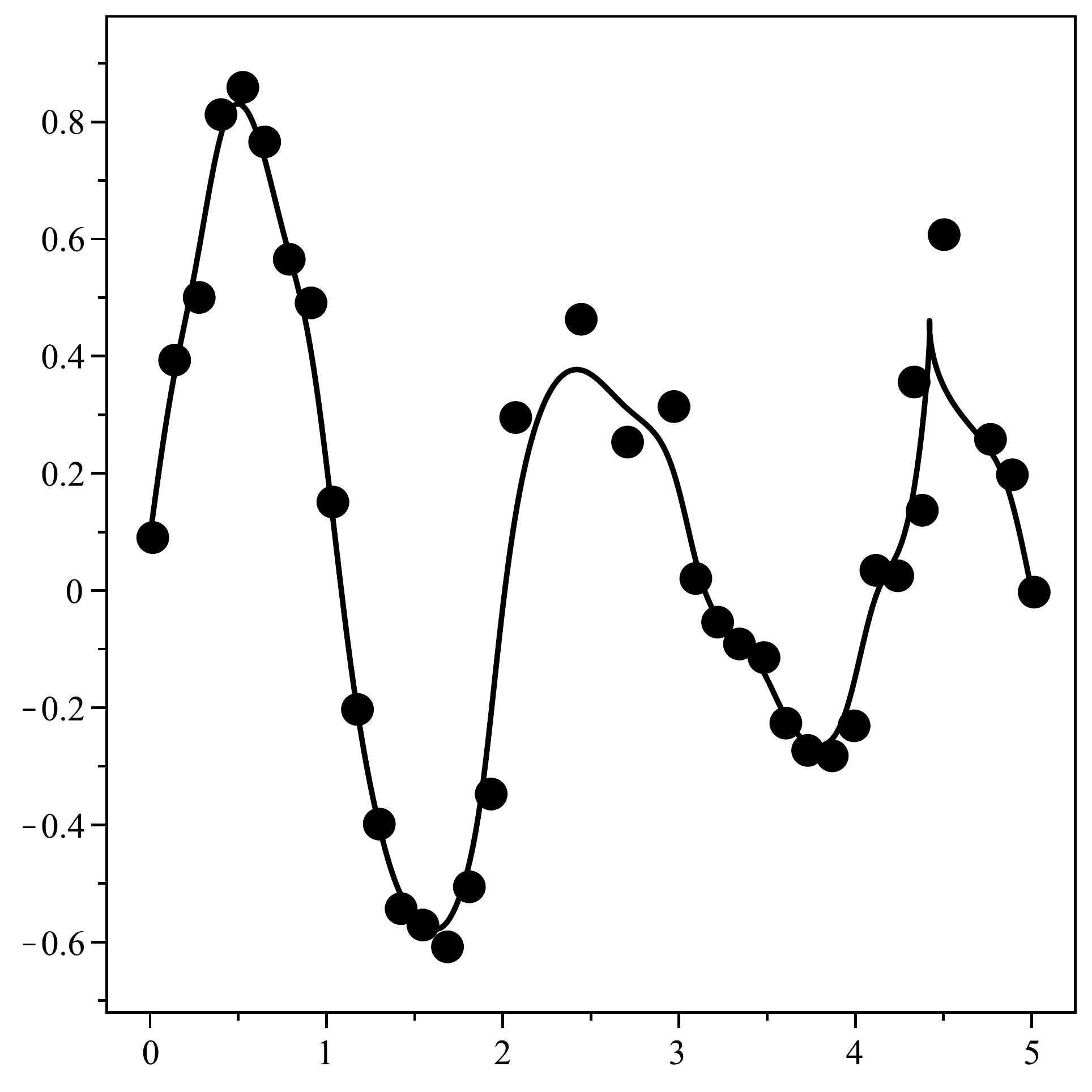, width=2.0 in} & \epsfig{file=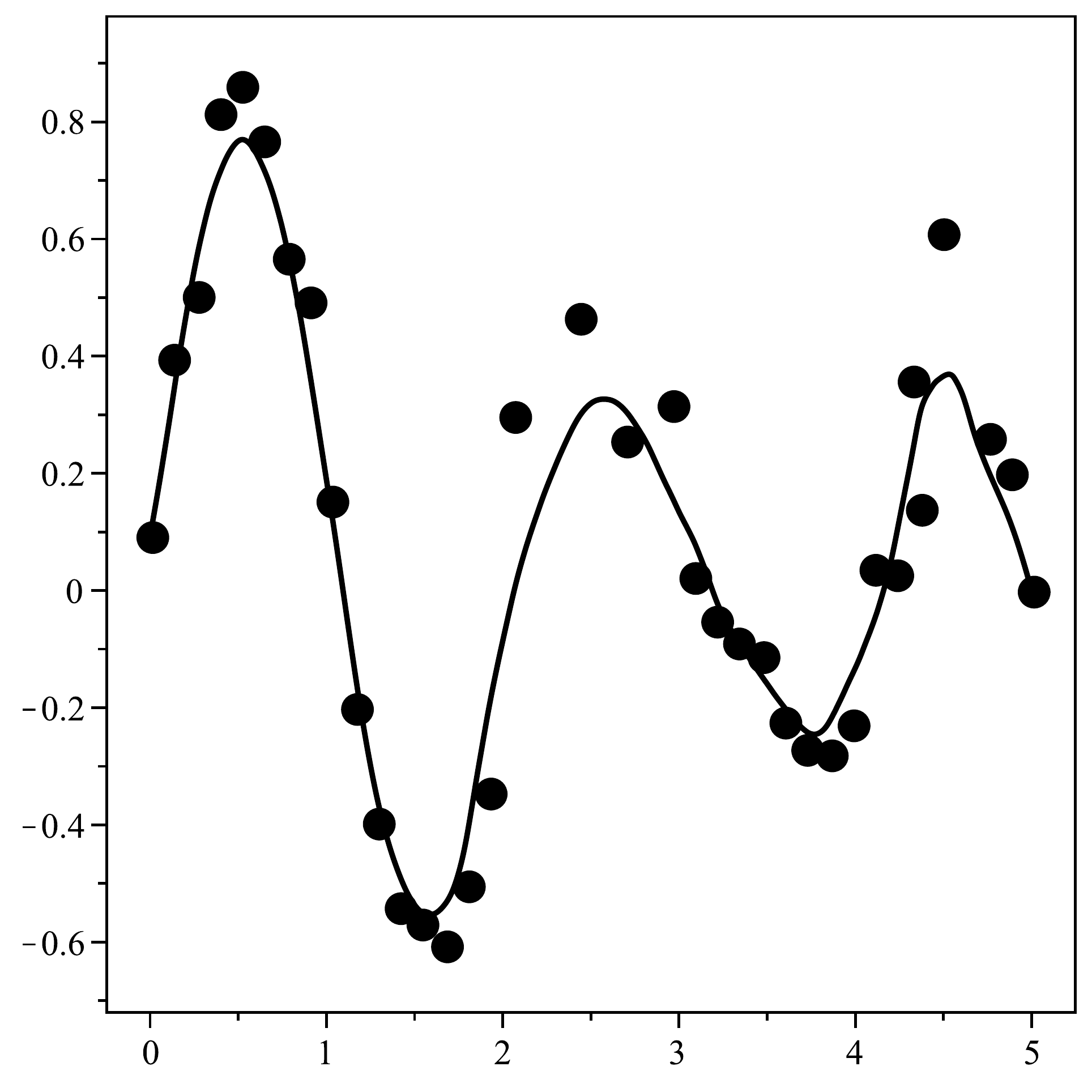, width=2.0 in} \\
(a) & (b) & (c) \\
$4$-point interpolatory scheme & $4$-point B-spline scheme & Proposed $4$-point scheme \\
& & for $\alpha=-2.5$
\end{tabular}
\end{center}
 \caption[]{\label{motivation1}\emph{Black bullets are the initial points, whereas the solid lines show the limit curves generated by the binary primal schemes.
 }}
\end{figure}

\begin{figure}[h!] 
\begin{center}
\begin{tabular}{ccc}
\epsfig{file=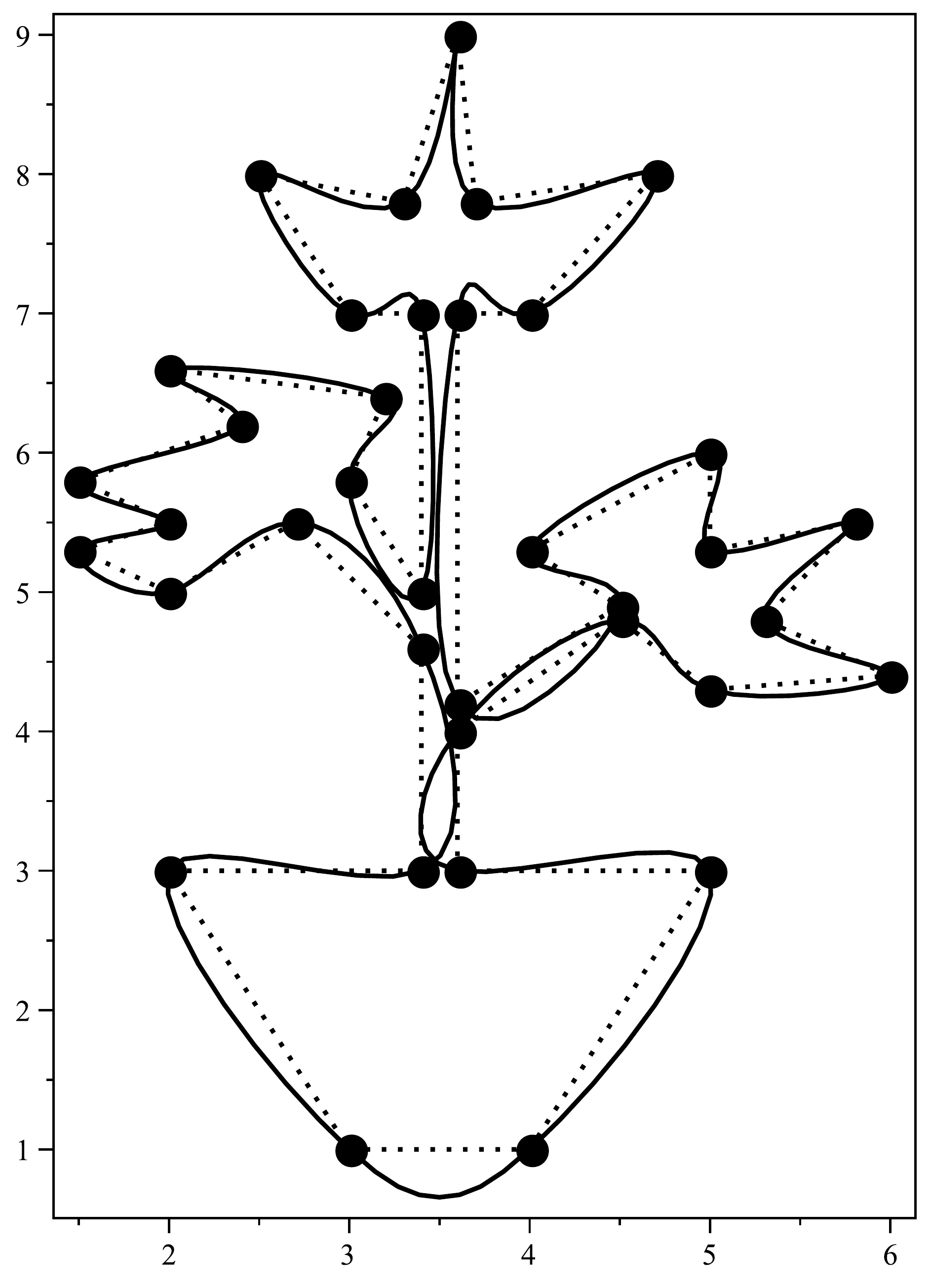, width=2.0 in} & \epsfig{file=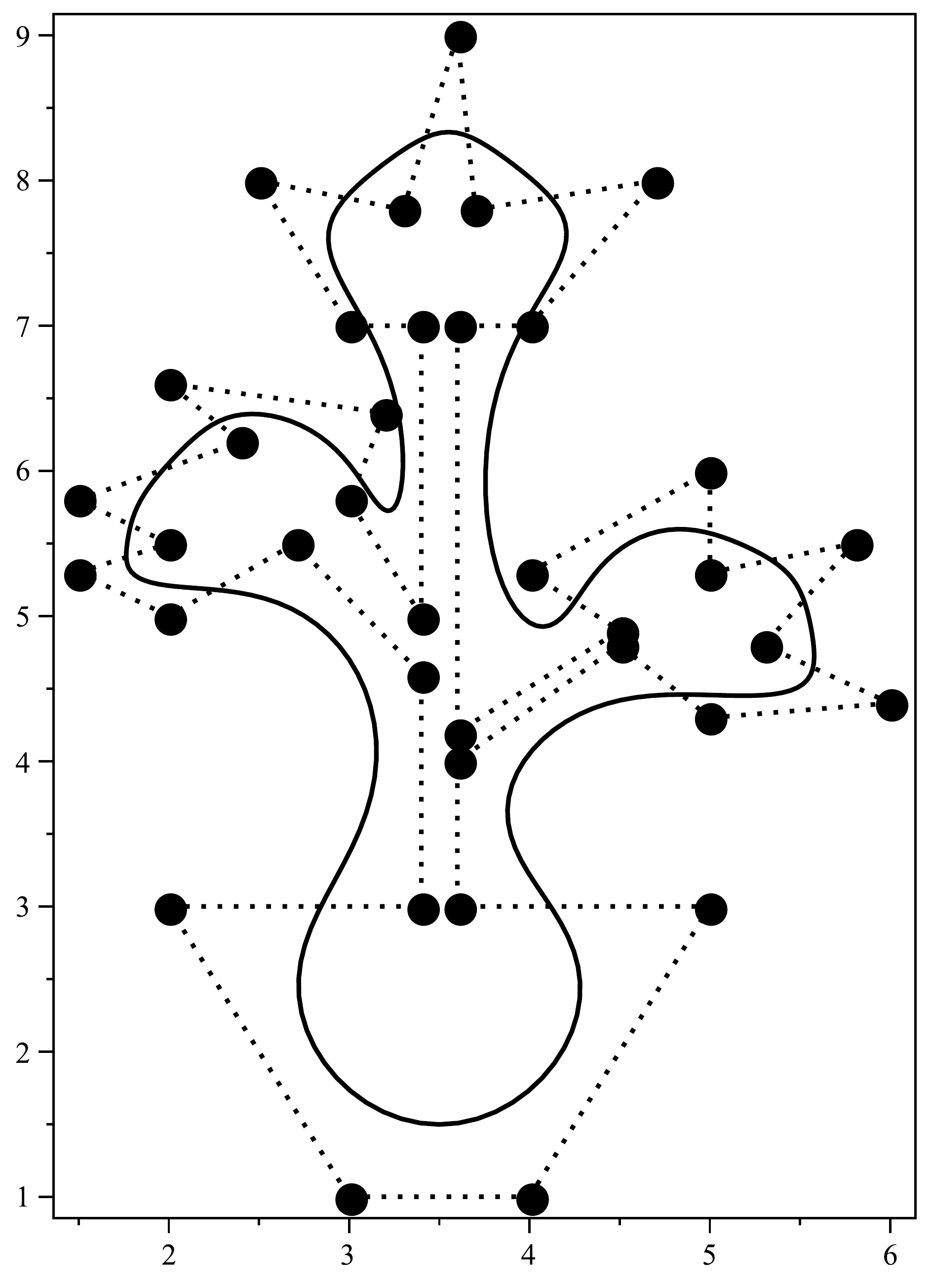, width=2.0 in} & \epsfig{file=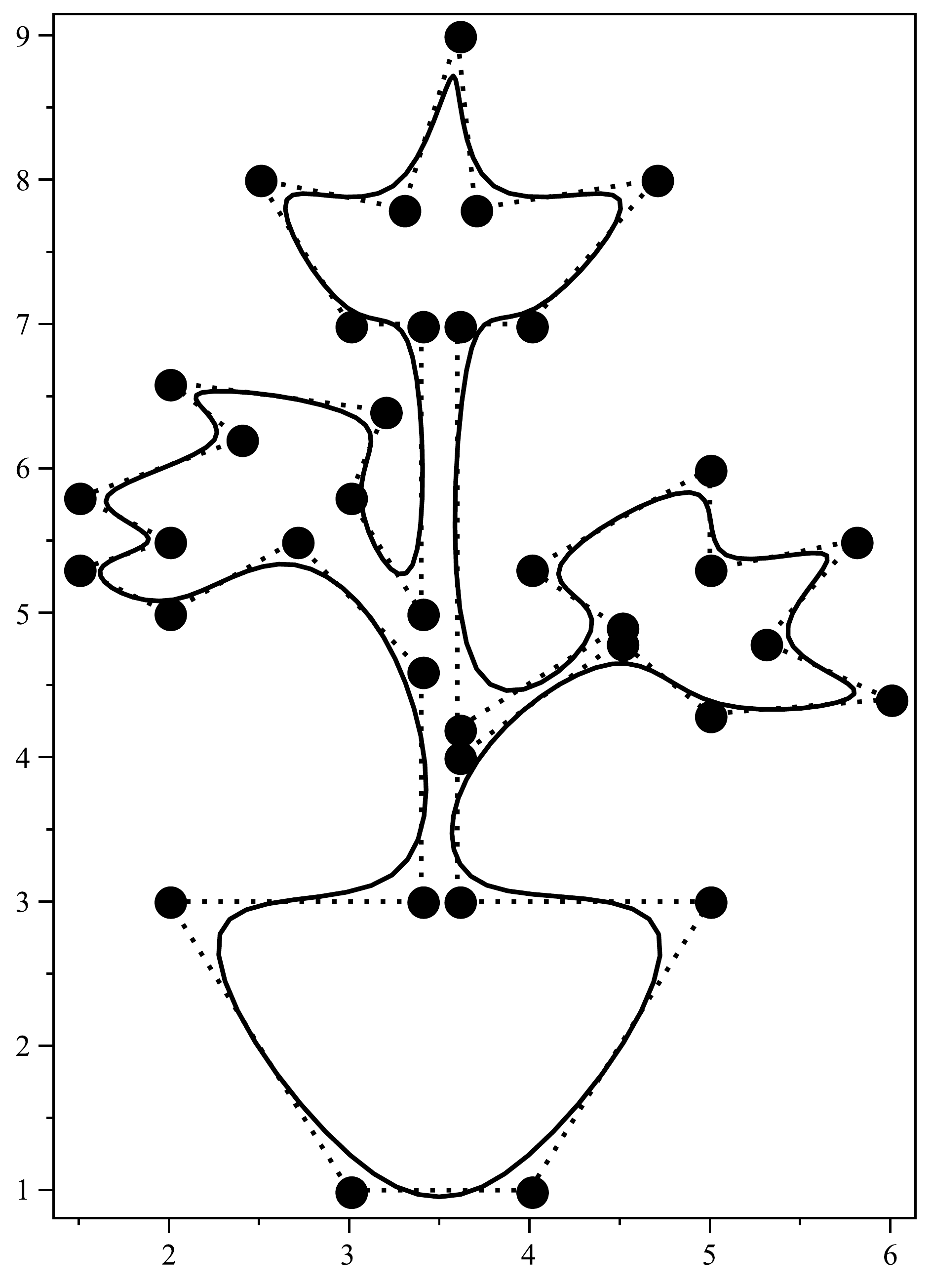, width=2.0 in} \\
(a) & (b) & (c) \\
$6$-point interpolatory scheme & $6$-point B-spline scheme & Proposed $6$-point scheme \\
& & for $\alpha=-0.32$
\end{tabular}
\end{center}
 \caption[]{\label{motivation2}\emph{Dotted lines and black bullets are the initial polygons and initial points respectively. Whereas the solid lines show the limit curves generated by the binary primal schemes.
 }}
\end{figure}

\subsection{Framework for the construction of family of subdivision schemes}
The family of $(2n+2)$-point interpolatory subdivision schemes, constructed by the Lagrange's interpolatory polynomials in \cite{Dubuc}, is
\begin{eqnarray}\label{DD}
\left\{\begin{array}{ccc}
 R_{2i,2n+2}^{k+1} &=&P_{i,2n+2}^{k}, \\ \\
 R_{2i+1,2n+2}^{k+1} &=&\frac{1}{2^{4n+1}}\sum\limits_{j=-n-1}^{n}\frac{(-1)^{j}(n+1)}{(2j+1)}\left(\begin{array}{c}2n+1\\n\end{array}\right)\left(\begin{array}{c}2n+1\\n+j+1\end{array}\right)P^{k}_{i+j+1,2n+2}.
\end{array}\right.
\end{eqnarray}
The refinement rules of the family of $(2n+2)$-point relaxed approximating schemes, i.e. degree-$(4n+1)$ binary B-spline subdivision schemes, are
\begin{eqnarray}\label{BSpline}
\left\{\begin{array}{ccc}
  Q_{2i,2n+2}^{k+1} &=&\frac{1}{2^{4n+1}}\sum\limits_{j=0}^{2n}\left(\begin{array}{c}4n+2\\2j+1\end{array}\right)P_{i+j-n,2n+2}^{k}, \\ \\
  Q_{2i+1,2n+2}^{k+1} &=&\frac{1}{2^{4n+1}}\sum\limits_{j=0}^{2n+1}\left(\begin{array}{c}4n+2\\2j\end{array}\right)P_{i+j-n,2n+2}^{k},
\end{array}\right.
\end{eqnarray}
where $n \in \mathbb{Z}^{+}$. Let us denote the displacement vectors from the degree-$(4n+1)$ B-spline refinement points $Q_{2i,2n+2}^{k+1}$ $\&$ $Q_{2i+1,2n+2}^{k+1}$ to the refinement points $R_{2i,2n+2}^{k+1}$ $\&$ $R_{2i+1,2n+2}^{k+1}$ of $(2n+2)$-point interpolatory subdivision scheme defined in (\ref{DD}) after one step of refinements by $D_{2i,2n+2}^{k+1}$ $\&$ $D_{2i+1,2n+2}^{k+1}$ respectively. Hence the displacement vectors are
\begin{eqnarray}\label{vectors}
\left\{\begin{array}{ccc}
  D_{2i,2n+2}^{k+1} &=& R_{2i,2n+2}^{k+1}-Q_{2i,2n+2}^{k+1}, \\ \\
  D_{2i+1,2n+2}^{k+1} &=& R_{2i+1,2n+2}^{k+1}-Q_{2i+1,2n+2}^{k+1}.
  \end{array}\right.
\end{eqnarray}
Now a new family of primal combined $(2n+2)$-point subdivision schemes can be obtained by translating the points $R_{2i,2n+2}^{k+1}$ $\&$ $R_{2i+1,2n+2}^{k+1}$, that are provided by the binary interpolatory subdivision schemes (\ref{DD}) to the new position according to the displacement vectors $\alpha D_{2i,2n+2}^{k+1}$ $\&$ $\alpha D_{2i+1,2n+2}^{k+1}$ respectively.
\begin{eqnarray}\label{vectors1}
\left\{\begin{array}{ccc}
  P_{2i,2n+2}^{k+1} &=& R_{2i,2n+2}^{k+1}+\alpha D_{2i,2n+2}^{k+1}, \\ \\
  P_{2i+1,2n+2}^{k+1} &=& R_{2i+1,2n+2}^{k+1}+\alpha D_{2i+1,2n+2}^{k+1}.
  \end{array}\right.
\end{eqnarray}
Which can be further written as
\begin{eqnarray}\label{proposed}
\left\{\begin{array}{ccc}
  P_{2i,2n+2}^{k+1} &=&\frac{-\alpha}{2^{4n+1}}\sum\limits_{j=0,j\neq n}^{2n}\left(\begin{array}{c}4n+2\\2j+1\end{array}\right)P_{i+j-n,2n+2}^{k}+\frac{1}{2^{4n+1}}\left[2^{4n+1}(1+\alpha)-\alpha \times \right.\\&&\left. \left(\begin{array}{c}4n+2\\2n+1\end{array}\right)\right]P_{i,2n+2}^{k}, \\ \\
  P_{2i+1,2n+2}^{k+1} &=&\frac{1}{2^{4n+1}}\sum\limits_{j=0}^{2n+1}\left[\frac{(-1)^{j-n-1}(n+1)}{2j-2n-1}\left(\begin{array}{c}2n+1\\n\end{array}\right)\left(\begin{array}{c}2n+1\\j\end{array}\right)+\alpha\left\{\frac{(-1)^{j-n-1}(n+1)}{2j-2n-1}\times \right.\right.\\&&\left.\left.
  \left(\begin{array}{c}2n+1\\n\end{array}\right)\left(\begin{array}{c}2n+1\\j\end{array}\right) -\left(\begin{array}{c}4n+2\\2j\end{array}\right)\right\}\right]P_{i+j-n,2n+2}^{k}.
\end{array}\right.
\end{eqnarray}
The rules defined in (\ref{proposed}) are the refinement rules of the proposed family of $(2n+2)$-point combined schemes.
Which are equivalent to the following general form
\begin{eqnarray}\label{proposed1}
\left\{\begin{array}{ccc}
  P_{2i,2n+2}^{k+1} &=&\sum\limits_{j=0}^{2n}a_{2j+1,2n+2}P_{i+j-n,2n+2}^{k}, \\ \\
  P_{2i+1,2n+2}^{k+1} &=&\sum\limits_{j=0}^{2n+1}a_{2j,2n+2}P_{i+j-n,2n+2}^{k},.
\end{array}\right.
\end{eqnarray}
The mask symbol of the above scheme is
\begin{eqnarray}\label{symbol-proposed1}
a_{2n+2}(z)&=&\sum\limits_{j=0}^{4n+2}a_{j,2n+2}z^{j}=\frac{(1+z)^{2n+2}}{2^{2n+1}}A_{2n+2}(z),
\end{eqnarray}
with
\begin{eqnarray}\label{symbol-proposed2}
A_{2n+2}(z)&=&\sum\limits_{j=0}^{2n}A_{j,2n+2}z^{j}.
\end{eqnarray}
where $a_{j,2n+2}:j=0,1,\ldots,4n+2$ and $A_{j,2n+2}:j=0,1,\ldots,2n$ are the coefficients of $z^{j}:j=0,1,\ldots,4n+2$ and $z^{j}:j=0,1,\ldots,2n$ in (\ref{symbol-proposed1}) and (\ref{symbol-proposed2}) respectively.
\begin{rem}
If $\alpha \rightarrow 0$, the family of proposed $(2n+2)$-point schemes converges to the family of $(2n+2)$-point interpolatory schemes proposed in \cite{Dubuc}. Whereas if $\alpha \rightarrow -1$, the family of proposed $(2n+2)$-point schemes converges to the family of $(2n+2)$-point relaxed approximating B-spline schemes of degree-$(4n+1)$.
\end{rem}
\subsection{The geometrical and mathematical interpretation of the family of schemes}
\begin{figure}[htb] 
\begin{center}
\begin{tabular}{c}
\epsfig{file=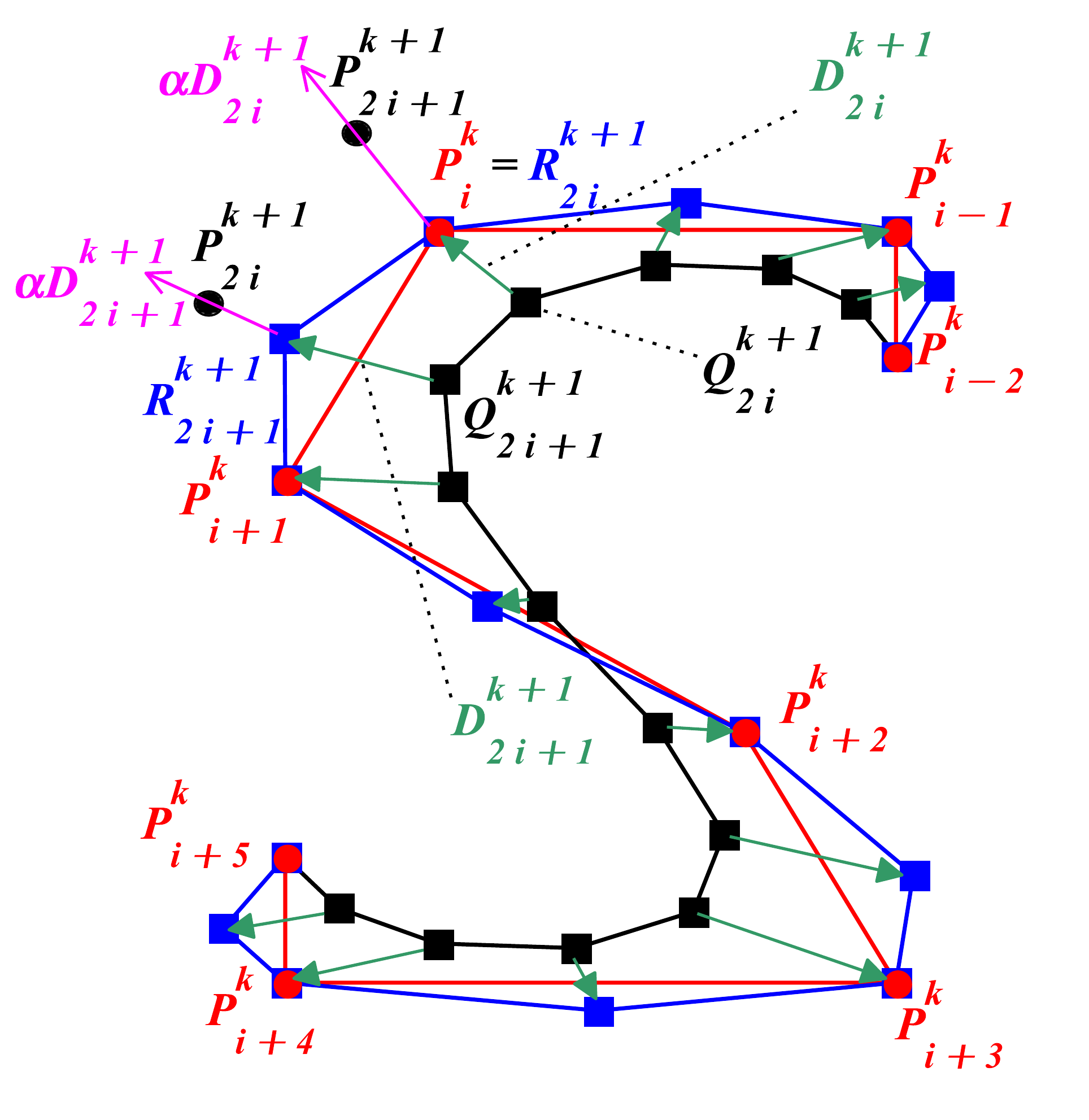, width=3.5 in}
\end{tabular}
\end{center}
 \caption[]{\label{construction}\emph{Geometrical construction of the scheme for $k=0$ and $n=1$.}}
\end{figure}
The family of proposed schemes (\ref{proposed}) is obtained by moving the points $R_{2i,2n+2}^{k+1}$ and $R_{2i+1,2n+2}^{k+1}$ of the family of schemes (\ref{DD}) to the new position according to the displacement vectors $\alpha D_{2i,2n+2}^{k+1}$ and $\alpha D_{2i+1,2n+2}^{k+1}$ respectively. In order to make the construction of the family of schemes understandable, we explain the following steps by fixing $n=1$ and $k=0$.

 \textbf{Step 1:} We take the points $P^{0}_{i-2,4}=P^{k}_{i-2}$, $P^{0}_{i-1,4}=P^{k}_{i-1}$, $P^{0}_{i,4}=P^{k}_{i}$, $P^{0}_{i+1,4}=P^{k}_{i+1}$, $P^{0}_{i+2,4}=P^{k}_{i+2}$, $P^{0}_{i+3,4}=P^{k}_{i+3}$, $P^{0}_{i+4,4}=P^{k}_{i+4}$ and $P^{0}_{i+5,4}=P^{k}_{i+5}$ as the initial control points, i.e. at zeroth level of subdivision. The points are shown by red bullets in Figure \ref{construction}. In this figure, the combination of red lines make the initial polygon.

 \textbf{Step 2:} Now we calculate points $R_{2i}^{k+1}$ and $R_{2i+1}^{k+1}$ which are denoted by blue solid squares in Figure \ref{construction}. We calculate fifteen points (blue solid squares) by applying the following refinement rules of the $4$-point interpolatory scheme of Deslauriers and Dubuc \cite{Dubuc} on the initial points
\begin{eqnarray}\label{11}
\left\{\begin{array}{ccc}
R_{2i}^{k+1}=R_{2i,4}^{k+1} &=& P_{i}^{k}, \\ \\
R_{2i+1}^{k+1} =R_{2i+1,4}^{k+1} &=& -\frac{1}{16}P^{k}_{i-1}+\frac{9}{16}P^{k}_{i}+\frac{9}{16}P^{k}_{i+1}-\frac{1}{16}P^{k}_{i+2}.
\end{array}\right.
\end{eqnarray}
All the points inserted by the above subdivision scheme are shown in Figure \ref{construction}, however we restrict our next calculations to only two points $R_{2i}^{k+1}$ and $R_{2i+1}^{k+1}$. In this figure, blue polygon is the initial polygon obtained from the scheme (\ref{11}).

\textbf{Step 3:} In this step, we calculate and restrict next calculations to only the points $Q_{2i}^{k+1}$ and $Q_{2i+1}^{k+1}$ denoted by black solid squares. Again the fifteen points, denoted by black solid squares in Figure \ref{construction}, are calculated by the refinement rules of the $4$-point relaxed binary B-spline scheme of degree-5 which is defined below
\begin{eqnarray}\label{22}
\left\{\begin{array}{ccc}
 Q_{2i}^{k+1}=Q_{2i,4}^{k+1} &=& \frac{3}{16}P^{k}_{i-1}+\frac{10}{16}P^{k}_{i}+\frac{3}{16}P^{k}_{i+1}, \\ \\
Q_{2i+1}^{k+1}=Q_{2i+1,4}^{k+1} &=& \frac{1}{32}P^{k}_{i-1}+\frac{15}{32}P^{k}_{i}+\frac{15}{32}P^{k}_{i+1}+\frac{1}{32}P^{k}_{i+2}.
\end{array}\right.
\end{eqnarray}
In this figure, black polygon is the initial polygon made by the above scheme.

\textbf{Step 4:} Now we calculate the displacement vectors from the points defined in (\ref{22}) to the points defined in (\ref{11}). These vectors are $D_{2i}^{k+1}=D_{2i,4}^{k+1}=R_{2i}^{k+1}-Q_{2i}^{k+1}=-\frac{3}{16}P^{k}_{i-1}+\frac{3}{8}P^{k}_{i}-\frac{3}{16}P^{k}_{i+1}$ and $D_{2i+1}^{k+1}=D_{2i+1,4}^{k+1}=R_{2i+1}^{k+1}-Q_{2i+1}^{k+1}=-\frac{3}{32}P^{k}_{i-1}+\frac{3}{32}P^{k}_{i}+\frac{3}{32}P^{k}_{i+1}-\frac{3}{32}P^{k}_{i+2}$. In Figure \ref{construction}, the above displacement vectors and the other related vectors are shown by green vectors.

\textbf{Step 5:} In this step, we calculate vectors $\alpha D_{2i}^{k+1}$ and $\alpha D_{2i+1}^{k+1}$, where $\alpha$ is any real number that can be change the magnitude or the direction or both of the vectors $D_{2i}^{k+1}$ and $D_{2i+1}^{k+1}$. In Figure \ref{construction}, these vectors are denoted by magenta lines.

\textbf{Step 6:} The last step of our method is to translate the points obtained by the refinement rules in (\ref{11}) to the new position by using displacement vectors $\alpha D_{2i}^{k+1}$ and $\alpha D_{2i+1}^{k+1}$. So we get following refinement rules of the new combined $4$-point relaxed scheme
\begin{eqnarray*}
P_{2i}^{k+1}=P_{2i,4}^{k+1} &=& -\frac{3}{16}\alpha P^{k}_{i-1}+\left(1+\frac{3}{8}\alpha \right)P^{k}_{i}-\frac{3}{16}\alpha P^{k}_{i+1}, \\
P_{2i+1}^{k+1}=P_{2i+1,4}^{k+1} &=& -\left(\frac{1}{16}+\frac{3}{32}\alpha \right)P^{k}_{i-1}+\left(\frac{9}{16}+\frac{3}{32}\alpha \right)P^{k}_{i}+\left(\frac{9}{16}+\frac{3}{32}\alpha \right)P^{k}_{i+1}\\&&-\left(\frac{1}{16}+\frac{3}{32}\alpha \right)P^{k}_{i+2}.
\end{eqnarray*}
In Figure \ref{construction}, the points of the above proposed scheme are denoted by black solid circles.
\begin{rem}
For the positive values of tension parameter $\alpha$, the curves generated by the proposed subdivision schemes lie outside the initial polygon, whereas for negative values of $\alpha$, the curves lie inside the initial polygon. For $\alpha=0$, the curve generated by the proposed schemes passes through the initial control points.
\end{rem}

\section{Analysis of the proposed family of subdivision schemes}
In this section, we prove the characteristics of the proposed subdivision schemes analytically. We prove few theorems and present various algorithms to check these properties.
\subsection{Support}
The support of a subdivision scheme represents the portion of the limit curve effected by the displacement of a single control point from its initial place. Now we calculate the support of the proposed family of subdivision schemes.
\begin{thm}
If $\phi_{2n+2}$ is the basic limit function of the family of subdivision schemes (\ref{proposed}). Then its support is $supp(\phi_{2n+2})=[-(2n+1),(2n+1)]$.
\end{thm}
\begin{proof}
 To calculate support of the proposed family of subdivision schemes, we define a function which give zero value to all control points except one point, $P_{0}^{0}$, whose value is one.
\begin{equation}\label{BF}
\phi(i) = P_{i}^{0}= \begin{cases}
0  & \text{for $i=0$}, \\
 1  & \text{for  $i \neq 0$}.
\end{cases}
\end{equation}
  Let $\phi_{2n+2}:\mathbb{R}\rightarrow\mathbb{R}$ be the basic limit function such that $\phi_{2n+2}\left(\frac{i}{2^{k}}\right)=P_{i,2n+2}^{k}$. Moreover,  $\phi(i)=\phi_{2n+2}(i)$ and $P_{i}^{0}=P_{i,2n+2}^{0}$. Therefore, the support of basic limit function $\phi_{2n+2}$ is equal to the support of the subdivision scheme $S_{a_{2n+2}}$.

When we apply a proposed $(2n+2)$-point relaxed scheme on the initial data $P_{i}^{0}$, then at first subdivision step the non-zero vertices are \\ $\phi_{2n+2}\left(\frac{-2^{0}(2n+1)}{2^{1}}\right)$ $=$ $P^{1}_{-2^{0}(2n+1),2n+2}$, $P^{1}_{-2^{0}(2n+1)+1,2n+2}$, $\ldots$, $P^{1}_{2^{0}(2n+1)-1,2n+2}$, $P^{1}_{2^{0}(2n+1),2n+2}$ $=$ $\phi_{2n+2}\left(\frac{2^{0}(2n+1)}{2^{1}}\right)$.

At second subdivision step the non-zero vertices are \\ $\phi_{2n+2}\left(\frac{-\{2^{0}+2^{1}\}(2n+1)}{2^{2}}\right)$ $=$ $P^{2}_{-\{2^{0}+2^{1}\}(2n+1),2n+2}$, $P^{2}_{-\{2^{0}+2^{1}\}(2n+1)+1,2n+2}$, $\ldots$, $P^{2}_{\{2^{0}+2^{1}\}(2n+1)-1,2n+2}$, $P^{2}_{\{2^{0}+2^{1}\}(2n+1),2n+2}$ $=$ $\phi_{2n+2}\left(\frac{\{2^{0}+2^{1}\}(2n+1)}{2^{2}}\right)$.

At third subdivision step the non-zero vertices are \\ $\phi_{2n+2}\left(\frac{-\{2^{0}+2^{1}+2^{2}\}(2n+1)}{2^{3}}\right)$ $=$ $P^{3}_{-\{2^{0}+2^{1}+2^{2}\}(2n+1),2n+2}$, $P^{3}_{-\{2^{0}+2^{1}+2^{2}\}(2n+1)+1,2n+2}$, $\ldots$,\\ $P^{3}_{\{2^{0}+2^{1}+2^{2}\}(2n+1)-1,2n+2}$, $P^{3}_{\{2^{0}+2^{1}+2^{2}\}(2n+1),2n+2}$ $=$ $\phi_{2n+2}\left(\frac{\{2^{0}+2^{1}+2^{2}\}(2n+1)}{2^{3}}\right)$.

Similarly, at $k$-th subdivision step the non-zero vertices are \\ $\phi_{2n+2}\left(\frac{-\sum\limits_{j=0}^{k-1}2^{j}(2n+1)}{2^{k}}\right)$ $=$ $P^{k}_{-\sum\limits_{j=0}^{k-1}2^{j}(2n+1),2n+2}$, $P^{k}_{-\sum\limits_{j=0}^{k-1}2^{j}(2n+1)+1,2n+2}$, $\ldots$, $P^{k}_{\sum\limits_{j=0}^{k-1}2^{j}(2n+1)-1,2n+2}$,\\ $P^{k}_{\sum\limits_{j=0}^{k-1}2^{j}(2n+1),2n+2}$ $=$ $\phi_{2n+2}\left(\frac{\sum\limits_{j=0}^{k-1}2^{j}(2n+1)}{2^{k}}\right)$.

Hence the support size of basic limit function is the difference between subscripts of the maximum non-zero vertex and the minimum non-zero vertex, i.e.
\begin{eqnarray*}
\mbox{support} \,\ \mbox{size} \,\ \mbox{of} \,\ \phi_{2n+2}&=&\left[\frac{\sum\limits_{j=0}^{k-1}2^{j}(2n+1)}{2^{k}}-\frac{-\sum\limits_{j=0}^{k-1}2^{j}(2n+1)}{2^{k}}\right]\\ &=& 2(2n+1)\left[\frac{\sum\limits_{j=0}^{k-1}2^{j}}{2^{k}}\right]=2(2n+1)\sum\limits_{j=0}^{k}\frac{1}{2^{j}}.
\end{eqnarray*}
Applying limit $k \rightarrow \infty$ we get, support size of $\phi_{2n+2}$ $=$ $2(2n+1)$. Hence $supp(\phi_{2n+2})=[-(2n+1),(2n+1)]$ is the support region of $\phi_{2n+2}$.
\end{proof}
\subsection{Smoothness analysis}
Now we calculate the order of parametric continuity of the proposed subdivision schemes by using Laurent polynomial method. Detailed information about refinement rules, Laurent polynomials and convergence of a subdivision scheme can be found in \cite{Charina, Dyn6, Dyn2}.
The continuity of the subdivision schemes can be analyzed by the following theorems.

\begin{thm}\label{thmcontinuity}
\cite{Dyn6} A convergent subdivision scheme $S_{a}$ corresponding to the symbol
\begin{eqnarray*}
a(z)&=&\left(\frac{1+z}{2\, z}\right)^{n}b(z),
\end{eqnarray*}
 is $C^{n}$-continuous iff the subdivision scheme $S_{b}$ corresponding to the symbol $b(z)$ is convergent.
\end{thm}
\begin{thm}\label{thmcontinuity--1}
The scheme $S_{b}$ corresponding to the symbol $b(z)$ is convergent iff its difference scheme $S_{c}$ corresponding to the symbol $c(z)$ is contractive, where $b(z)=(1+z)c(z)$. The scheme $S_{c}$ is contractive if
\begin{eqnarray*}
&&||c^{l}||_{\infty}=\mbox{max}\left\{\sum\limits_{i}|c^{l}_{j-2^{l_{i}}}|:0\leq j <2^{l}\right\}<1, \,\ l \in \mathbb{N},
\end{eqnarray*}
where $c^{l}_{i}$ are the coefficients of the scheme $S^{l}_{c}$ with symbol
\begin{eqnarray*}
&&c^{l}(z)=c(z)c(z^2)\ldots c(z^{2^{l-1}}).
\end{eqnarray*}
\end{thm}
In the following theorem, we prove that the proposed subdivision schemes satisfy the condition which is necessary for the convergence of a subdivision schemes.
\begin{thm}\label{bell-shaped-theorem}
The family of subdivision schemes with the refinement rules define in (\ref{proposed}) and (\ref{proposed1}) satisfies the basic sum rule.
\end{thm}
\begin{proof}
The family of subdivision schemes (\ref{proposed}) satisfies basic sum rule if $\sum\limits_{j=0}^{2n}a_{2j+1,2n+2}=\sum\limits_{j=0}^{2n+1}a_{2j,2n+2}=1$.\\
Since $\sum\limits_{j=0}^{2n}\left(\begin{array}{c}4n+2\\2j+1\end{array}\right)=\sum\limits_{j=0}^{2n+1}\frac{(-1)^{j-n-1}(n+1)}{2j-2n-1}\left(\begin{array}{c}2n+1\\n\end{array}\right)
\left(\begin{array}{c}2n+1\\j\end{array}\right)=\sum\limits_{j=0}^{2n+1}\left(\begin{array}{c}4n+2\\2j\end{array}\right)=2^{4n+1}$. Hence
\begin{eqnarray*}
\sum\limits_{j=0}^{2n}a_{2j+1,2n+2}&=&-\frac{\alpha}{2^{4n+1}}\sum\limits_{j=0}^{2n}\left(\begin{array}{c}4n+2\\2j+1\end{array}\right)+(1+\alpha)=-\frac{\alpha}{2^{4n+1}}2^{4n+1}+(1+\alpha)=1,
\end{eqnarray*}
and
\begin{eqnarray*}
\sum\limits_{j=0}^{2n+1}a_{2j,2n+2}&=&\frac{1}{2^{4n+1}}\sum\limits_{j=0}^{2n+1}\frac{(-1)^{j-n-1}(n+1)}{2j-2n-1}\left(\begin{array}{c}2n+1\\n\end{array}\right)
\left(\begin{array}{c}2n+1\\j\end{array}\right)(1+\alpha)\\&&
-\frac{\alpha}{2^{4n+1}}\sum\limits_{j=0}^{2n+1}\left(\begin{array}{c}4n+2\\2j\end{array}\right)=\frac{1}{2^{4n+1}}2^{4n+1}(1+\alpha)
-\frac{\alpha}{2^{4n+1}}2^{4n+1}=1.
\end{eqnarray*}
Therefore, the family of subdivision schemes with the refinement rules define in (\ref{proposed}) satisfy the basic sum rule which is also the necessary condition for the convergence of a subdivision scheme.
\end{proof}

\begin{algorithm}[ht!] 
   \caption[]{\label{continuity-algorithm} \emph{Continuity of the schemes by Laurent polynomial}}
    \begin{algorithmic}[1]
    \State \textbf{input:}  The symbol of the proposed schemes, i.e. $a_{2n+2}(z)=a_{0,2n+2}z^{0}+a_{1,2n+2}z^{1}+a_{2,2n+2}z^{2}+a_{3,2n+2}z^{3}+\ldots+a_{4n+2,2n+2}z^{4n+2}$, $n \in\mathbb{Z}^{+}$, $L\in\mathbb{Z}^{+}$, $\alpha=-1$(optional)
    \Comment $a_{i,2n+2}:i=0,1,\ldots,4n+2$ are the coefficients (depends on $\alpha$) of $z^{i}:i=0,1,\ldots,4n+2$ respectively in $a_{2n+2}(z)$
             \If {$\alpha=-1$}
                \State $a_{2n+2}(z)=\frac{(1+z)^{4n+2}}{2^{4n+1}}$
                \For {$j$ $=$ $0$ to $4n+1$}
                    \State calculate: $c_{2n+2}(z)= \frac{2^{j}}{(1+z)^{j+1}}a_{2n+2}(z)$
                   \For {$l$ $=$ $0$ to $2^{L}-1$}
                      \State calculate: $M_{l}=\sum\limits_{m}\left|c_{2^{L}m+l,2n+2}\right|$, $m \in \mathbb{Z}$
                   \EndFor
                   \State calculate: $M$ $=$ max$\{M_{l}:l=0,1,\ldots,2^{L}-1\}$
                   \If {$M<1$}
                       \State \Return scheme corresponding to $a_{2n+2}(z)$ is $C^{j}$-continuous for $\alpha=-1$
                   \Else {}
                       \State \Goto{C}
                   \EndIf
                \EndFor
            \Else {}
                \State $a_{2n+2}(z)=\frac{(1+z)^{2n+2}}{2^{2n+1}}(A_{0,2n+2}z^{0}+A_{1,2n+2}z^{1}+A_{2,2n+2}z^{2}+A_{3,2n+2}z^{3}+\ldots+A_{2n,2n+2}z^{2n})$
                \Comment $A_{i,2n+2}:i=0,1,\ldots,2n$ are the coefficients (depends on $\alpha$) of $z^{i}:i=0,1,\ldots,2n$ respectively in $a_{2n+2}(z)$
                \For {$j$ $=$ $0$ to $2n+1$}
                    \State calculate: $c_{2n+2}(z)= \frac{2^{j}}{(1+z)^{j+1}}a_{2n+2}(z)$
                   \For {$l$ $=$ $0$ to $2^{L}-1$}
                      \State calculate: an interval, say $(\alpha^{j}_{1,l},\alpha^{j}_{2,l})$, for $\alpha$ by solving $\sum\limits_{m}\left|c_{2^{L}m+l,2n+2}\right|<1$, $m \in \mathbb{Z}$
                   \EndFor
                   \State calculate: intersection, say $(\alpha^{j}_{1},\alpha^{j}_{2})$, of the $2^{L}$ intervals $(\alpha^{j}_{1,l},\alpha^{j}_{2,l}):l=0,1,\ldots,2^{L}-1$
                   \State \Return scheme is $C^{j}$-continuous for  $\alpha \in (\alpha^{j}_{1},\alpha^{j}_{2})$
                \EndFor
           \EndIf \label{C}
    \State \textbf{output:} Level of continuity of the proposed schemes for an interval of $\alpha$ or for $\alpha=-1$
\end{algorithmic}
\end{algorithm}

\begin{landscape}
\begin{table}[ht] 
 \caption[]{\label{continuity-L-1}\emph{Implementation of Algorithm \ref{continuity-algorithm} for $L=1$ $\&$ $n=1,2,3$: specific range/value of $\alpha$ for which the proposed schemes have certain level of continuity.}}
\centering \setlength{\tabcolsep}{0.5pt}
\begin{center}
\begin{tabular}{||c||c|c|c||}
  \hline \hline
  $n$ $\rightarrow$  & 1 & 2 & 3 \\
  \hline \hline
  Schemes $\rightarrow$  & $4$-point scheme & $6$-point scheme & $8$-point scheme \\
  Level of continuity & (relaxed) & (relaxed) & (relaxed) \\
  $\downarrow$ &&&\\
  \hline \hline
  & & & \\
   $C^{0}$-continuous & $ -4<\alpha<1.333333333 $ & $ -2.888888889<\alpha<0.8210526316 $ & $ -2.550443906<\alpha<0.5234765235 $\\
  \hline
    & & & \\
   $C^{1}$-continuous & $ -2.666666667\alpha < 0.0000000000 $ & $ -2.133333333< \alpha < 0.0000000000 $ & $ -2.031746032<\alpha<0.0000000000 $\\
  \hline
    & & & \\
   $C^{2}$-continuous & $ -2.666666667<\alpha<0.0000000000 $ & $ -2.133333333<\alpha<0.0000000000 $ & $ -1.997973658<\alpha<0.0000000000 $\\
  \hline
    & & & \\
   $C^{3}$-continuous & $ -1.333333333 < \alpha < -.6666666667 $ & $ -1.600000000<\alpha<-0.5333333333 $ & $ -1.523809524<\alpha<-0.5079365079 $\\
  \hline
    & & & \\
   $C^{4}$-continuous & $\alpha=-1$ & $ -1.542857143<\alpha<-0.6285714286 $ & $ -1.500952381<\alpha<-0.5257142857 $\\
  \hline
    & & & \\
   $C^{5}$-continuous & & $ -1.100000000<\alpha<-0.7000000000 $ & $ -1.245421245 < \alpha < -0.7765567766 $\\
  \hline
    & & & \\
   $C^{6}$-continuous & & $\alpha=-1$ & $ -1.182266010 < \alpha < -0.8669950739$\\
  \hline
    & & & \\
   $C^{7}$-continuous & & $\alpha=-1$ & $ -1.028571429 < \alpha < -0.9714285714 $\\
  \hline
    & & & \\
   $C^{8}$-continuous & & $\alpha=-1$ & $\alpha=-1$\\
  \hline
    & & & \\
   $C^{9}$-continuous & & & $\alpha=-1$\\
  \hline
    & & & \\
   $C^{10}$-continuous & & & $\alpha=-1$\\
  \hline
    & & & \\
   $C^{11}$-continuous & & & $\alpha=-1$\\
  \hline
    & & & \\
   $C^{12}$-continuous & & & $\alpha=-1$\\
  \hline \hline
\end{tabular}
\end{center}
\end{table}
\end{landscape}

\begin{landscape}
\begin{table}[ht] 
 \caption[]{\label{continuity-L-2}\emph{Implementation of Algorithm \ref{continuity-algorithm} for $L=2$ $\&$ $n=1,2,3$: specific range/value of $\alpha$ for which the proposed schemes have certain level of continuity.}}
\centering \setlength{\tabcolsep}{0.5pt}
\begin{center}
\begin{tabular}{||c||c|c|c||}
  \hline \hline
  $n$ $\rightarrow$ & 1 & 2 & 3 \\
  \hline \hline
  Schemes $\rightarrow$  & $4$-point scheme & $6$-point scheme & $8$-point scheme \\
  Level of continuity & (relaxed) & (relaxed) & (relaxed) \\
  $\downarrow$ &&& \\
  \hline \hline
  & & & \\
   $C^{0}$-continuous & $ -5.271476716<\alpha<1.568233303$ & $ -3.988172738<\alpha<1.01039496 $ & $ -3.383263797<\alpha<0.857832085$\\
  \hline
    & & & \\
   $C^{1}$-continuous & $ -3.581520882<\alpha<0.248187548 $ & $ -3.049774258 < \alpha < 0.304546042 $ & $ -2.799119823 < \alpha < 0.326701153$\\
  \hline
    & & & \\
   $C^{2}$-continuous & $ -2.666666667<\alpha<0.0000000000 $ & $ -2.624109757<\alpha<0.261216305 $ & $ -2.733101772<\alpha<0.172894526 $\\
  \hline
    & & & \\
   $C^{3}$-continuous & $ -1.745355992<\alpha<-0.6666666667 $ & $ -1.920669152<\alpha<-0.4485616181 $ & $ -1.998133604<\alpha <-0.3186270319$\\
  \hline
    & & & \\
   $C^{4}$-continuous & $\alpha=-1$ & $ -1.549222613<\alpha<-0.6138320373 $ & $ -1.690568592<\alpha<-0.4286478882$\\
  \hline
    & & & \\
   $C^{5}$-continuous & & $ -1.226979197<\alpha<-0.6765487168 $ & $ -1.357201949<\alpha<-0.7502504449$\\
  \hline
    & & & \\
   $C^{6}$-continuous & & $\alpha=-1$ & $ -1.197285679<\alpha<-0.850088102$\\
  \hline
    & & & \\
   $C^{7}$-continuous & & $\alpha=-1$ & $ -1.074359658<\alpha<-0.9559157159$\\
  \hline
    & & & \\
   $C^{8}$-continuous & & $\alpha=-1$ & $\alpha=-1$\\
  \hline
    & & & \\
   $C^{9}$-continuous & & & $\alpha=-1$\\
  \hline
    & & & \\
   $C^{10}$-continuous & & & $\alpha=-1$\\
  \hline
    & & & \\
   $C^{11}$-continuous & & & $\alpha=-1$\\
  \hline
    & & & \\
   $C^{12}$-continuous & & & $\alpha=-1$\\
  \hline \hline
\end{tabular}
\end{center}
\end{table}
\end{landscape}
We use Theorems \ref{thmcontinuity}-\ref{thmcontinuity--1} to design Algorithm \ref{continuity-algorithm}. This algorithm is designed to calculate the level of parametric continuity of the limit curves generated by the proposed subdivision schemes. In Tables \ref{continuity-L-1} and \ref{continuity-L-2}, we collect the results obtained from Algorithm \ref{continuity-algorithm} for $L=1$ and $L=2$ respectively. We use the step count method to find out the time complexity of the iterative Algorithms \ref{continuity-algorithm}. Firstly we compute the if-case time complexity.
\begin{center}
Time complexity of Steps 2-15 of Algorithm \ref{continuity-algorithm} (if-case).
\begin{tabular}{cccc}
  \hline
  Steps & Cost & Frequency & Total Cost \\
  \hline
2 & $C_{1}$ & 1              & $C_{1}$ \\
3 & $C_{2}$ & 1              & $C_{2}$ \\
4 & $C_{3}$ & $4n+2$         & $C_{3}(4n+2)$\\
5 & $C_{4}$ & $4n+2$         & $C_{4} (4n+2)$\\
6 & $C_{5}$ & $2^{L} (4n+2)$ & $C_{5} (2^{L} (4n+2))$\\
7 & $C_{6}$ & $2^{L} (4n+2)$ & $C_{6} (2^{L} (4n+2))$\\
9 & $C_{7}$ & $4n+2$         & $C_{7} (4n+2)$\\
8 & $C_{8}$ & 1              & $C_{8}$   \\
10& $C_{9}$ & $4n+2$         & $C_{9} (4n+2)$\\
11& $C_{10}$ & $4n+2$        & $C_{10} (4n+2)$\\
12& $C_{11}$& $1$            & $C_{11}$\\
13& $C_{12}$& $1$            & $C_{12}$\\
14& $C_{13}$ & 1             & $C_{13}$ \\
15& $C_{14}$ & 1             & $C_{14}$ \\
  \hline
\end{tabular}
\end{center}
Hence the if-case time complexity of Algorithm \ref{continuity-algorithm} is
\begin{eqnarray*}
T_{I}(n,L)&=& C_{1}+C_{2}+2C_{3}+2C_{4}+2C_{7}+C_{8}+2C_{9}+2C_{10}+C_{13}+C_{14}+2^{L+1}C_{5}\\&&+2^{L+1}C_{6}+4nC_{3}+4nC_{4}+4nC_{7}+4nC_{9}+4n C_{10}+4n 2^{L}C_{5}+4n 2^{L} C_{6},\\
&=&\tilde{C}_{1}+2^{L+1}\tilde{C}_{2}+4n\tilde{C}_{3}+4n 2^{L} \tilde{C}_{4}\in \mathcal{O}(2^{L}n),
\end{eqnarray*}
where $\tilde{C}_{1}=C_{1}+C_{2}+2C_{3}+2C_{4}+2C_{7}+C_{8}+2C_{9}+2C_{10}+C_{13}+C_{14}$, $\tilde{C}_{2}=C_{5}+C_{6}$, $\tilde{C}_{3}=C_{3}+C_{4}+C_{7}+C_{9}+C_{10}$ and $\tilde{C}_{4}=C_{5}+C_{6}$.
Now we compute the else-case time complexity of Algorithm \ref{continuity-algorithm}.
\begin{center}
Time complexity of Steps 16-25 of Algorithm \ref{continuity-algorithm} (else-case).
\begin{tabular}{cccc}
  \hline
  Steps & Cost & Frequency & Total Cost \\
  \hline
16 & $C_{15}$ & 1              & $C_{15}$ \\
17 & $C_{16}$ & 1              & $C_{16}$ \\
18 & $C_{17}$ & $2n+2$         & $C_{17} (2n+2)$ \\
19 & $C_{18}$ & $2n+2$         & $C_{18} (2n+2)$ \\
20 & $C_{19}$ & $2^{L} (2n+2)$ & $C_{19}(2^{L} (2n+2))$\\
21 & $C_{20}$ & $2^{L} (2n+2)$ & $C_{20} (2^{L} (2n+2))$\\
22 & $C_{21}$ & 1              & $C_{21}$ \\
23 & $C_{22}$ & $2n+2$         & $C_{22} (2n+2)$\\
24 & $C_{23}$ & $2n+2$         & $C_{23} (2n+2)$\\
25 & $C_{24}$ & 1              & $C_{24}$ \\
  \hline
\end{tabular}
\end{center}
The else-case time complexity of Algorithm \ref{continuity-algorithm} is
\begin{eqnarray*}
T_{E}(n,L) &=& C_{15}+C_{16}+2C_{17}+2C_{18}+C_{21}+2C_{22}+2C_{23}+C_{24}+2nC_{17}+2nC_{18}\\&&+2nC_{22}+2nC_{23}+2^{L+1}C_{19}+2^{L+1}C_{20}+2n 2^{L}C_{19}+2n 2^{L} C_{20},\\
&=&\tilde{C}_{5}+2^{L+1}\tilde{C}_{6}+2n\tilde{C}_{7}+2n 2^{L} \tilde{C}_{8}\in \mathcal{O}(2^{L}n),
\end{eqnarray*}
where $\tilde{C}_{5}=C_{15}+C_{16}+2C_{17}+2C_{18}+C_{21}+2C_{22}+2C_{23}+C_{24}$, $\tilde{C}_{6}=C_{17}+C_{18}+C_{22}+C_{23}$, $\tilde{C}_{7}=C_{19}+C_{20}$ and $\tilde{C}_{8}=C_{19}+C_{20}$. Hence the time complexity of Algorithm \ref{continuity-algorithm} is $ \mathcal{O}(2^{L}n)$.

\subsection{Degrees of polynomial generation and polynomial reproduction}
Generation and reproduction degrees are used to examine the behaviors of a subdivision scheme when the original data points lie on the curve of a polynomial. Suppose that the original data points are taken from a polynomial of degree $d$. If the control points of the limit curve lie on curve of the polynomial having same degree (i.e. $d$) then we say that the subdivision scheme generates polynomials of degree $d$. If the control points of the limit curve lie on curve of the same polynomial then we say that the subdivision scheme reproduces polynomials of degree $d$. Mathematically, let $\Pi_{d}$ denote the space of polynomials of degree $d$ and $g$, $h$ $\in$ $\Pi_{d}$, an approximation operator $\mathbf{O}$ generates polynomials of degree $d$ if $\mathbf{O}g=h$ $\forall$ $g, h \in \Pi_{d}$, whereas $\mathbf{O}$ reproduces polynomials of degree $d$ if $\mathbf{O}g=g$ $\forall$ $g \in \Pi_{d}$. Furthermore, the generation degree of a subdivision scheme is the maximum degree of polynomials that can theoretically be generated by the scheme, provided that the initial data is taken correctly. Evidently, it is not less than the reproduction degree. The readers may consult \cite{Charina,Conti} to find out more detail about polynomial generation and polynomial reproduction of a subdivision scheme. We design Algorithm \ref{generation-algorithm} and Algorithm \ref{reproduction-algorithm}, based on the results given in \cite{Conti}, to check the degrees of polynomials generation and polynomials reproduction of proposed schemes respectively. Outputs obtained by implementing these algorithms are summarized in Table \ref{generation-table} and Table \ref{reproduction-table} respectively.
\begin{algorithm}[p] 
   \caption[]{\label{generation-algorithm} \emph{generation degree of the schemes}}
    \begin{algorithmic}[1]
    \State \textbf{input:}  The symbol of the proposed schemes, i.e. $a_{2n+2}(z)=a_{0,2n+2}z^{0}+a_{1,2n+2}z^{1}+a_{2,2n+2}z^{2}+a_{3,2n+2}z^{3}+\ldots+a_{4n+2,2n+2}z^{4n+2}$, $n \in\mathbb{Z}^{+}$
    \Comment $a_{i,2n+2}:i=0,1,\ldots,4n+2$ are the coefficients (depends on $\alpha$) of $z^{i}:i=0,1,\ldots,4n+2$ respectively in $a_{2n+2}(z)$
    \If {$\alpha=-1$}
                \State $a_{2n+2}(z)=\frac{(1+z)^{4n+2}}{2^{4n+1}}$
       \For {$j$ $=$ $0$ to $4n+1$}
          \State calculate: $A=\left.a^{(j)}(z)\right|_{z=-1}$ \Comment $\left.a^{(j)}(z)\right|_{z=-1}$ denotes the $j$-th derivative of $a(z)$ evaluated at $z=-1$
          \If {$A=0$}
             \State \Return scheme generates polynomials up to degree $j$ for $\alpha=-1$
          \Else {}
             \State \Goto{A}
          \EndIf
       \EndFor
    \Else {}
       \State $a_{2n+2}(z)=\frac{(1+z)^{2n+2}}{2^{2n+1}}(A_{0,2n+2}z^{0}+A_{1,2n+2}z^{1}+A_{2,2n+2}z^{2}+A_{3,2n+2}z^{3}+\ldots+A_{2n,2n+2}z^{2n})$
                \Comment $A_{i,2n+2}:i=0,1,\ldots,2n$ are the coefficients (depends on $\alpha$) of $z^{i}:i=0,1,\ldots,2n$ respectively in $a_{2n+2}(z)$
       \For {$j$ $=$ $0$ to $2n+1$}
          \State calculate: $A=\left.a^{(j)}(z)\right|_{z=-1}$
          \If {$A=0$}
             \State \Return scheme generates polynomials up to degree $j$ for all $\alpha$
          \Else {}
             \State \Goto{A}
          \EndIf
       \EndFor
    \EndIf \label{A}
    \State \textbf{output:} degree of generation of the schemes corresponding to $a_{2n+2}(z)$ for all $\alpha$ or for $\alpha=-1$
\end{algorithmic}
\end{algorithm}

Now we find the time complexity of Algorithms \ref{generation-algorithm}. Firstly we compute the if-case (Steps 2-11) time complexity and then we find the else-case time complexity (Steps 12-22). Hence the worst-case time complexity will be the time complexity of Algorithms \ref{generation-algorithm}.
\begin{center}
Time complexity of Steps 2-11 of Algorithm \ref{generation-algorithm} (if-case).
\begin{tabular}{cccc}
  \hline
  Steps & Cost & Frequency & Total Cost \\
  \hline
2 & $C_{1}$ & 1              & $C_{1}$ \\
3 & $C_{2}$ & 1              & $C_{2}$ \\
4 & $C_{3}$ & $4n+2$         & $C_{3}(4n+2)$\\
5 & $C_{4}$ & $4n+2$         & $C_{4} (4n+2)$\\
6 & $C_{5}$ & $4n+2$         & $C_{5}(4n+2)$\\
7 & $C_{6}$ & $4n+2$         & $C_{6} (4n+2)$\\
8 & $C_{7}$ & $1$            & $C_{7}$\\
9 & $C_{8}$ & $1$            & $C_{8}$\\
10& $C_{9}$ & $1$            & $C_{9}$\\
11& $C_{10}$ & $1$           & $C_{10}$\\
  \hline
\end{tabular}
\end{center}
 To calculate the time complexity of Steps 2-11, we either consider the total cost of Steps 6-7 or of Steps 8-9. So we consider the worst-case which is Steps 6-7. Hence the if-case time complexity of Algorithm \ref{generation-algorithm} is
\begin{eqnarray*}
T_{I}(n)&=& C_{1}+C_{2}+2C_{3}+2C_{4}+2C_{5}+2C_{6}+C_{9}+C_{10}+4nC_{3}+4nC_{4}+4nC_{5}+4nC_{6},\\
&=&\tilde{C}_{1}+4n\tilde{C}_{2} \in \mathcal{O}(n),
\end{eqnarray*}
where $\tilde{C}_{1}=C_{1}+C_{2}+2C_{3}+2C_{4}+2C_{5}+2C_{6}+C_{9}+C_{10}$ and $\tilde{C}_{2}=C_{3}+C_{4}+C_{5}+C_{6}$.
Now we compute the else-case time complexity (Steps 12-21).
\begin{center}
Time complexity of Steps 12-21 of Algorithm \ref{generation-algorithm} (else-case).
\begin{tabular}{cccc}
  \hline
  Steps & Cost & Frequency & Total Cost \\
  \hline
12 & $C_{11}$ & 1              & $C_{11}$ \\
13 & $C_{12}$ & 1              & $C_{12}$ \\
14 & $C_{13}$ & $2n+2$         & $C_{13} (2n+2)$ \\
15 & $C_{14}$ & $2n+2$         & $C_{14} (2n+2)$ \\
16 & $C_{15}$ & $2n+2$         & $C_{15} (2n+2))$\\
17 & $C_{16}$ & $2n+2$         & $C_{16} (2n+2))$\\
18 & $C_{17}$ & 1              & $C_{17}$\\
19 & $C_{18}$ & 1              & $C_{18}$\\
20 & $C_{17}$ & 1              & $C_{19}$\\
21 & $C_{18}$ & 1              & $C_{20}$\\
  \hline
\end{tabular}
\end{center}
To calculate time complexity of Steps 12-21, we ignore total cost of Steps 18-19. Therefore, the else-case time complexity of Algorithm \ref{generation-algorithm} is
\begin{eqnarray*}
T_{E}(n)&=& C_{11}+C_{12}+2C_{13}+2C_{14}+2C_{15}+2C_{16}+C_{19}+C_{20}\\&&+2nC_{13}+2nC_{14}+2nC_{15}+2nC_{16}=\tilde{C}_{3}+2n\tilde{C}_{4} \in \mathcal{O}(n),
\end{eqnarray*}
where $\tilde{C}_{3}=C_{11}+C_{12}+2C_{13}+2C_{14}+2C_{15}+2C_{16}+C_{17}+C_{18}+C_{19}+C_{20}$ and $\tilde{C}_{4}=C_{13}+C_{14}+C_{15}+C_{16}$.
Hence time complexity of Algorithm \ref{generation-algorithm} is $\mathcal{O}(n)$. In the same manner, we can prove that the time complexity of Algorithm \ref{reproduction-algorithm} is also $\mathcal{O}(n)$.

\begin{table}[p] 
 \caption[]{\label{generation-table}\emph{Implementation of Algorithm \ref{generation-algorithm} for $n=1,2,3,4$.}}
\centering \setlength{\tabcolsep}{0.5pt}
\begin{center}
\begin{tabular}{||c||c|c|c|c||}
  \hline \hline
  $n$ $\rightarrow$  & 1 & 2 & 3 & 4\\
  \hline \hline
  Schemes $\rightarrow$  & $4$-point scheme & $6$-point scheme & $8$-point scheme & $10$-point scheme\\
  & (relaxed) & (relaxed) & (relaxed) & (relaxed) \\
  \hline \hline
  & & & & \\
  generation degree $\forall$ $\alpha$ & 3 & 5 & 7 & 9 \\
  \hline
    & & & & \\
  generation degree $\alpha=-1$ & 5 & 9 & 13 & 17 \\
    \hline \hline
\end{tabular}
\end{center}
\end{table}

\begin{algorithm}[p] 
   \caption[]{\label{reproduction-algorithm} \emph{Reproduction degree of the schemes}}
    \begin{algorithmic}[1]
    \State \textbf{input:}  The symbol of the proposed schemes, i.e. $a_{2n+2}(z)=a_{0,2n+2}z^{0}+a_{1,2n+2}z^{1}+a_{2,2n+2}z^{2}+a_{3,2n+2}z^{3}+\ldots+a_{4n+2,2n+2}z^{4n+2}=\frac{(1+z)^{2n+2}}{2^{2n+1}}(A_{0,2n+2}z^{0}+A_{1,2n+2}z^{1}+A_{2,2n+2}z^{2}+A_{3,2n+2}z^{3}+\ldots+A_{2n,2n+2}z^{2n})$, $n \in\mathbb{Z}^{+}$
    \Comment $a_{i,2n+2}:i=0,1,\ldots,4n+2$ are the coefficients (depends on $\alpha$) of $z^{i}:i=0,1,\ldots,4n+2$ respectively in $a_{2n+2}(z)$
    \State calculate: $\tau=\frac{a^{'}(z)}{2}$
    \For {$j$ $=$ $0$ to $4n+3$}
       \State calculate: $A_{j}=\left.a^{(j)}_{2n+2}\right|_{z=1}$, $B_{j}=\prod\limits_{p=0}^{j-1}(\tau-p)$ and $C_{j}=\left.a^{(j)}_{2n+2}(z)\right|_{z=-1}$
       \If {$A_{j}=B_{j}$} and {$C_{j}=0$}
          \State scheme reproduce polynomials of degree $j$ $\forall$ $\alpha$
       \Else {}
          \State calculate: $A_{j}$, $B_{j}$ and $C_{j}$ for $\alpha=0$
          \If {$A_{j}=B_{j}$ and $C_{j}=0$ for $\alpha=0$}
             \State scheme reproduce polynomials of degree $j$ for $\alpha=0$
          \Else
             \State \Goto{B}
          \EndIf
       \EndIf
    \EndFor \label{B}
    \State \textbf{output:} degree of reproduction of the schemes corresponding to $a_{2n+2}(z)$ for all $\alpha$ or for $\alpha=0$
\end{algorithmic}
\end{algorithm}

\begin{table}[p] 
 \caption[]{\label{reproduction-table}\emph{Implementation of Algorithm \ref{reproduction-algorithm} for $n=1,2,3,4$.}}
\centering \setlength{\tabcolsep}{0.5pt}
\begin{center}
\begin{tabular}{||c||c|c|c|c||}
  \hline \hline
  $n$ $\rightarrow$  & 1 & 2 & 3 & 4 \\
  \hline \hline
  Schemes $\rightarrow$  & $4$-point scheme & $6$-point scheme & $8$-point scheme & $10$-point scheme\\
  & (relaxed) & (relaxed) & (relaxed) & (relaxed) \\
  \hline \hline
  & & & & \\
  reproduction degree $\forall$ $\alpha$ & 1 & 1 & 1 & 1\\
  \hline
    & & & & \\
  reproduction degree $\alpha=0$ & 3 & 5 & 7 & 9\\
    \hline \hline
\end{tabular}
\end{center}
\end{table}
\subsection{Gibbs oscillations}
According to \cite{Hameed1}, linear subdivision schemes exhibit overshoots or undershoots Gibbs oscillations near to the points of discontinuity of a discontinuous function. The schemes that show undershoots generate curves free from oscillations while the schemes with overshoots give unpleasant oscillations in the limit curves.
Let $0 \leqslant \xi \leqslant h$, a function $p$ is said to be discontinuous at a point $\xi$, if
\begin{eqnarray*}
&&\forall x\leqslant  \xi, \,\ p(x)=p_{-(x)} \,\ \mbox{with} \,\ p_{-} \in C^{\infty}(]-\infty, \xi]),\\
&&\forall x \geqslant \xi, \,\ p(x)=p_{+(x)} \,\ \mbox{with} \,\ p_{+} \in C^{\infty}(]\xi, \infty[),
\end{eqnarray*}
with $p_{-}(\xi)>p_{+}(\xi)$.

A scheme has the overshoots near to the point $\xi$ of discontinuity if
\begin{eqnarray*}
&&\forall x \leqslant \xi, \,\ p_{-(x)}<S_{A}^{\infty}(p^{0})(x),\\
&&\forall x \geqslant \xi, \,\  p_{+(x)}>S_{A}^{\infty}(p^{0})(x).
\end{eqnarray*}
A scheme has the undershoots near to the point $\xi$ of discontinuity if
\begin{eqnarray*}
&&\forall x \leqslant \xi, \,\ p_{-(x)}>S_{A}^{\infty}(p^{0})(x),\\
&&\forall x \geqslant \xi, \,\  p_{+(x)}<S_{A}^{\infty}(p^{0})(x).
\end{eqnarray*}
Algorithm \ref{gibbs-algorithm} is designed to calculate the range of tension parameter $\alpha$ for which the proposed schemes exhibit undershoot Gibbs oscillations. The results obtained from this algorithm are summarized in Table \ref{gibbs-table}. We calculate the time complexity of Algorithm \ref{gibbs-algorithm} which depends on the input value $k$ by using step count method. The time complexity of this algorithm does not depend on the input value $n$. Hence the time complexity of this algorithm is $\mathcal{O}(k)$.
\begin{algorithm}[p] 
   \caption[]{\label{gibbs-algorithm} \emph{Undershoot Gibbs phenomenon of the proposed schemes.}}
    \begin{algorithmic}[1]
    \State \textbf{input:} $n \in\mathbb{Z}^{+}$, $k$, $P_{i}^{0}:i \in \mathbb{Z}$
    \Comment  $P_{i}^{0}$ are taken from a discontinuous function, i.e. \begin{equation*}
 P_{i}^{0}= \begin{cases}
\,\ 10  & \text{for $i \in ]-\infty,-1]$}, \\
 -10  & \text{for  $i \in ]0,+\infty[$}.
\end{cases}
\end{equation*}
    \For {$m$ $=$ $0$ to $k$}
       \State calculate:   $P_{2i,n}^{m+1} =\frac{-\alpha}{2^{4n+1}}\sum\limits_{j=0,j\neq n}^{2n}\left(\begin{array}{c}4n+2\\2j+1\end{array}\right)P_{i+j-n,2n+2}^{m}+\frac{1}{2^{4n+1}}\left[2^{4n+1}(1+\alpha)-\alpha \times \right.$ $\left. \left(\begin{array}{c}4n+2\\2n+1\end{array}\right)\right]P_{i,2n+2}^{m},$
  \State calculate:   $P_{2i+1,2n+2}^{m+1} =\frac{1}{2^{4n+1}}\sum\limits_{j=0}^{2n+1}\left[\frac{(-1)^{j-n-1}(n+1)}{2j-2n-1}\left(\begin{array}{c}2n+1\\n\end{array}\right)\left(\begin{array}{c}2n+1\\j\end{array}\right)+\alpha \times \right.$ $\left.\left\{\frac{(-1)^{j-n-1}(n+1)}{2j-2n-1}
  \left(\begin{array}{c}2n+1\\n\end{array}\right)\left(\begin{array}{c}2n+1\\j\end{array}\right) -\left(\begin{array}{c}4n+2\\2j\end{array}\right)\right\}\right]P_{i+j-n,2n+2}^{m}.$
    \EndFor
     \State calculate: intervals of $\alpha$ for which $P_{-1,2n+2}^{0}>S^{k+1}_{a_{2n+2}}(P_{-1,2n+2}^{0})$\label{D}
     \State calculate: intervals of $\alpha$ for which $P_{0,2n+2}^{0}<S^{k+1}_{a_{2n+2}}(P_{0,2n+2}^{0})$\label{E}
     \State calculate: intersection of intervals of $\alpha$ from Steps \ref{D}-\ref{E}
    \State \textbf{output:} range of $\alpha$ for which the given scheme exibits undershoot Gibbs phenomenon at specific level of subdivision $k$.
\end{algorithmic}
\end{algorithm}

\begin{table}[p] 
 \caption[]{\label{gibbs-table}\emph{Implementation of Algorithm \ref{gibbs-algorithm} for $k=0,1,2,3$ and $n=1,2,3$.}}
\centering \setlength{\tabcolsep}{0.5pt}
\begin{center}
\begin{tabular}{||c||c|c|c||}
  \hline \hline
$n$ $\rightarrow$  & 1 & 2 & 3 \\
  \hline \hline
Schemes $\rightarrow$  & $4$-point scheme & $6$-point scheme & $8$-point scheme \\
   k & (relaxed) & (relaxed) & (relaxed)  \\
 $\downarrow$ &&& \\
  \hline \hline
  0  & $\alpha<0$ & $\alpha<0$ & $\alpha<0$  \\
  &  &  &  \\
  \hline
 1  & $-5.444444444<\alpha<0$ & $-4.013223140<\alpha<0$ & $-3.523131552<\alpha<0$  \\
  & & & \\
   \hline
 2  & $-9.494261050<\alpha<0$ & $-6.840555349<\alpha<0$ & $-5.821687339<\alpha<0$  \\
  & & & \\
    \hline
 3  & $-5.618133906<\alpha<0$ & $-5.583827202<\alpha<0$  &  $-7.639304050<\alpha<0$ \\
  & & & \\
    \hline \hline
\end{tabular}
\end{center}
\end{table}

\subsection{Shape preserving properties}
Here we discuss the shape preserving properties, monotonicity and convexity preservation, of the proposed family of subdivision schemes. We prove that the proposed subdivision schemes preserve monotonicity and convexity for a special interval of tension parameter. For this, firstly we calculate an interval of $\alpha$ for which the proposed subdivision schemes having the bell-shaped mask. Then by using results given in \cite{Hameed1}, we prove that the proposed schemes preserve monotonicity and convexity for a special interval of $\alpha$.

A $2N$-degree polynomial $\sum\limits_{j=0}^{2N}b_{j}x^{j}$ is said to be a polynomial with bell-shaped coefficients if it satisfied the following conditions

\begin{eqnarray}\label{}
\left\{\begin{array}{ccc}
&&\mbox{supp}(b_{j})=[0,2N], \\ \\
&&b_{j}>0 \,\ \mbox{for} \,\ j \in [0,2N],\\ \\
&&b_{j}=b_{2N-j} \,\ \mbox{for} \,\ j \in [0,2N],\\ \\
&&b_{j}<b_{j+1}  \,\ \mbox{for} \,\ j \in [0, N-1]
  \end{array}\right.
\end{eqnarray}

The mask symbol of the proposed family of schemes is defined in (\ref{proposed1}) which is a polynomial of degree-$(4n+2)$ satisfying the following two conditions.
\begin{eqnarray}\label{}
\left\{\begin{array}{ccc}
&&\mbox{supp}(a_{j,2n+2})=[0,4n+2], \\ \\
&&a_{j,2n+2}=a_{4n+2-j,2n+2} \,\ \mbox{for} \,\ j \in [0,4n+2].
  \end{array}\right.
\end{eqnarray}
The other two conditions depend on value of $\alpha$ which can be calculated by Algorithm \ref{bell-shaped}. The outputs obtained from Algorithm \ref{bell-shaped} are given in Table \ref{bell-mask-table}.
\begin{algorithm}[p] 
   \caption[]{\label{bell-shaped} \emph{Bell-shaped mask of the proposed schemes.}}
    \begin{algorithmic}[1]
    \State \textbf{input:}  The mask $a_{j,2n+2}:j=0,1,\ldots,4n+2$ of the family of proposed schemes whose mask symbol is defined in (\ref{symbol-proposed1}), $n \in\mathbb{Z}^{+}$
    \State set: $A_{j}=a_{j,2n+2}$
    \For {$j$ $=$ $0$ to $4n+2$}
       \State calculate: an interval, say $(\alpha_{1}^{j},\alpha_{2}^{j})$, for $\alpha$ by solving $A_{j}>0$
    \EndFor
    \State calculate: intersection, say $(\alpha_{1},\alpha_{2})$, of $4n+3$ intervals $(\alpha_{1}^{j},\alpha_{2}^{j}):j=0,1,\ldots,4n+2$
    \For {$j$ $=$ $0$ to $2n$}
       \State calculate: an interval, say $(\beta_{1}^{j},\beta_{2}^{j})$, for $\alpha$ by solving $A_{j}<A_{j+1}$
    \EndFor
    \State calculate: intersection, say $(\beta_{1},\beta_{2})$, of $2n+1$ intervals $(\beta_{1}^{j},\beta_{2}^{j}):j=0,1,\ldots,2n$
    \State calculate: intersection, say $(\gamma_{1},\gamma_{2})$, of $2$ intervals $(\alpha_{1},\alpha_{2})$ and $(\beta_{1},\beta_{2})$
    \State \Return $(\alpha_{1},\alpha_{2})$, $(\beta_{1},\beta_{2})$, $(\gamma_{1},\gamma_{2})$
    \State \textbf{output:} the proposed $(2n+2)$-point scheme has a bell-shaped mask for $\alpha \in (\gamma_{1},\gamma_{2})$
\end{algorithmic}
\end{algorithm}

\begin{table}[p] 
 \caption[]{\label{bell-mask-table}\emph{Implementation of Algorithm \ref{bell-shaped} for $n=1,2,3$.}}
\centering \setlength{\tabcolsep}{0.3pt}
\small
\begin{center}
\begin{tabular}{||c||c|c|c||}
\hline \hline
$n$ $\rightarrow$  & 1 & 2 & 3 \\
  \hline \hline
Schemes $\rightarrow$  & $4$-point scheme & $6$-point scheme & $8$-point scheme \\
    & (relaxed) & (relaxed) & (relaxed)  \\
\hline \hline
$(\alpha_{1},\alpha_{2})$ $\rightarrow$ & $(-2.666666667,-0.6666666667)$ & $(-1.200000000, -0.5263157895)$ & $(-1.721008403,-0.9523809524)$ \\
&&&\\
\hline
$(\beta_{1},\beta_{2})$ $\rightarrow$ & $(-1.555555556,0.6666666667)$ & $(-1.247058824,-0.5882352941)$ & $(-1.149842822,-0.6060606061)$\\
&&&\\
\hline
$(\gamma_{1},\gamma_{2})$ $\rightarrow$ & $(-1.555555556,-0.6666666667)$ & $(-1.200000000,-0.5882352941)$ & $(-1.149842822,-0.9523809524)$\\
&&&\\
\hline \hline
\end{tabular}
\end{center}
\end{table}

\begin{thm}
The proposed $4$-point, $6$-point and $8$-point subdivision schemes preserve monotonicity for $-1.555555556< \alpha <-0.6666666667$, $-1.200000000< \alpha <-0.5882352941$ and $-1.149842822 < \alpha < -0.9523809524$ respectively.
\end{thm}
\begin{proof}
By Theorem \ref{bell-shaped-theorem}, we see that the proposed subdivision schemes satisfy the basic sum rule. Also from Algorithm \ref{bell-shaped}, we get that the proposed $4$-point, $6$-point and $8$-point subdivision schemes have bell-shaped mask for $-1.555555556< \alpha <-0.6666666667$, $-1.200000000< \alpha <-0.5882352941$ and $-1.149842822 < \alpha < -0.9523809524$ respectively, so it is easy to prove the required result by using (\cite{Hameed1}, Theorem 4.3).
\end{proof}

\begin{thm}
The proposed $4$-point, $6$-point and $8$-point subdivision schemes preserve convexity for $-1.555555556< \alpha <-0.6666666667$, $-1.200000000< \alpha <-0.5882352941$ and $-1.149842822 < \alpha < -0.9523809524$ respectively.
\end{thm}
\begin{proof}
By Theorem \ref{bell-shaped-theorem}, we see that the proposed subdivision schemes satisfy the basic sum rule. Also from Algorithm \ref{bell-shaped}, we get that the proposed $4$-point, $6$-point and $8$-point subdivision schemes have bell-shaped mask for $-1.555555556< \alpha <-0.6666666667$, $-1.200000000< \alpha <-0.5882352941$ and $-1.149842822 < \alpha < -0.9523809524$ respectively. Further from (\ref{symbol-proposed1}), we can also see that the mask symbols of the proposed subdivision schemes have the factor $(1+z)^{2}$. Hence the required result is proved by (\cite{Hameed1}, Theorem 4.4).
\end{proof}

The proposed Algorithm \ref{bell-shaped} needs $\mathcal{O}(n)$ time to give output.
\section{Numerical examples}
In this section, we show the numerical performance, of the proposed schemes, based on the results which have proved in previous section. Figure \ref{BF-Schems} displays the support of the subdivision scheme when the initial data is taken from a basic limit function defined in (\ref{BF}). Figure \ref{scheme1} show the curves fitted by the proposed $(2n+2)$-point subdivision schemes. It is clear from this figure that the tension parameter involved in the proposed subdivision schemes provides relaxation in drawing and designing different types of curves by using the same initial points. Figure \ref{monotonic} and Figure \ref{convex} show the performance of the proposed subdivision schemes when the initial data is taken from monotonic and convex functions respectively. Figure \ref{flower} and Figure \ref{noise} show that the proposed subdivision schemes also carry the characteristics of keeping the shape of the initial polygon and of fitting the noisy data points respectively. Which is an important characteristic of the proposed family. In Literature, it is hard to find a single subdivision scheme which carry both characteristics simultaneously. The performance of the tensor product scheme of the proposed $4$-point scheme in surface fitting is shown in Figure \ref{surface}.  Sub-figures \ref{surface}(e), \ref{surface}(f), \ref{surface}(g), \ref{surface}(h), \ref{surface}(m), \ref{surface}(n), \ref{surface}(o) and \ref{surface}(p) are the mirror images in $ij$-plane of Sub-figures \ref{surface}(a), \ref{surface}(b), \ref{surface}(c), \ref{surface}(d), \ref{surface}(i), \ref{surface}(j), \ref{surface}(k) and \ref{surface}(l) respectively.

\begin{figure}[h!] 
\begin{center}
\begin{tabular}{ccc}
\epsfig{file=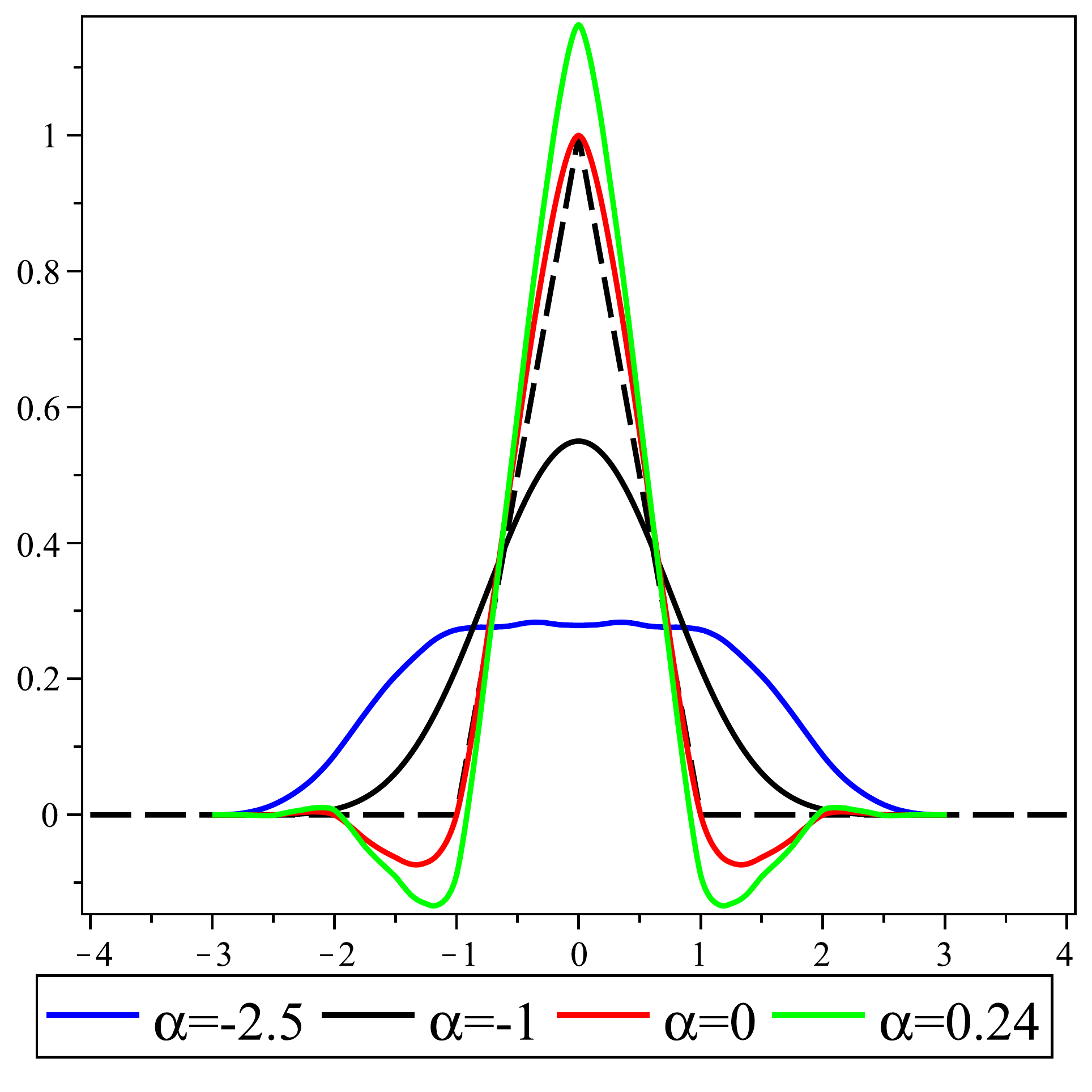, width=2.0 in} & \epsfig{file=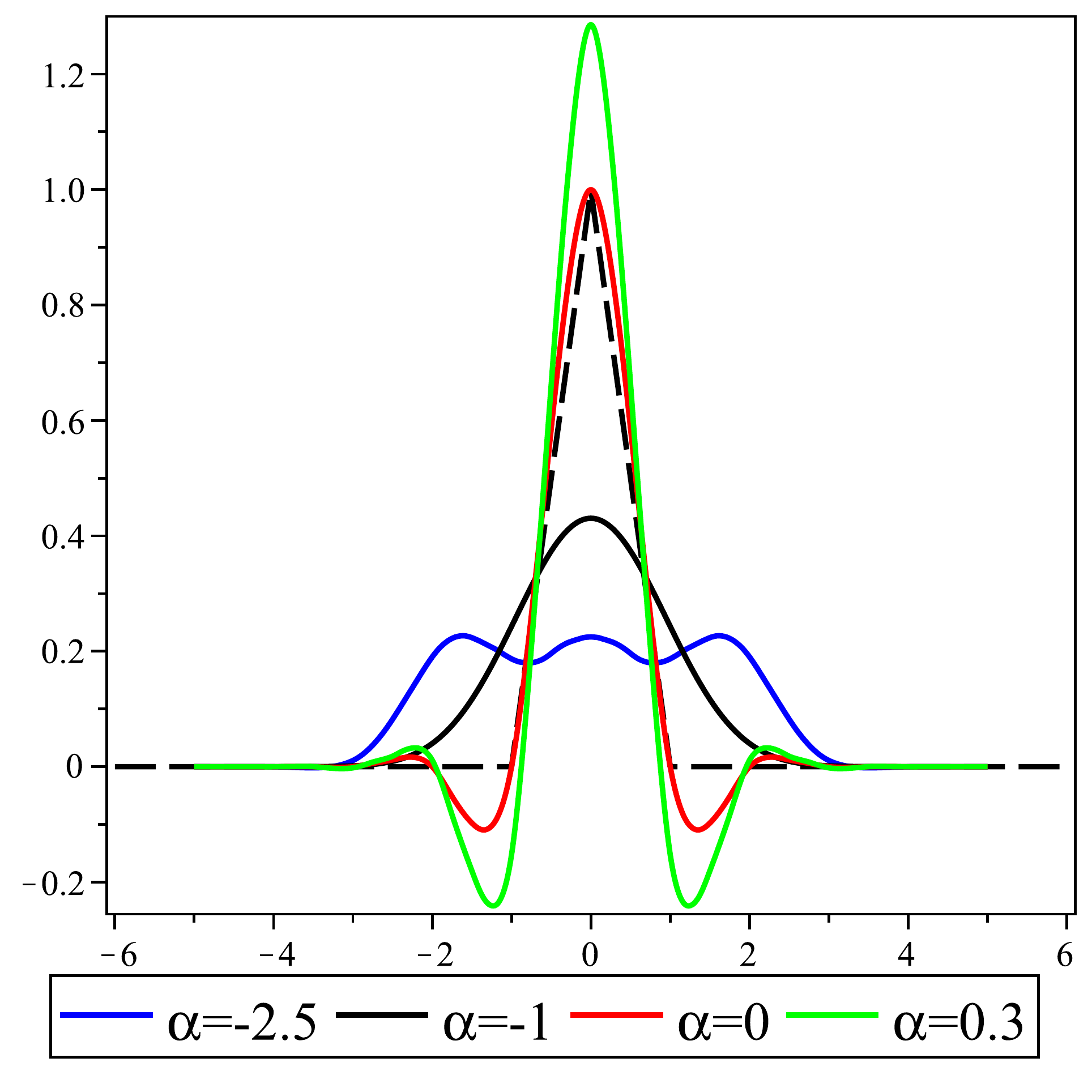, width=2.0 in} & \epsfig{file=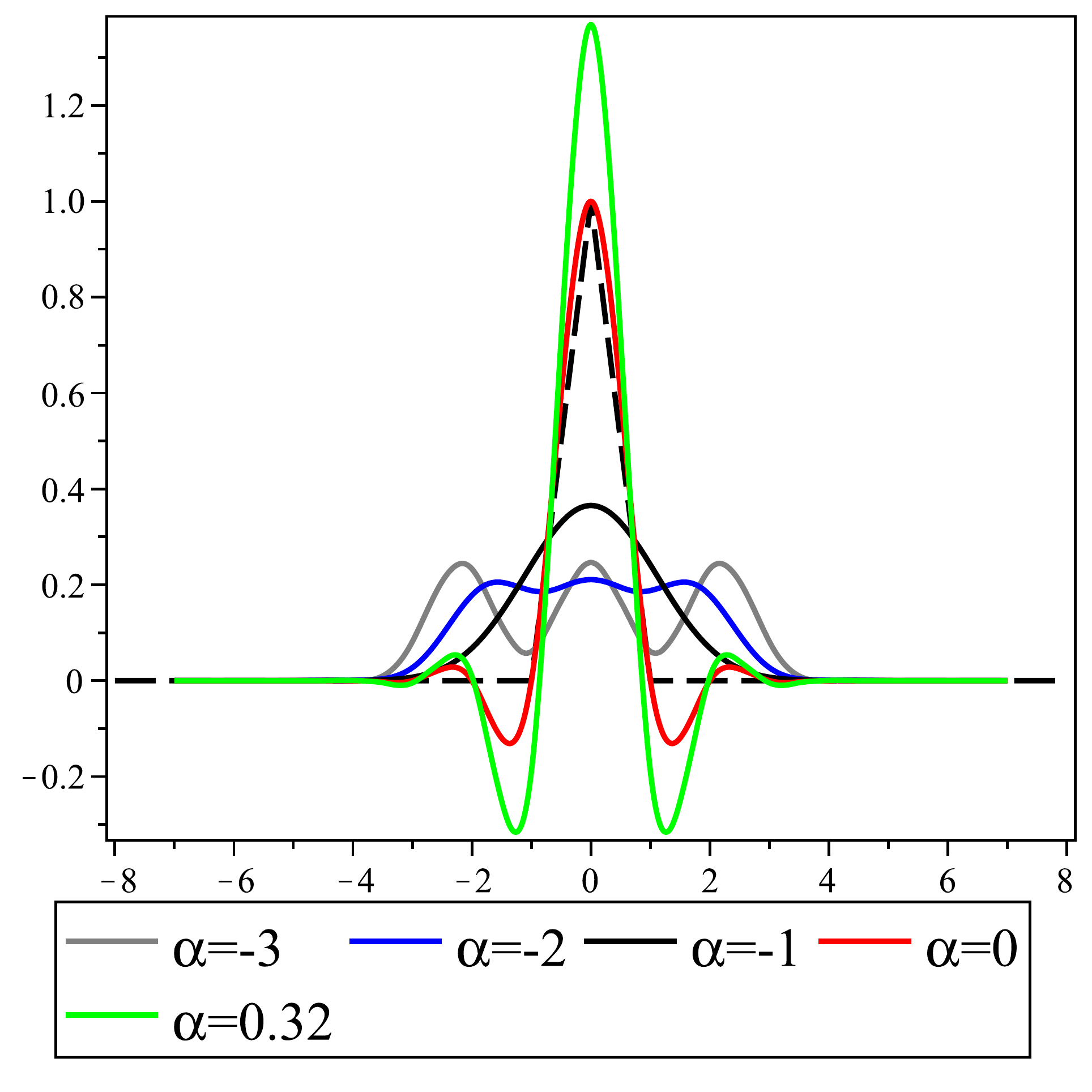, width=2.0 in} \\
(a) & (b) & (c) \\
$4$-point scheme & $6$-point scheme & $8$-point scheme
\end{tabular}
\end{center}
 \caption[]{\label{BF-Schems}\emph{Dashed lines are the initial polylines, while the solid lines show the limit curves generated by the proposed schemes at different values of parameter $\alpha$.
 }}
\end{figure}
\begin{figure}[h!] 
\begin{center}
\begin{tabular}{ccc}
\epsfig{file=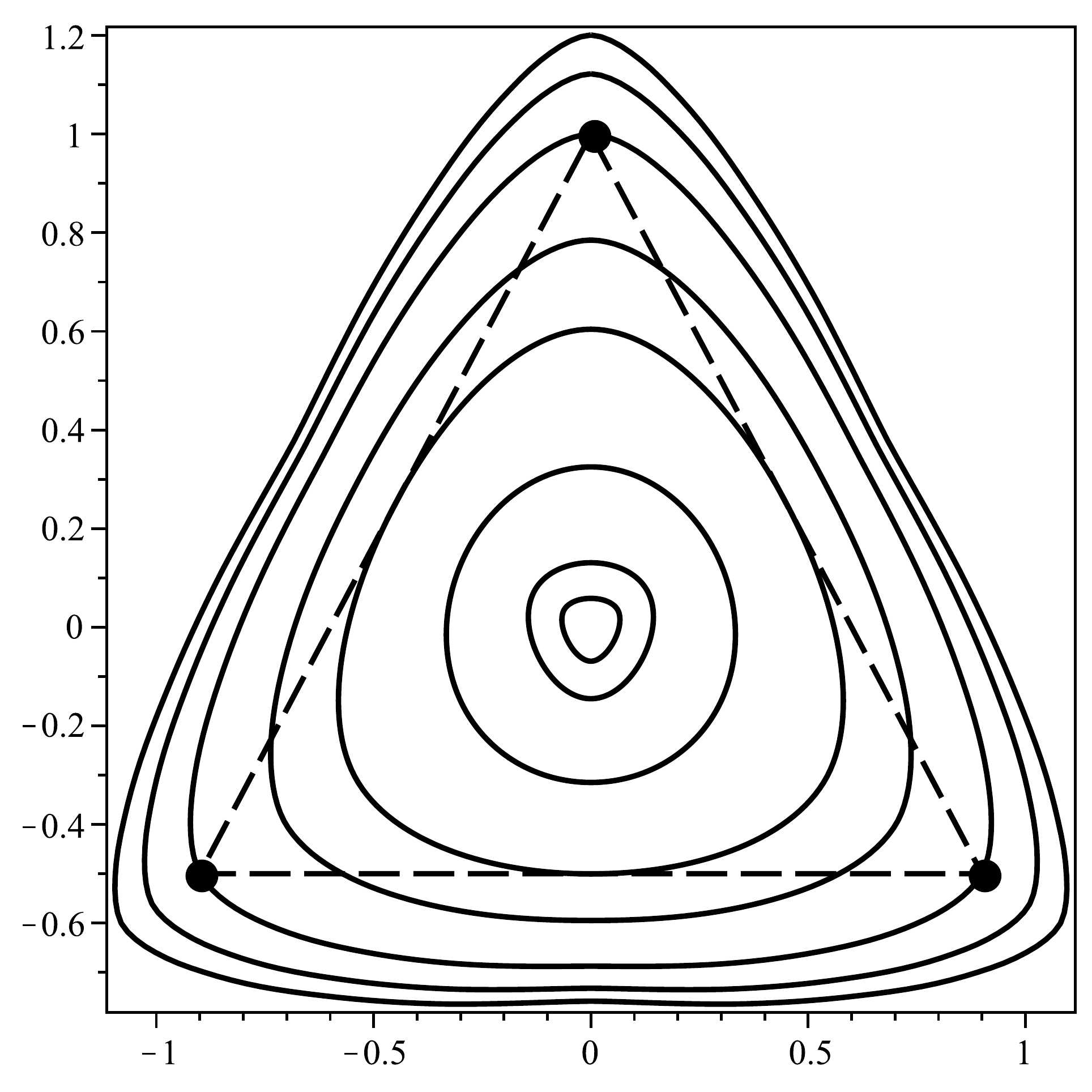, width=2.0 in} & \epsfig{file=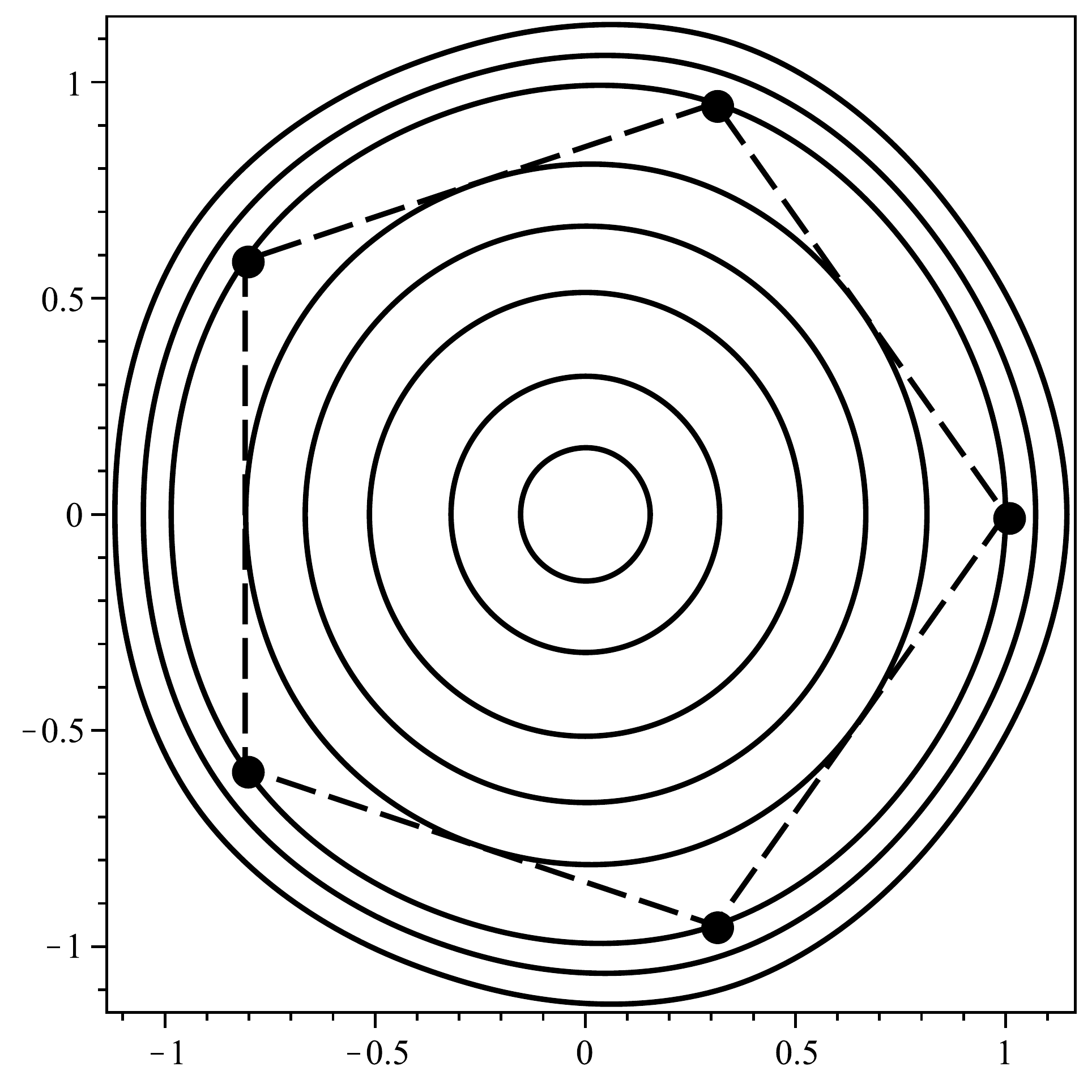, width=2.0 in} & \epsfig{file=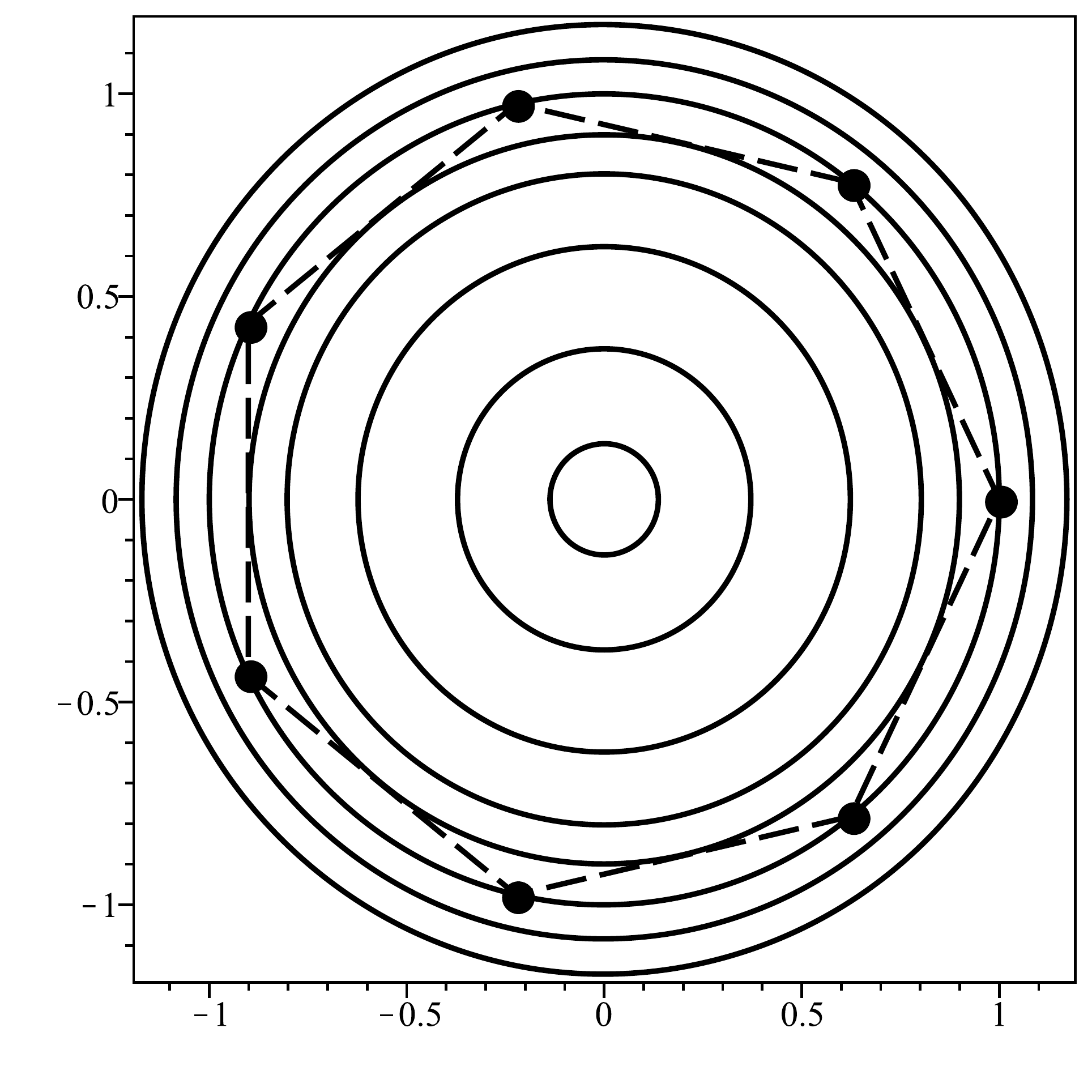, width=2.0 in} \\
(a) & (b) & (c)
\end{tabular}
\end{center}
 \caption[]{\label{scheme1}\emph{Black bullets are the initial control points, dashed lines polygons are the initial polygons, while black solid curves show the limit curves generated by the proposed subdivision schemes after 6 subdivision levels in (a) by the $4$-point scheme from outer to inner with $\alpha=0.2, 0.125, 0, -0.25, -0.5, -1,-1.5,-1.75$, in (b) by the $6$-point scheme from outer to inner with $\alpha=0.25, 0.125, 0, -0.35, -0.65, -1,-1.5,-2$ and in (c) by the $8$-point scheme from outer to inner with $\alpha=0.4, 0.2, 0, -0.25, -0.5, -1,-1.8,-2.7$.}}
\end{figure}

\begin{figure}[h!] 
\begin{center}
\begin{tabular}{ccc}
\epsfig{file=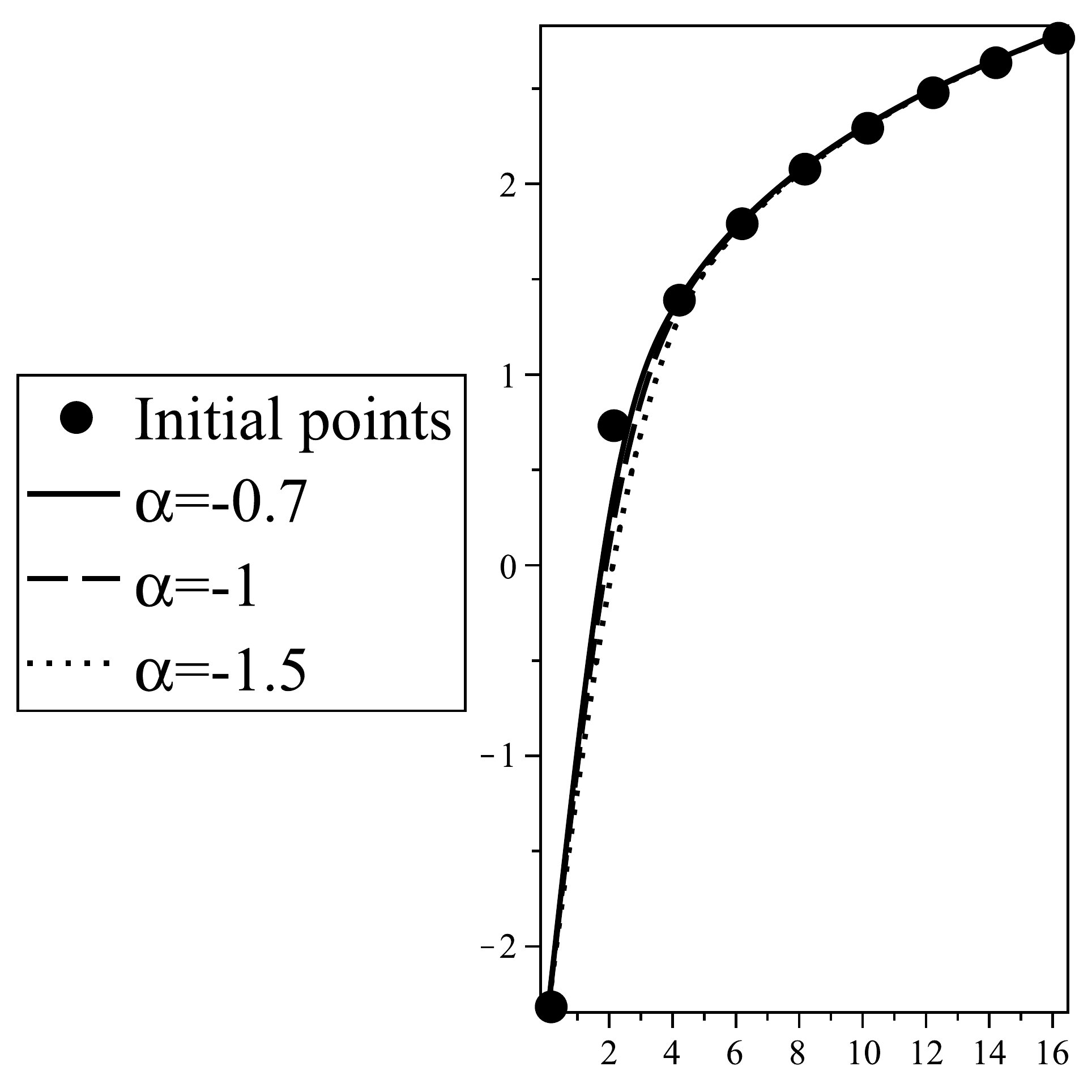, width=2.0 in} & \epsfig{file=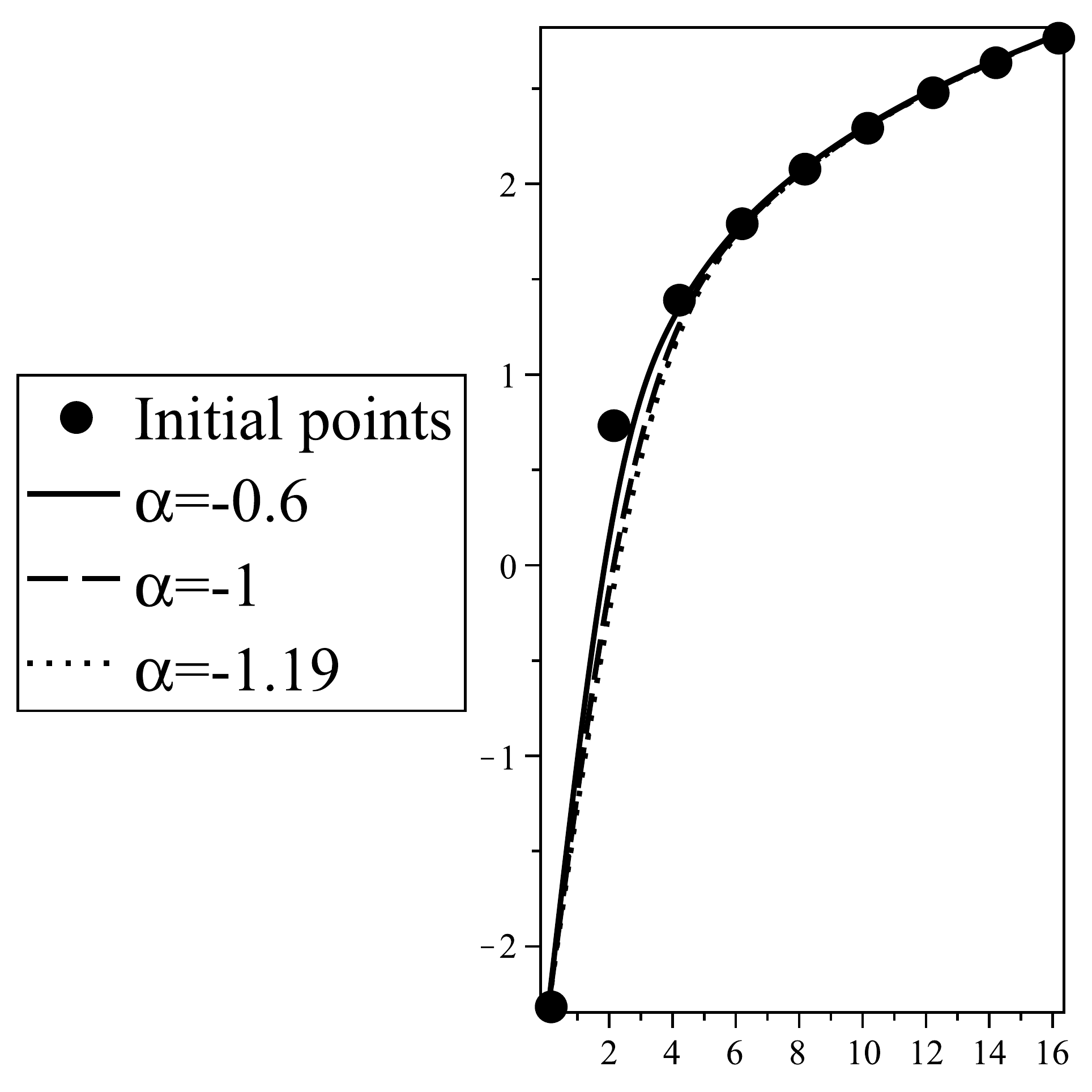, width=2.0 in} & \epsfig{file=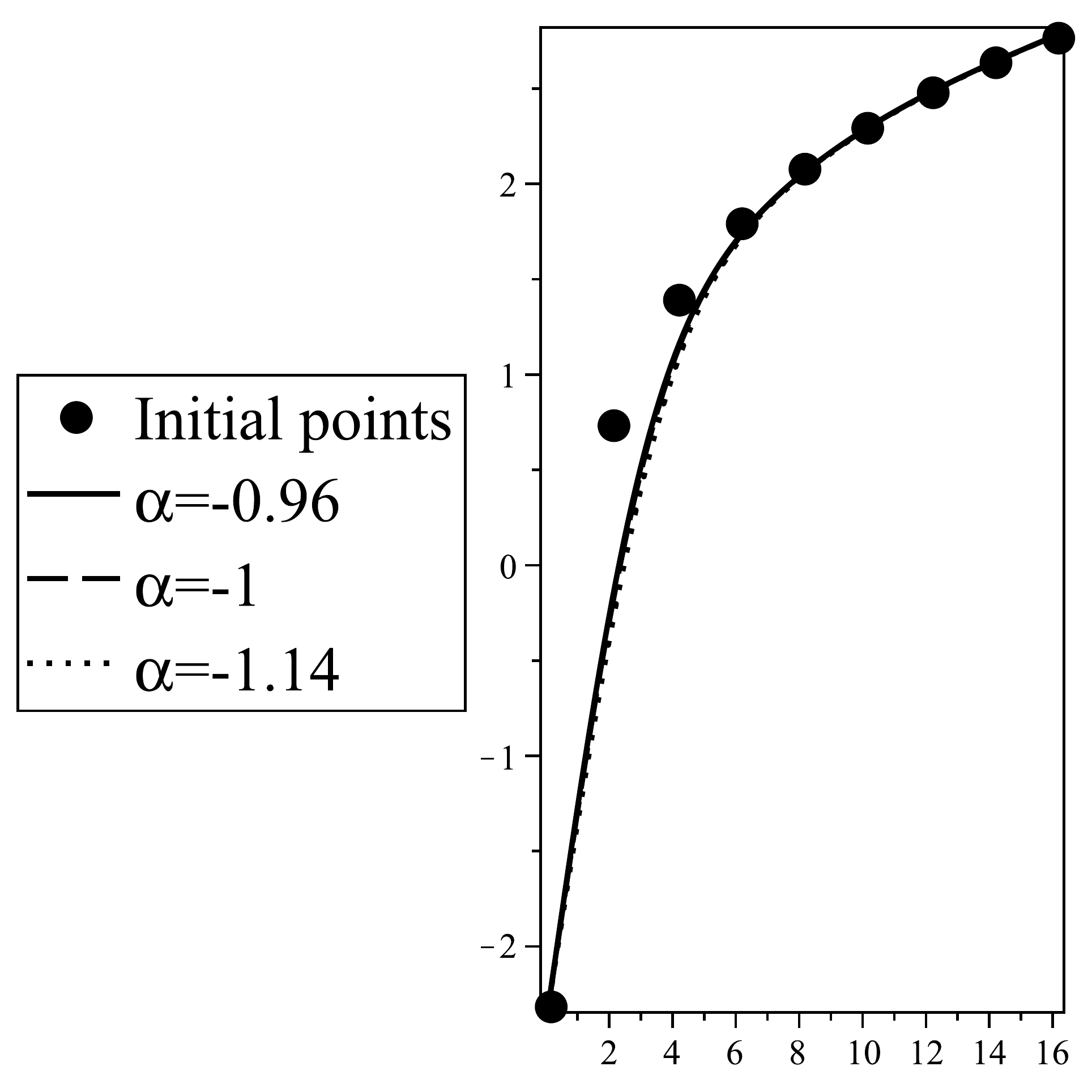, width=2.0 in} \\
(a) $4$-point scheme & (b) $6$-point scheme & (c) $8$-point scheme
\end{tabular}
\end{center}
 \caption[]{\label{monotonic}\emph{Curves generated by the proposed schemes after four subdivision steps where the initial data is taken from the monotonic function $f(x)=ln(x)$: $x \in (0,18)$.}}
\end{figure}

\begin{figure}[h!] 
\begin{center}
\begin{tabular}{ccc}
\epsfig{file=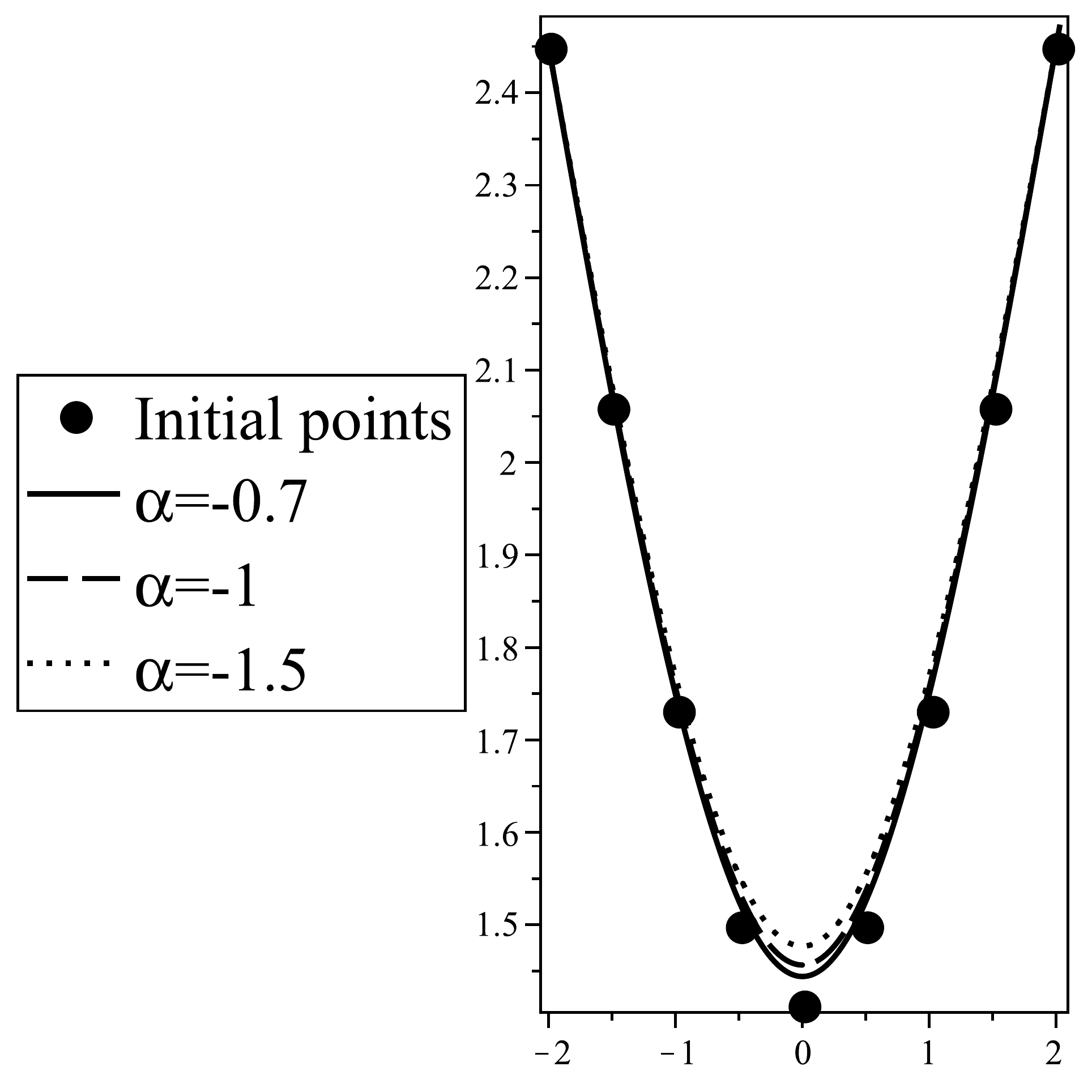, width=2.0 in} & \epsfig{file=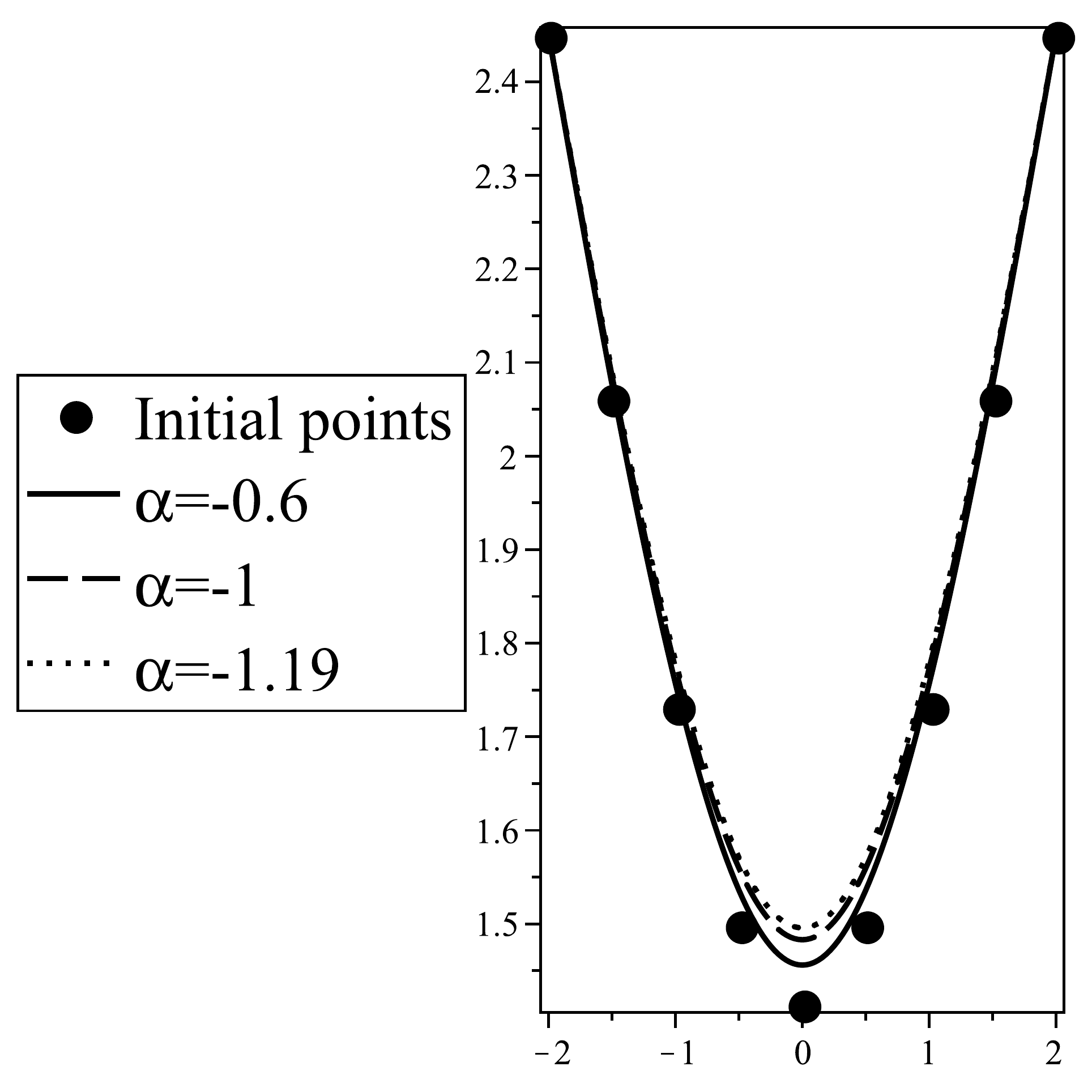, width=2.0 in} & \epsfig{file=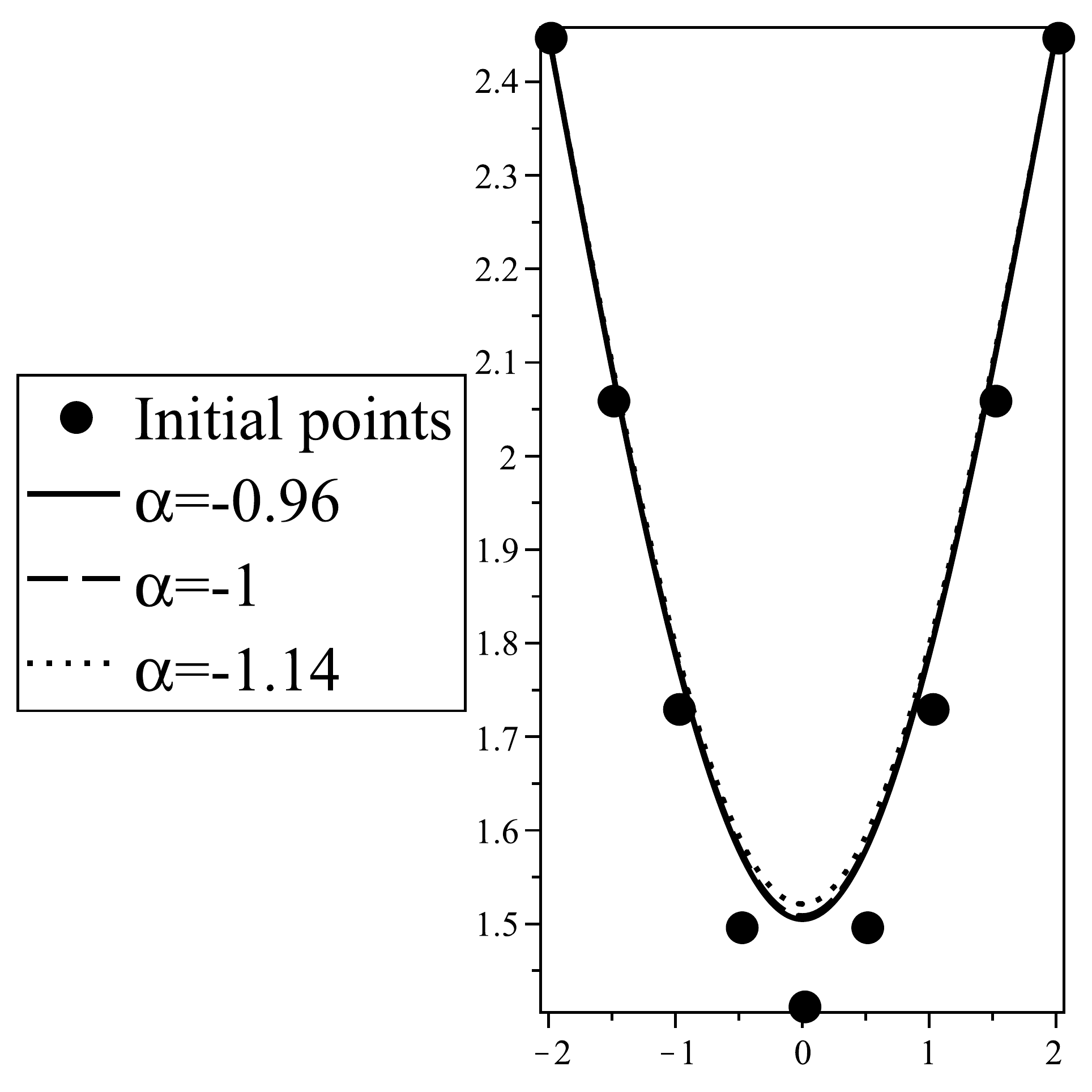, width=2.0 in} \\
(a) $4$-point scheme & (b) $6$-point scheme & (c) $8$-point scheme
\end{tabular}
\end{center}
 \caption[]{\label{convex}\emph{Curves generated by the proposed schemes after four subdivision steps where the initial data is taken from the convex function $f(x)=\sqrt{x^{2}+2}$.}}
\end{figure}

\begin{figure}[h!] 
\begin{center}
\begin{tabular}{ccc}
\epsfig{file=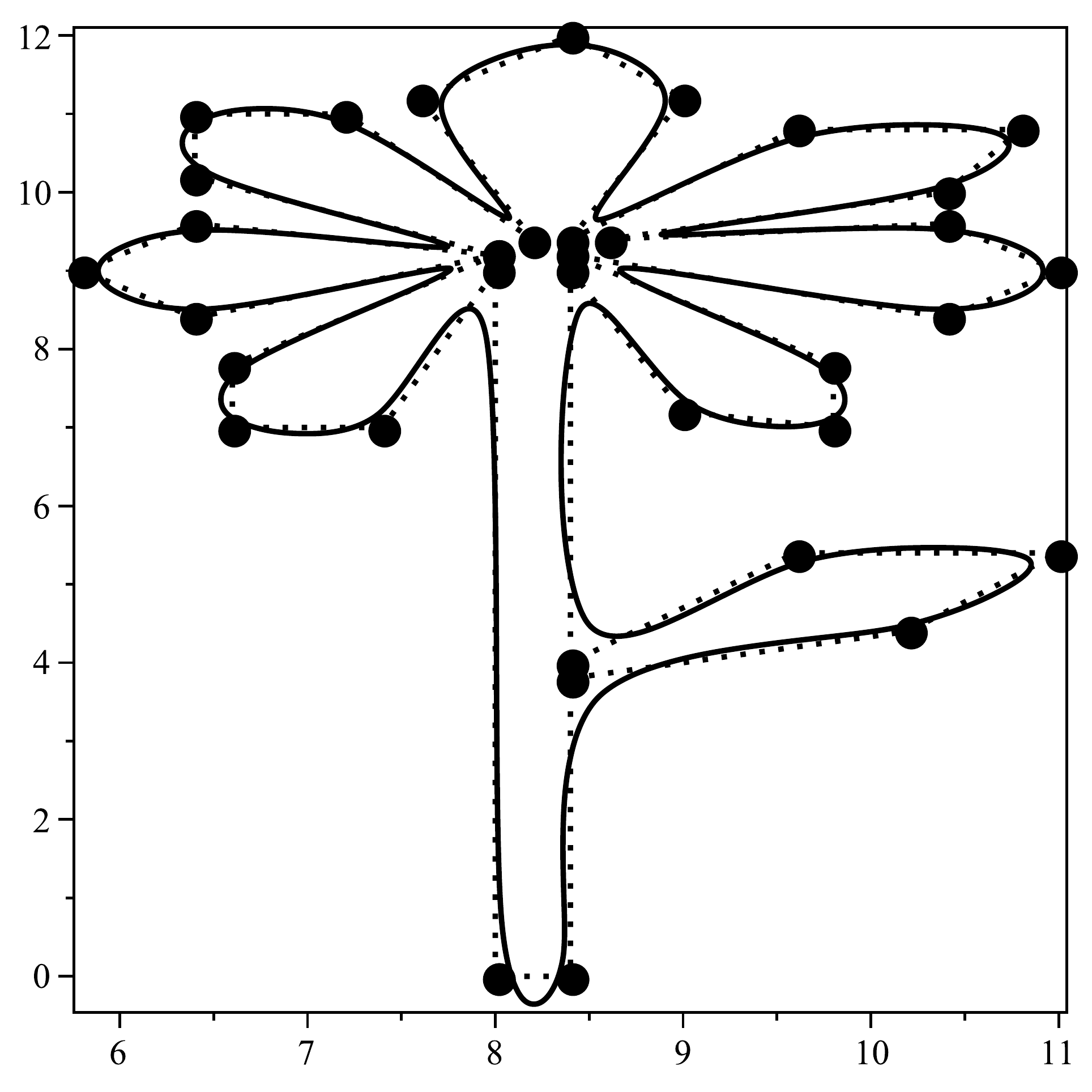, width=2.0 in} & \epsfig{file=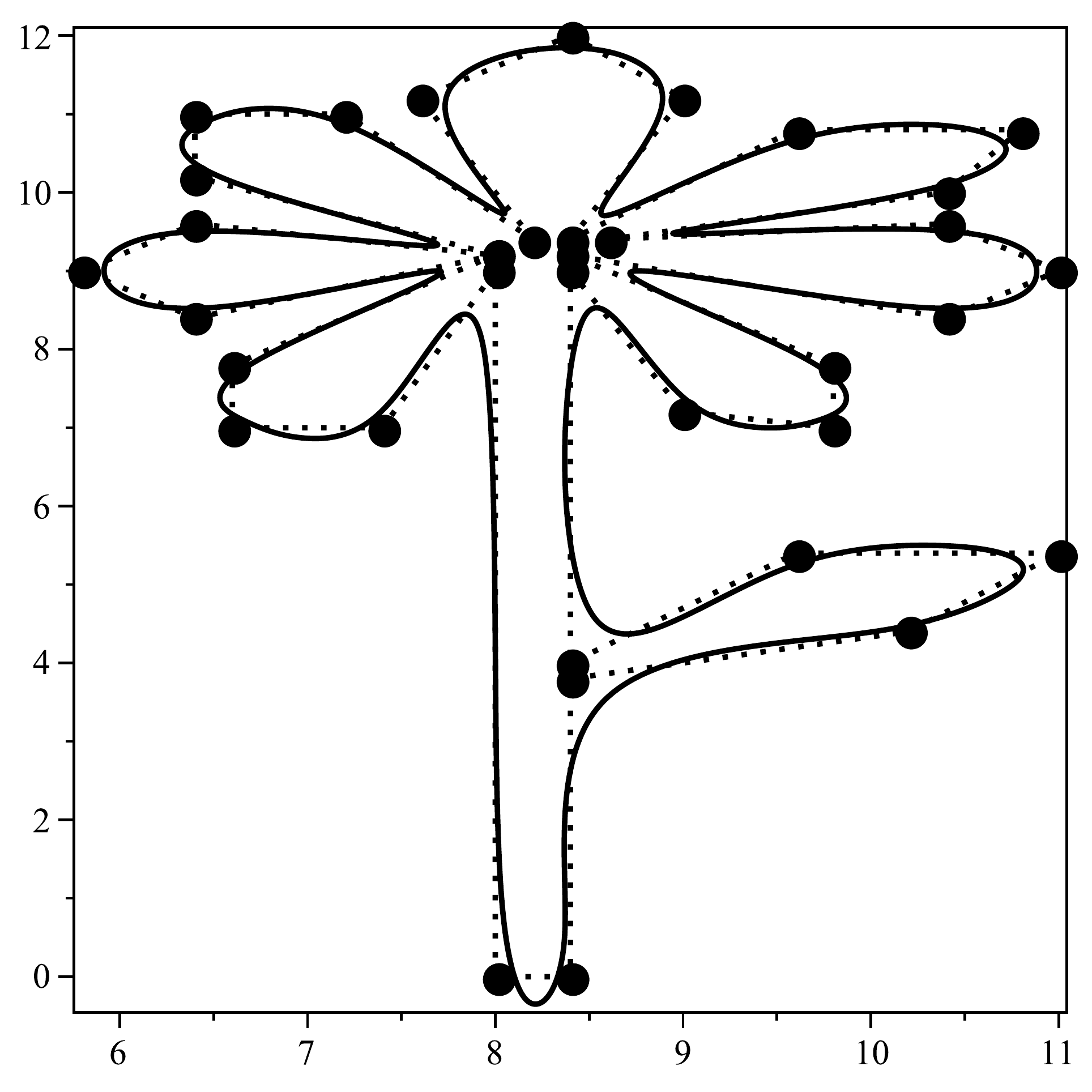, width=2.0 in} & \epsfig{file=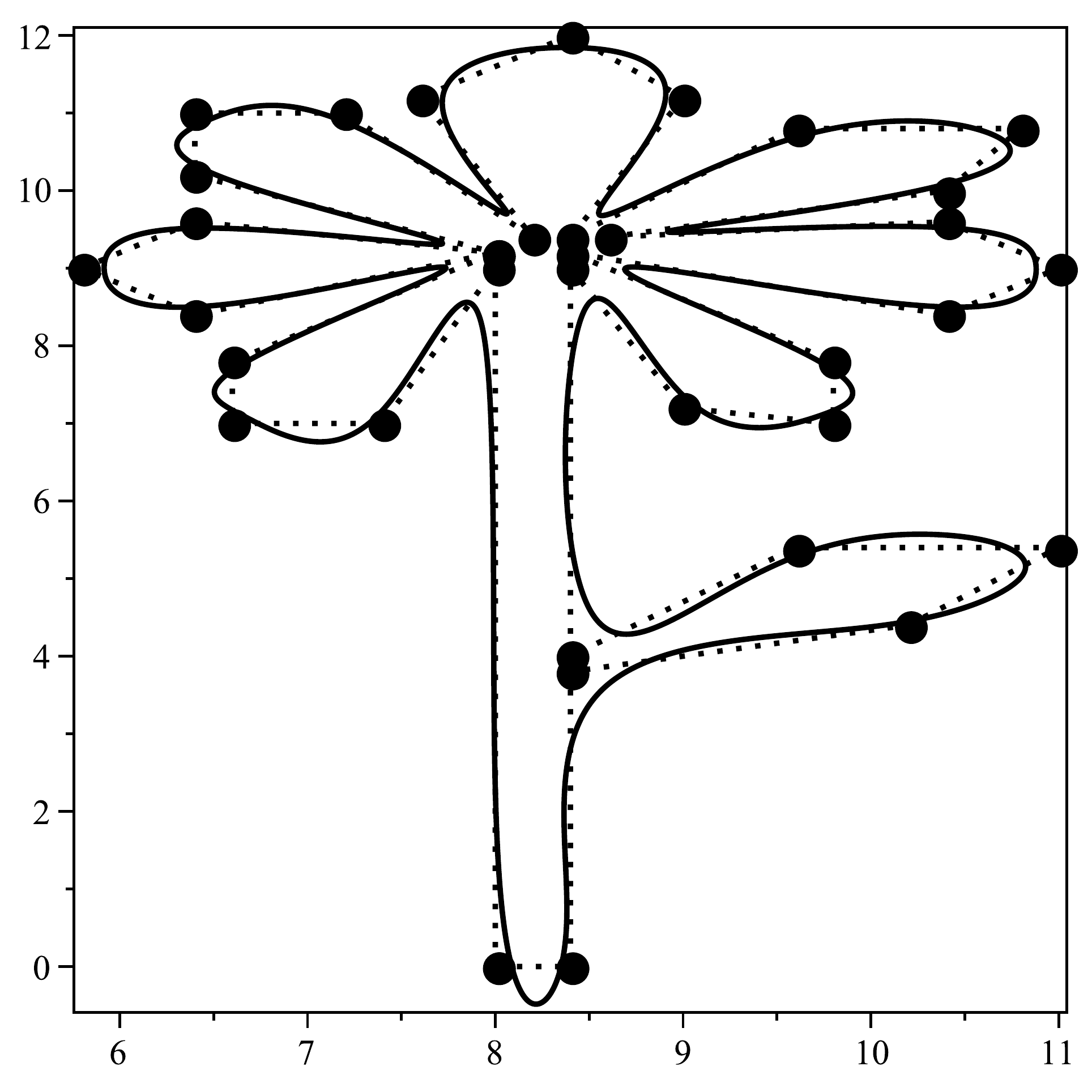, width=2.0 in} \\
(a) $\alpha=-0.28$ & (b) $\alpha=-0.25$ & (c) $\alpha=-0.19$\\
\end{tabular}
\end{center}
 \caption[]{\label{flower}\emph{Solid lines in (a), (b) and (c) show the limit curve generated by the $4$-point, $6$-point and $8$-point proposed schemes.
 }}
\end{figure}

\begin{figure}[h!] 
\begin{center}
\begin{tabular}{ccc}
\epsfig{file=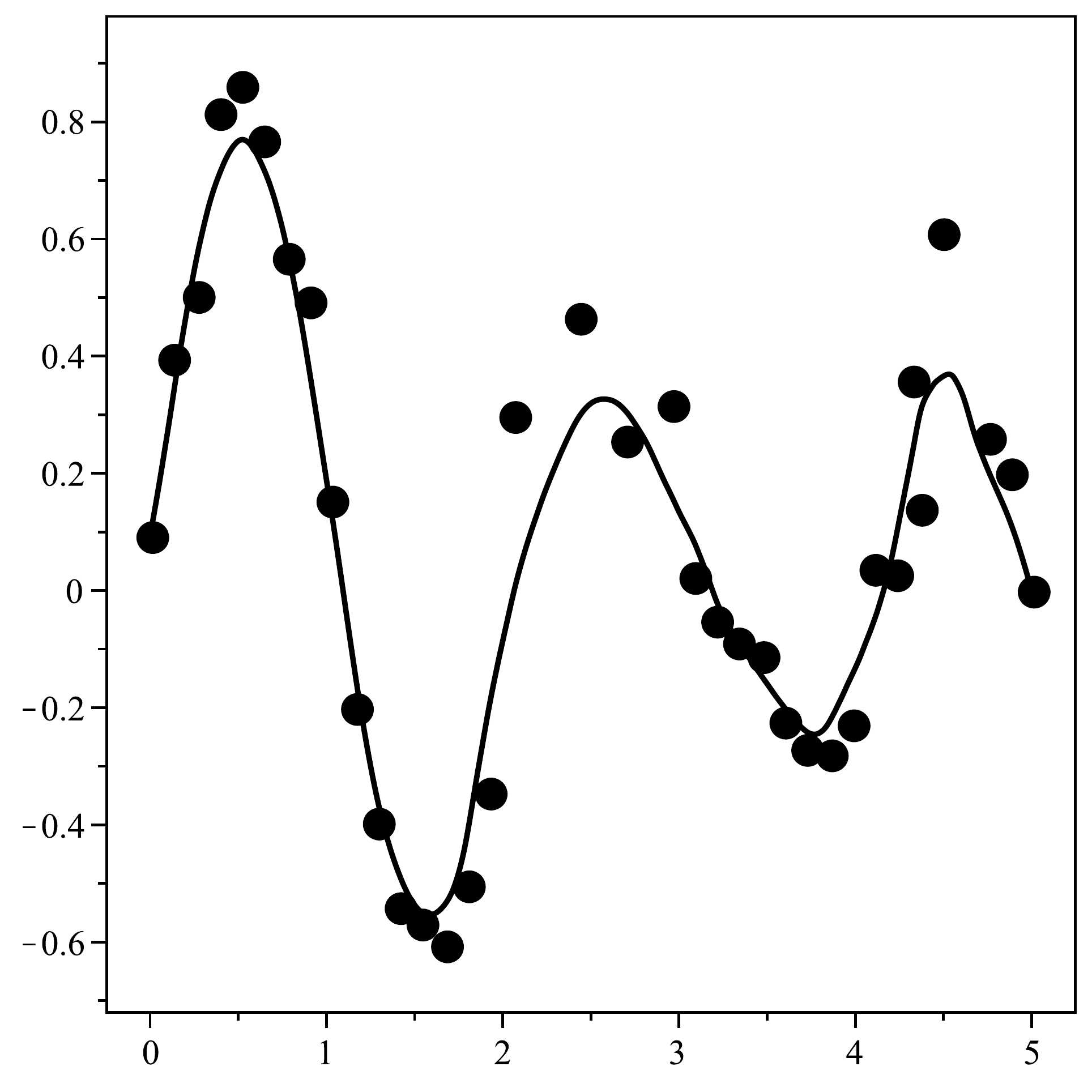, width=2.0 in} & \epsfig{file=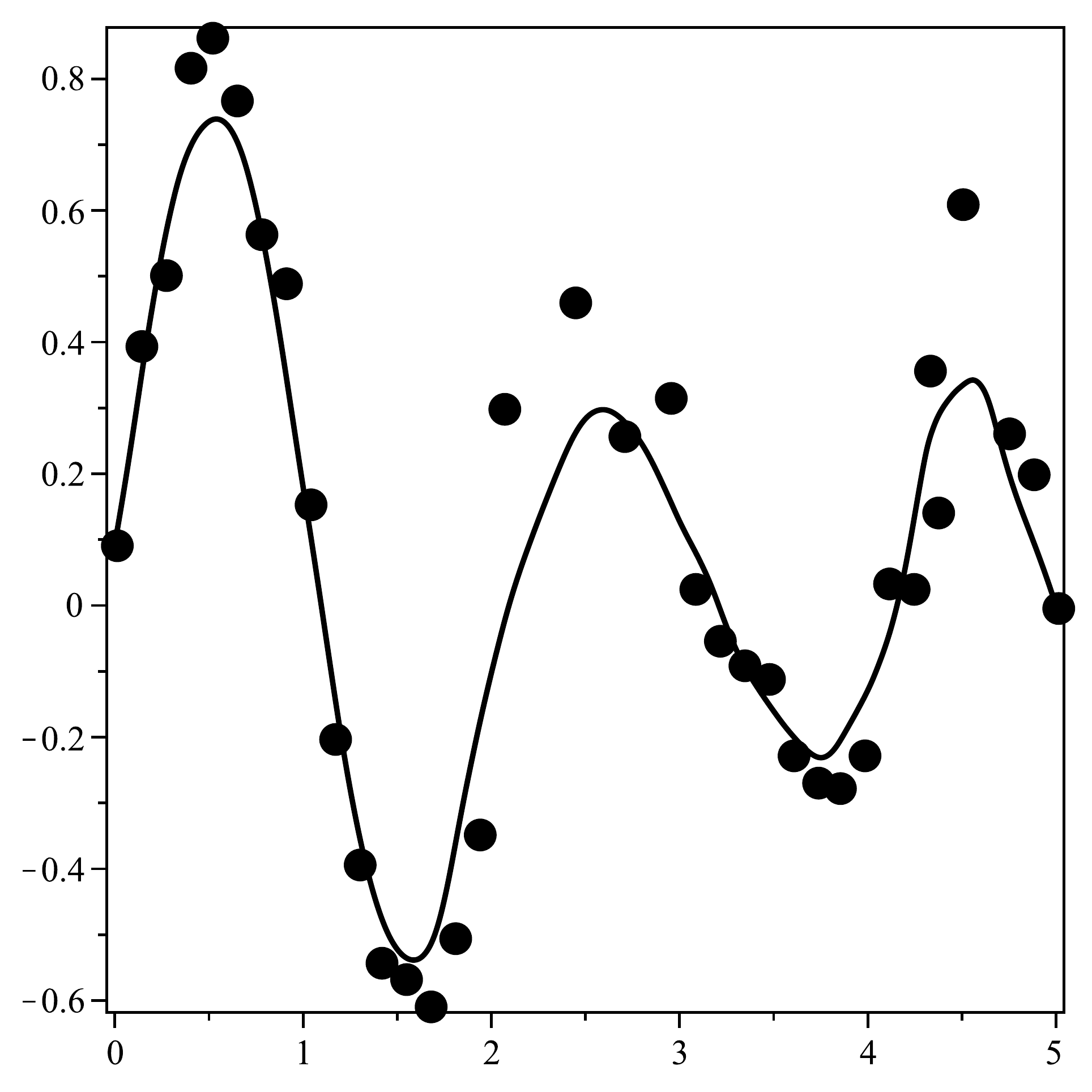, width=2.0 in} & \epsfig{file=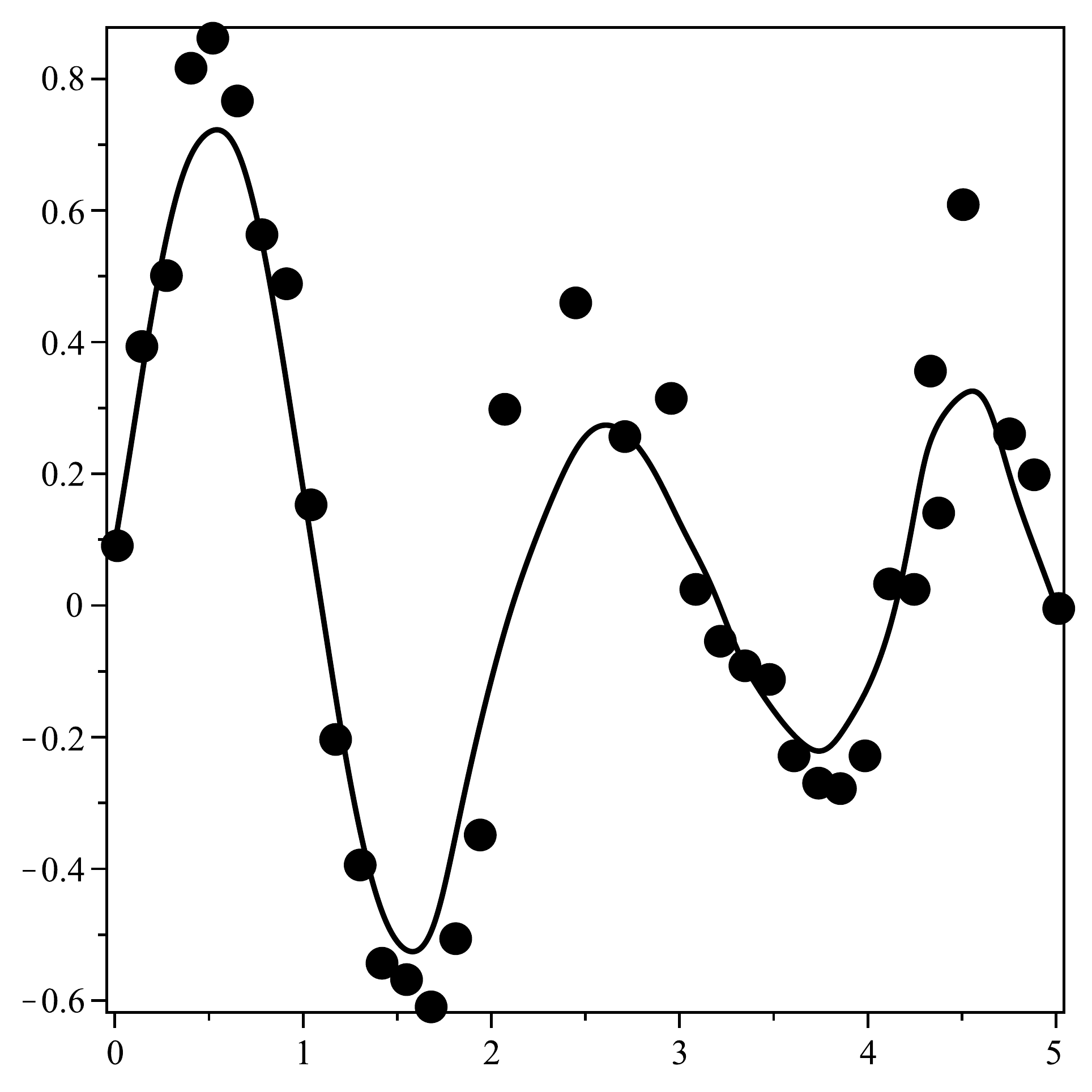, width=2.0 in} \\
(a) $\alpha=-2.5$ & (b) $\alpha=-2$ & (c) $\alpha=-1.7$\\
\end{tabular}
\end{center}
 \caption[]{\label{noise}\emph{Solid lines in (a), (b) and (c) show the limit curve generated by the $4$-point, $6$-point and $8$-point proposed schemes.
 }}
\end{figure}

\begin{landscape}
\begin{figure}[h!] 
\begin{center}
\begin{tabular}{cccc}
\epsfig{file=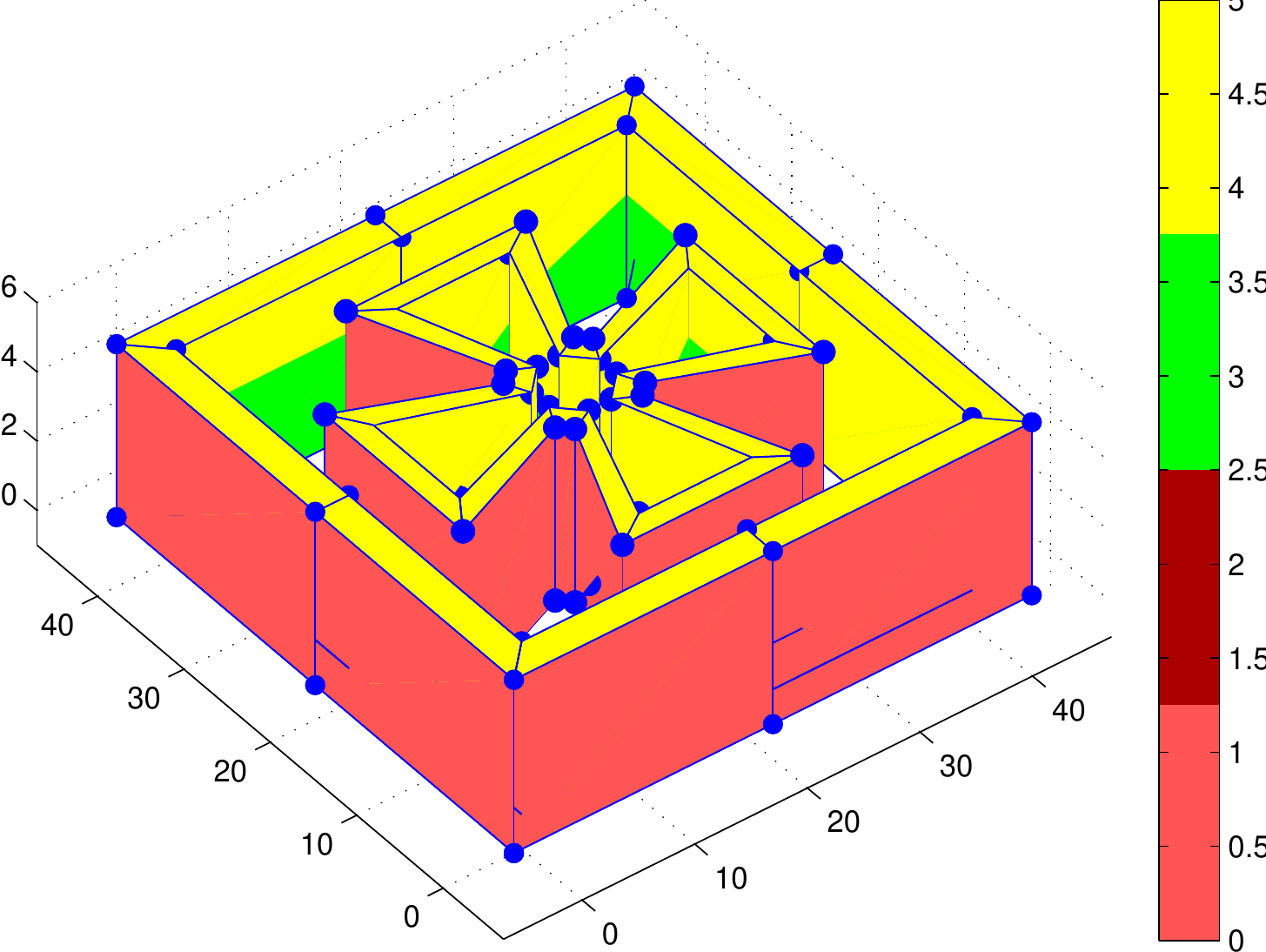, width=1.8 in} & \epsfig{file=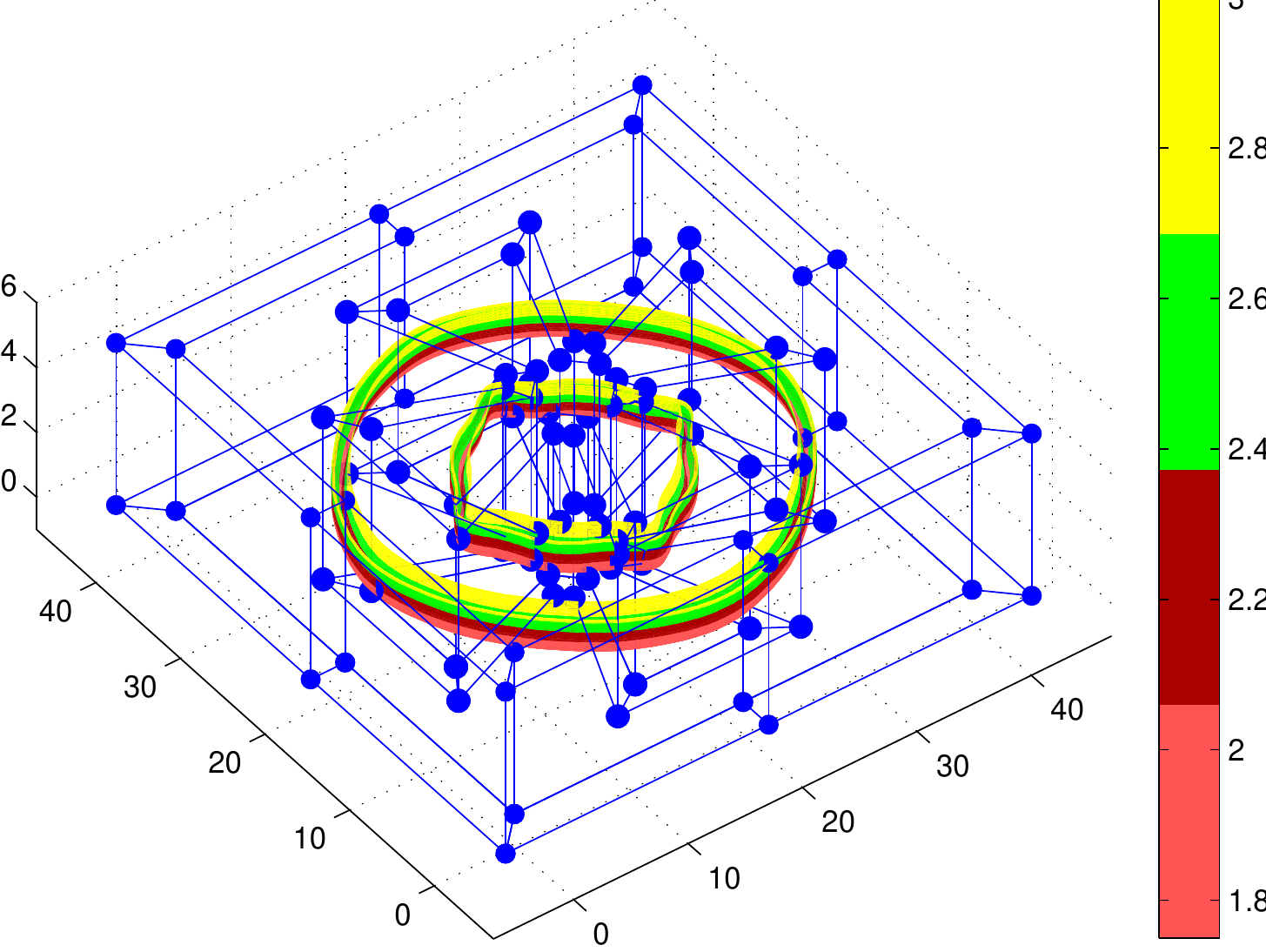, width=1.8 in} & \epsfig{file=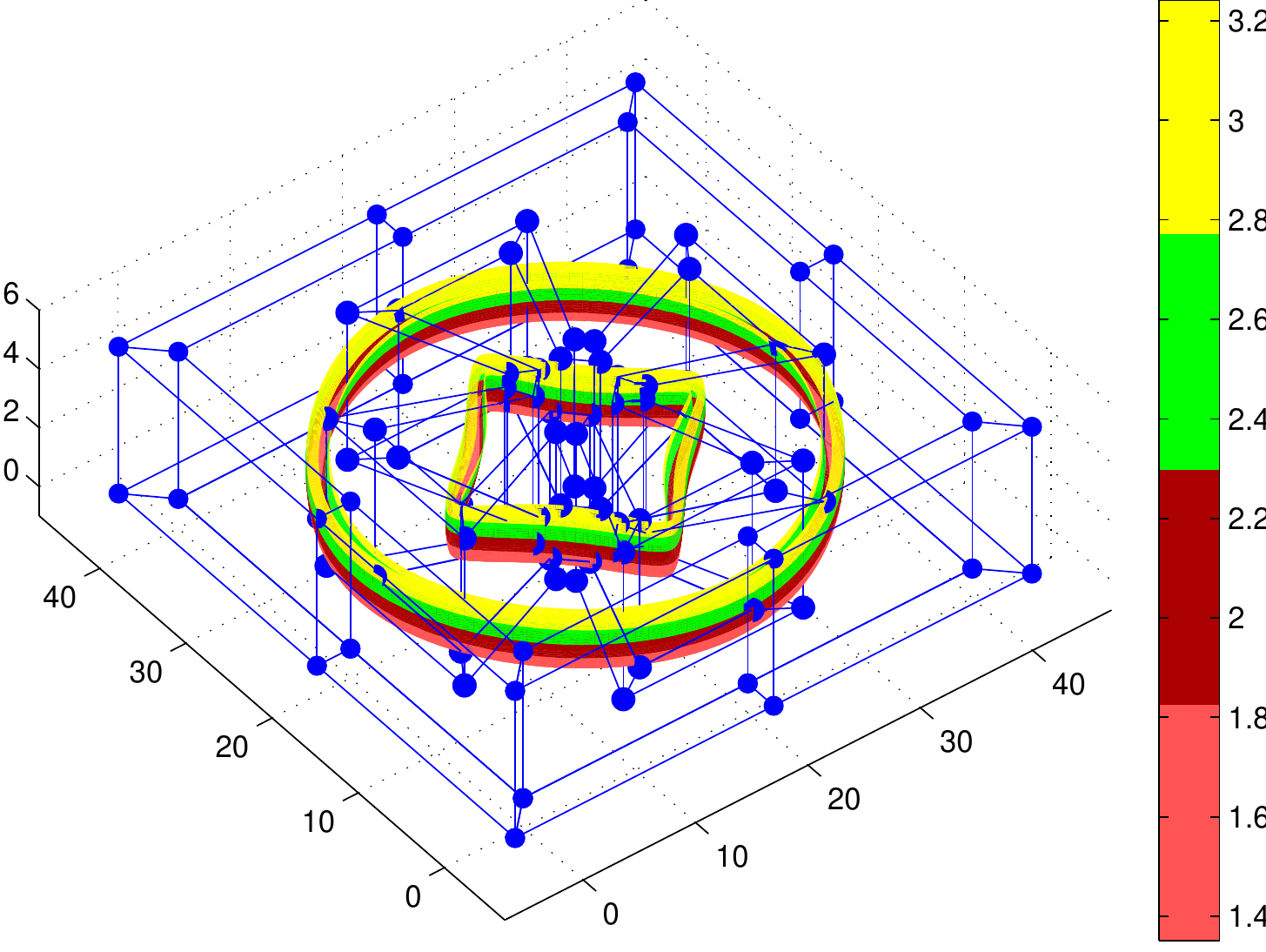, width=1.8 in} & \epsfig{file=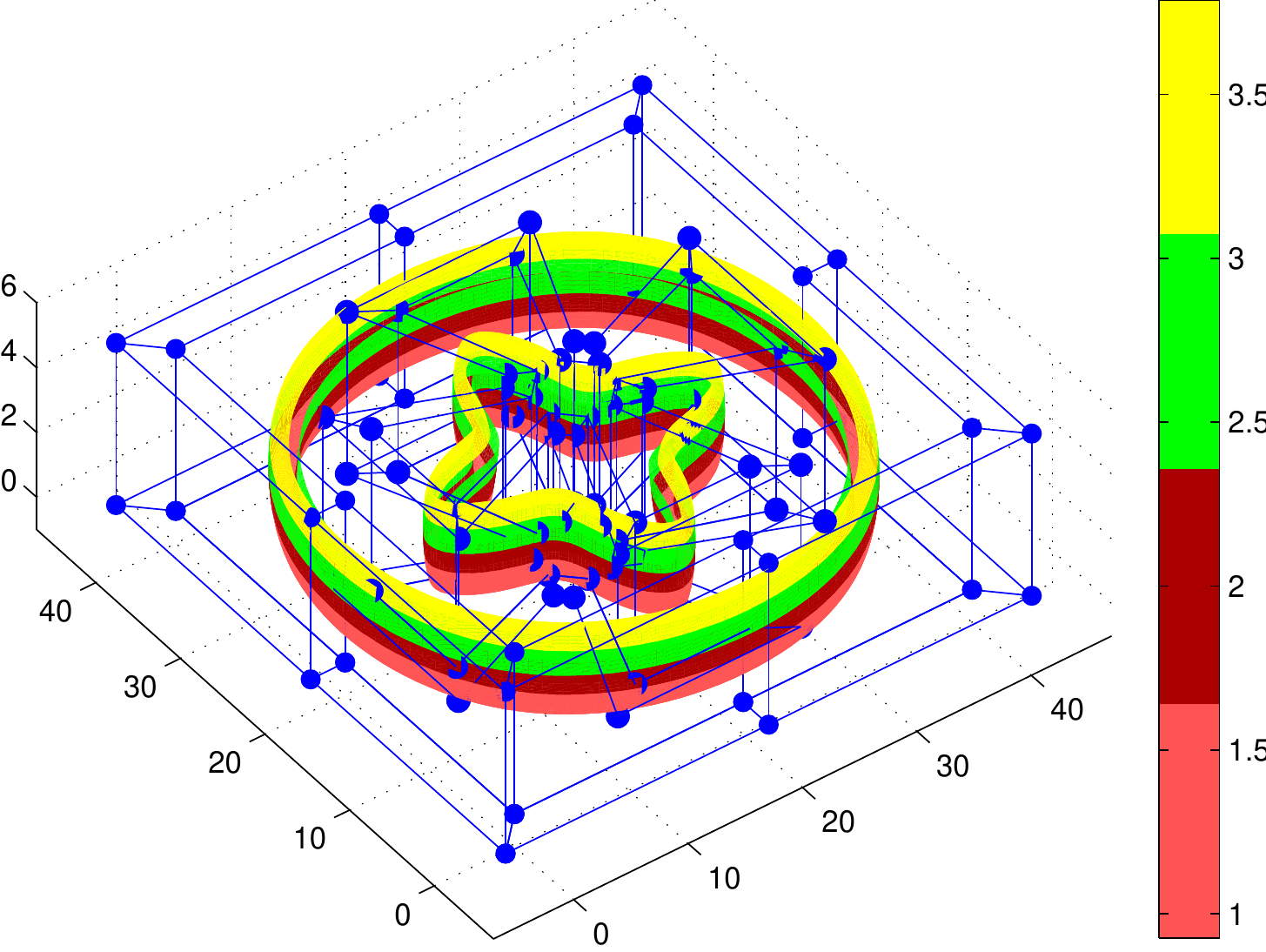, width=1.8 in}  \\
(a) initial mesh & (b) $\alpha=-2.5$ & (c) $\alpha=-2$ & (d) $\alpha=-1.5$\\
\epsfig{file=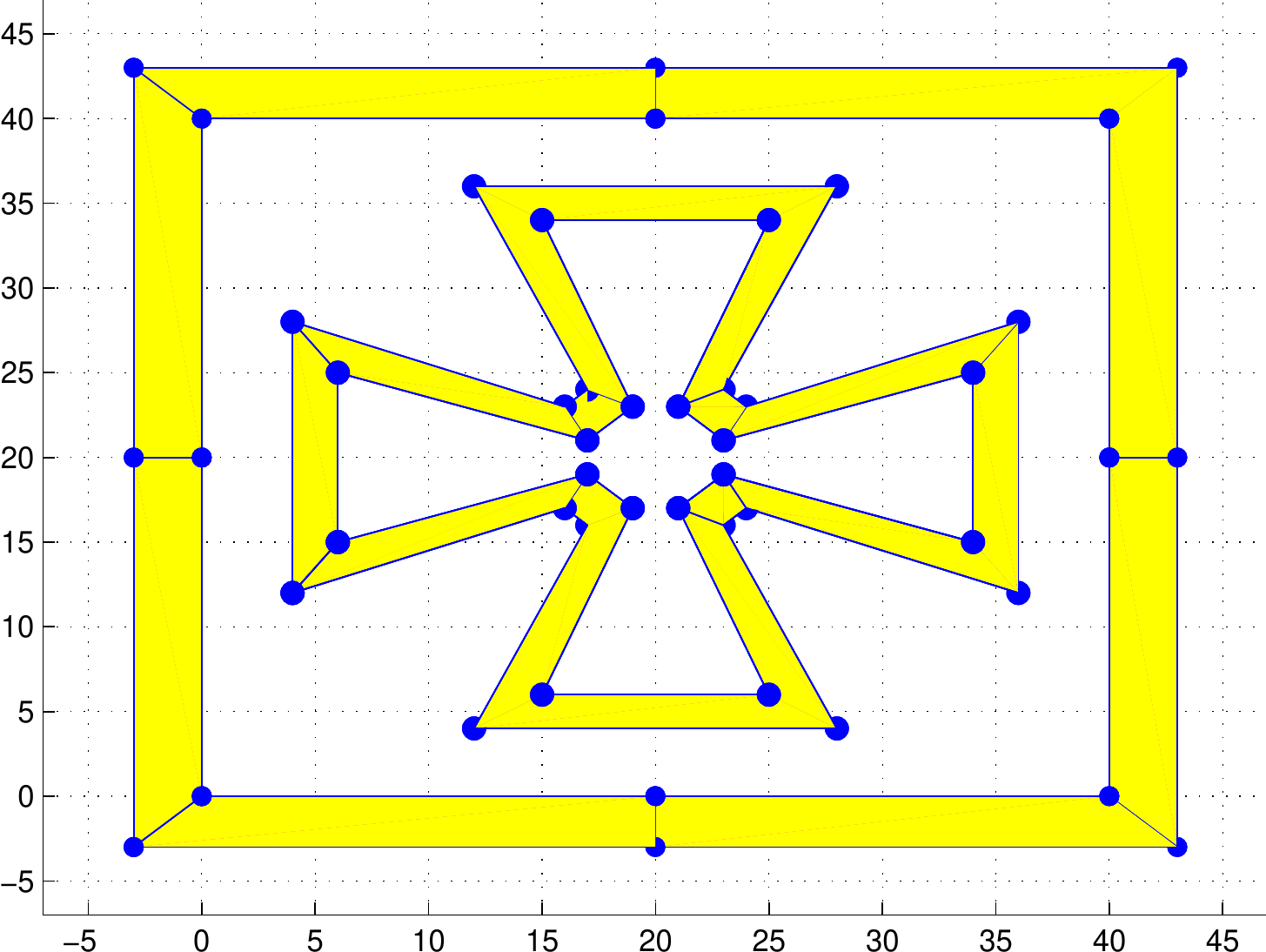, width=1.8 in} & \epsfig{file=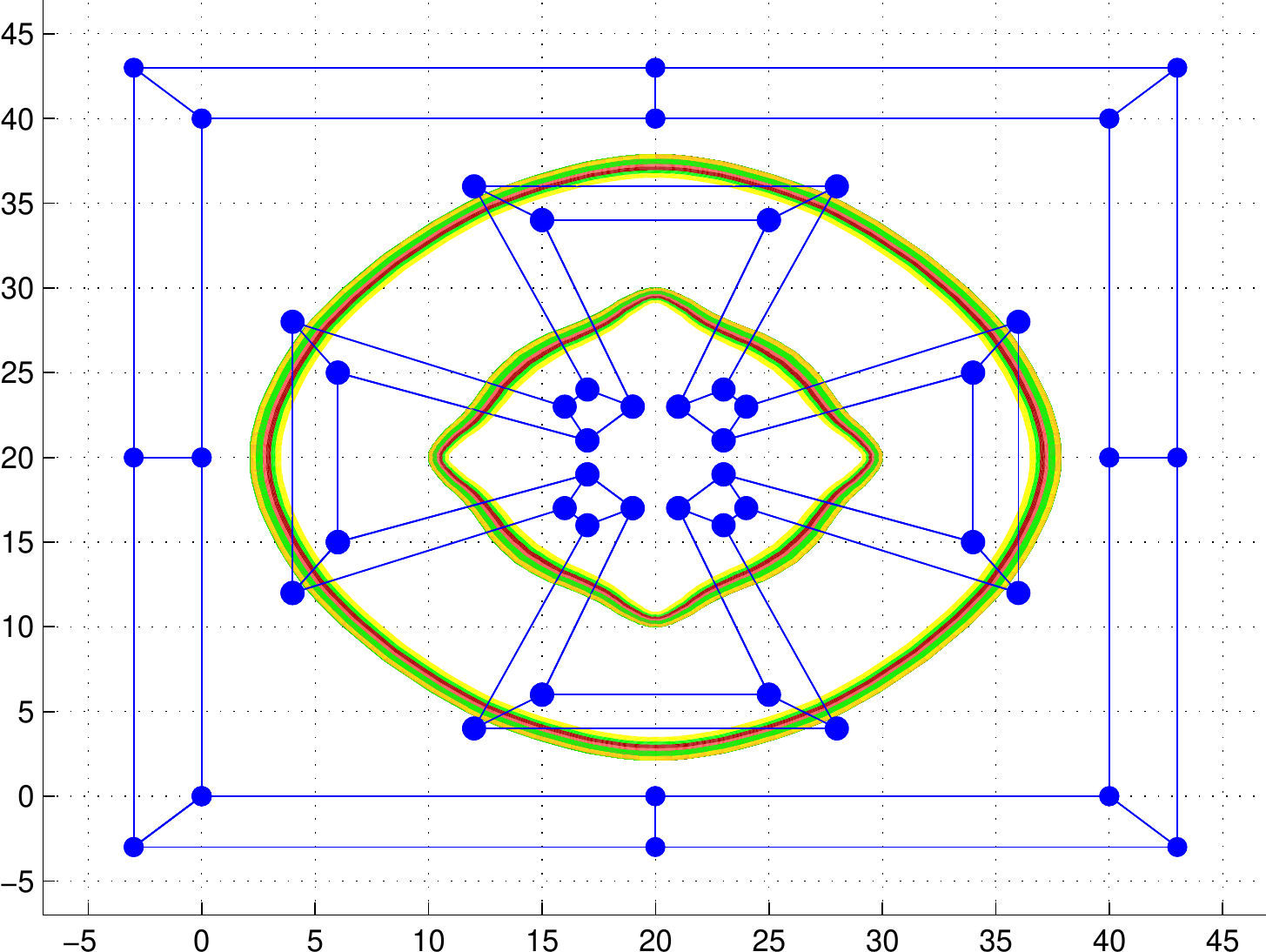, width=1.8 in} & \epsfig{file=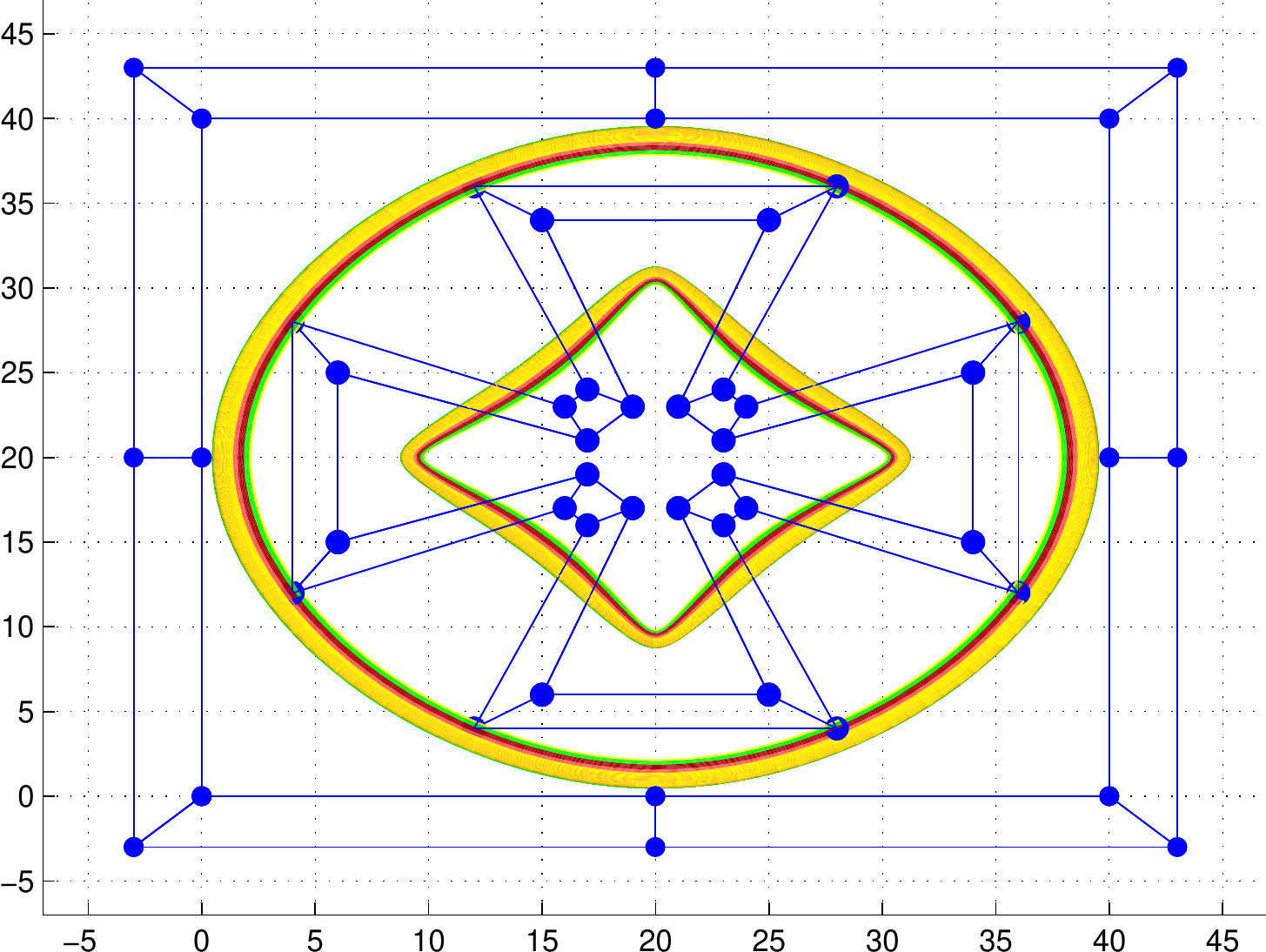, width=1.8 in} & \epsfig{file=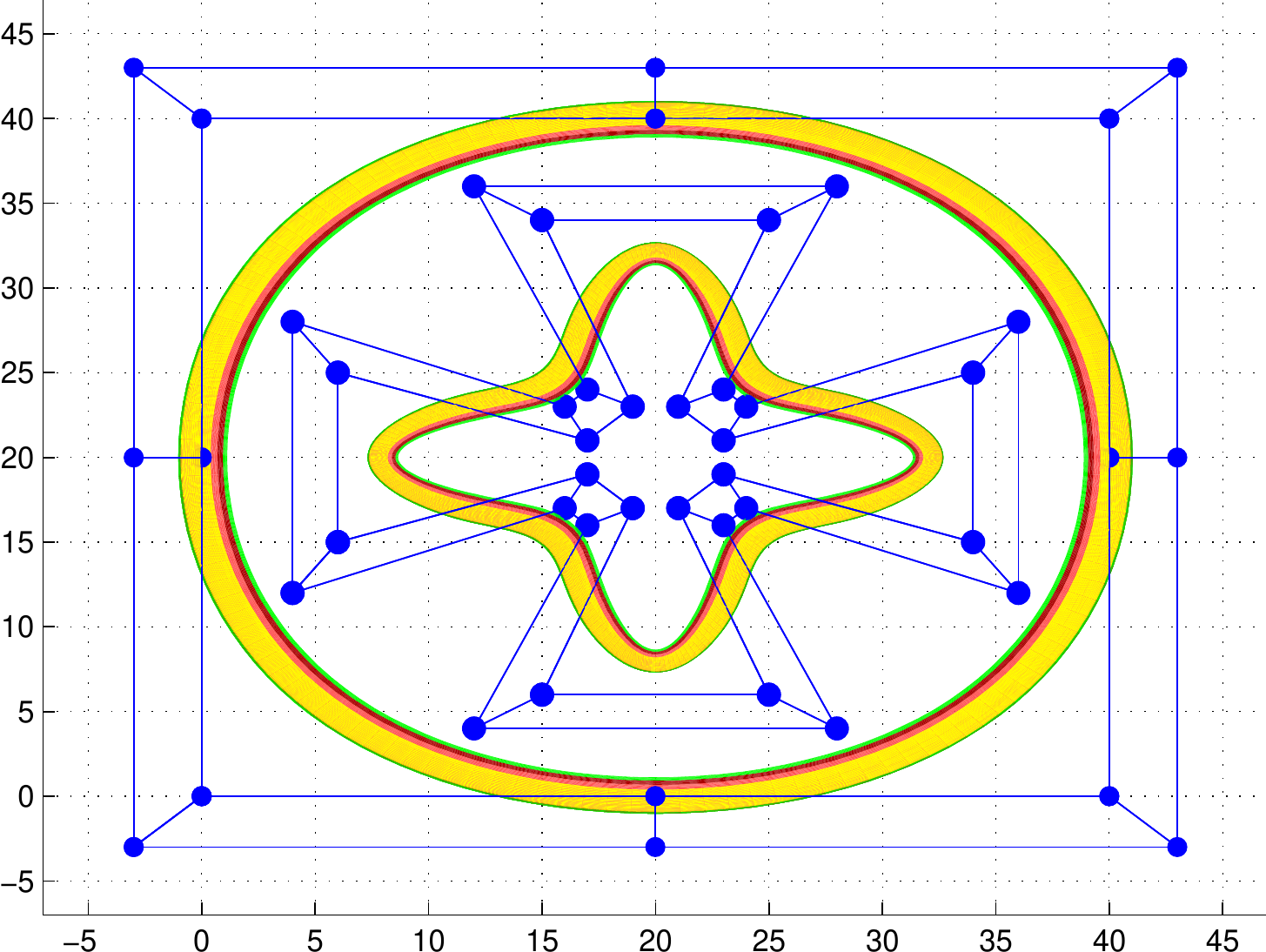, width=1.8 in}  \\
(e) & (f) & (g) & (h) \\
\epsfig{file=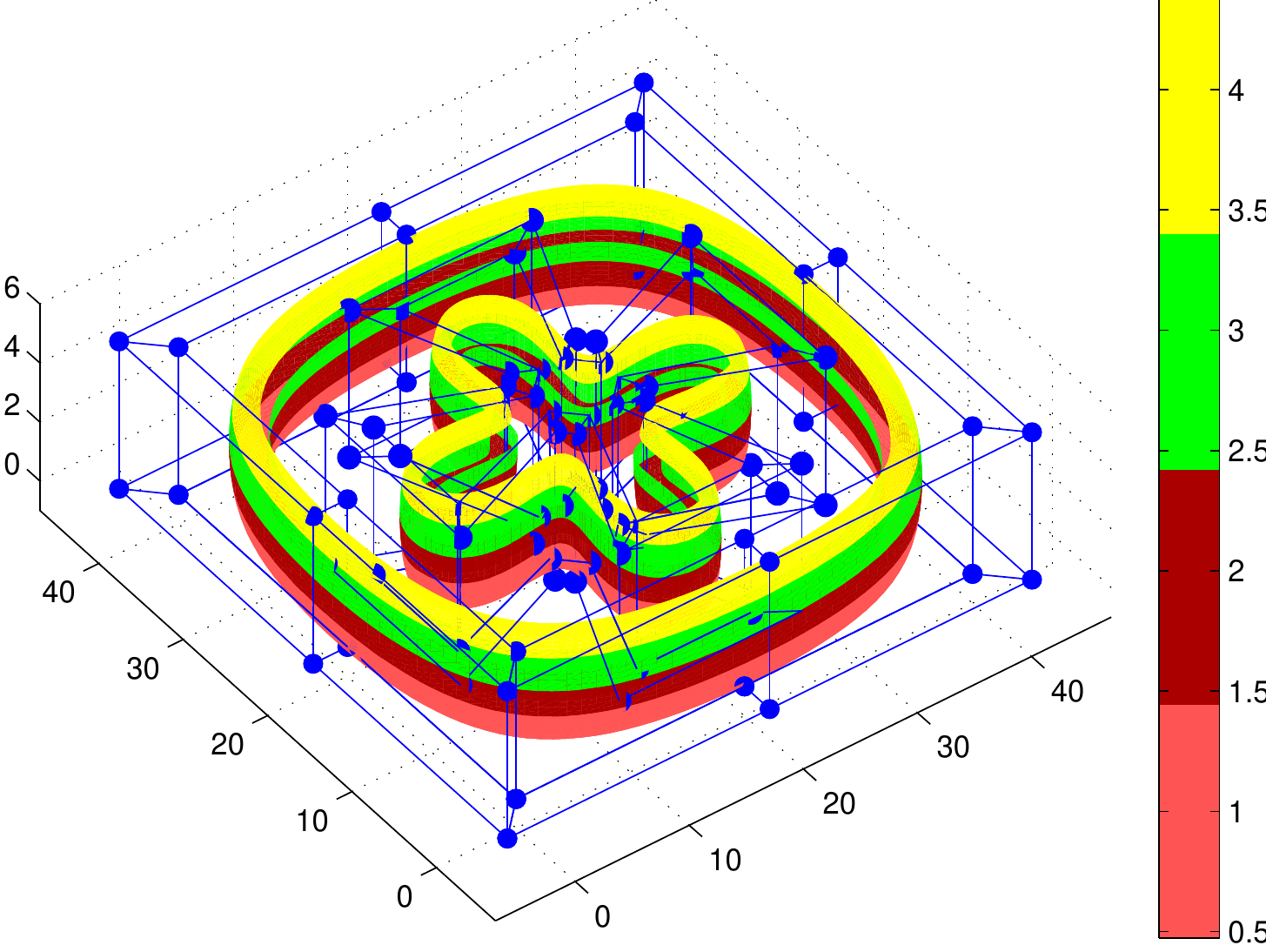, width=1.8 in} & \epsfig{file=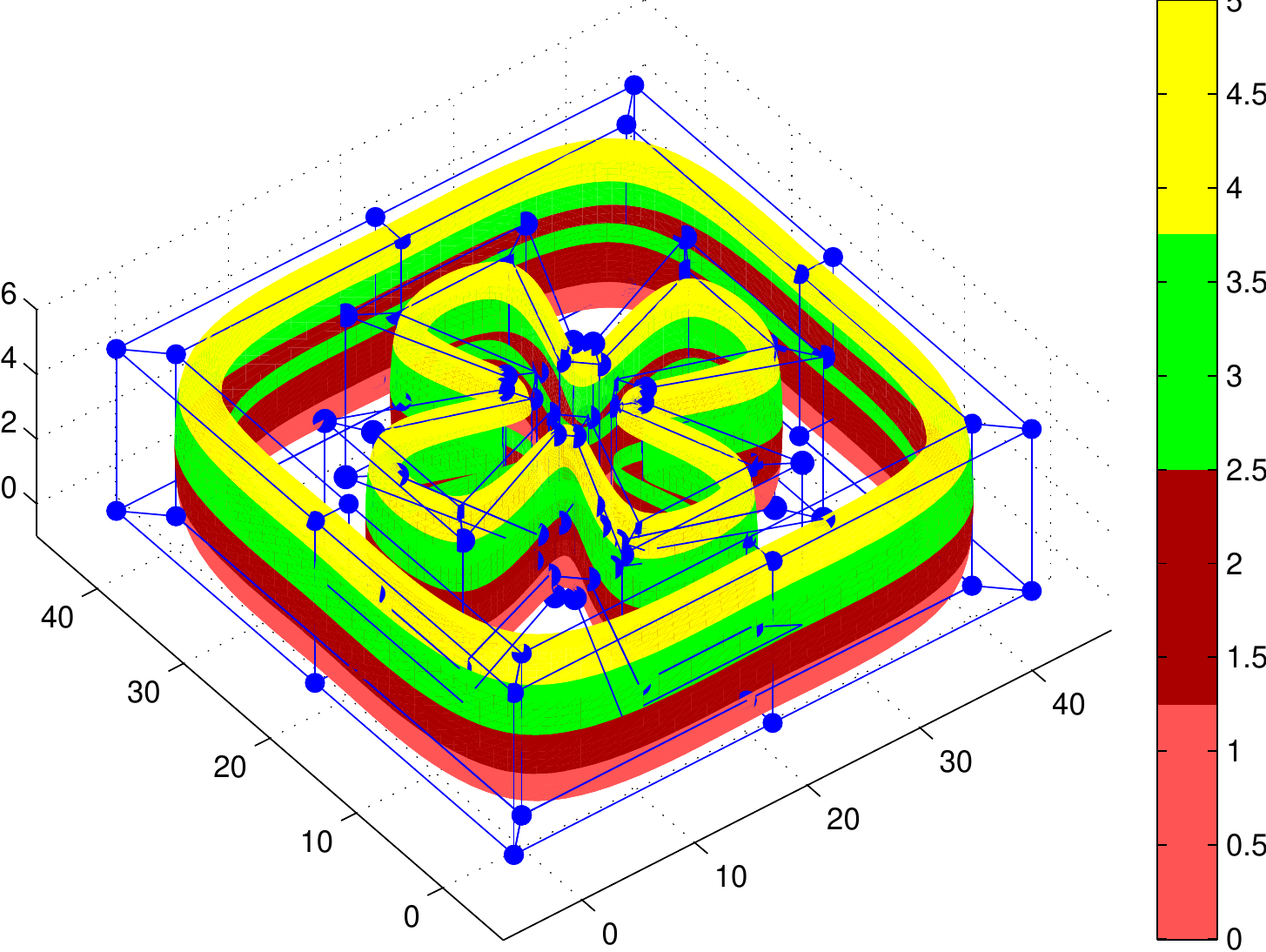, width=1.8 in} & \epsfig{file=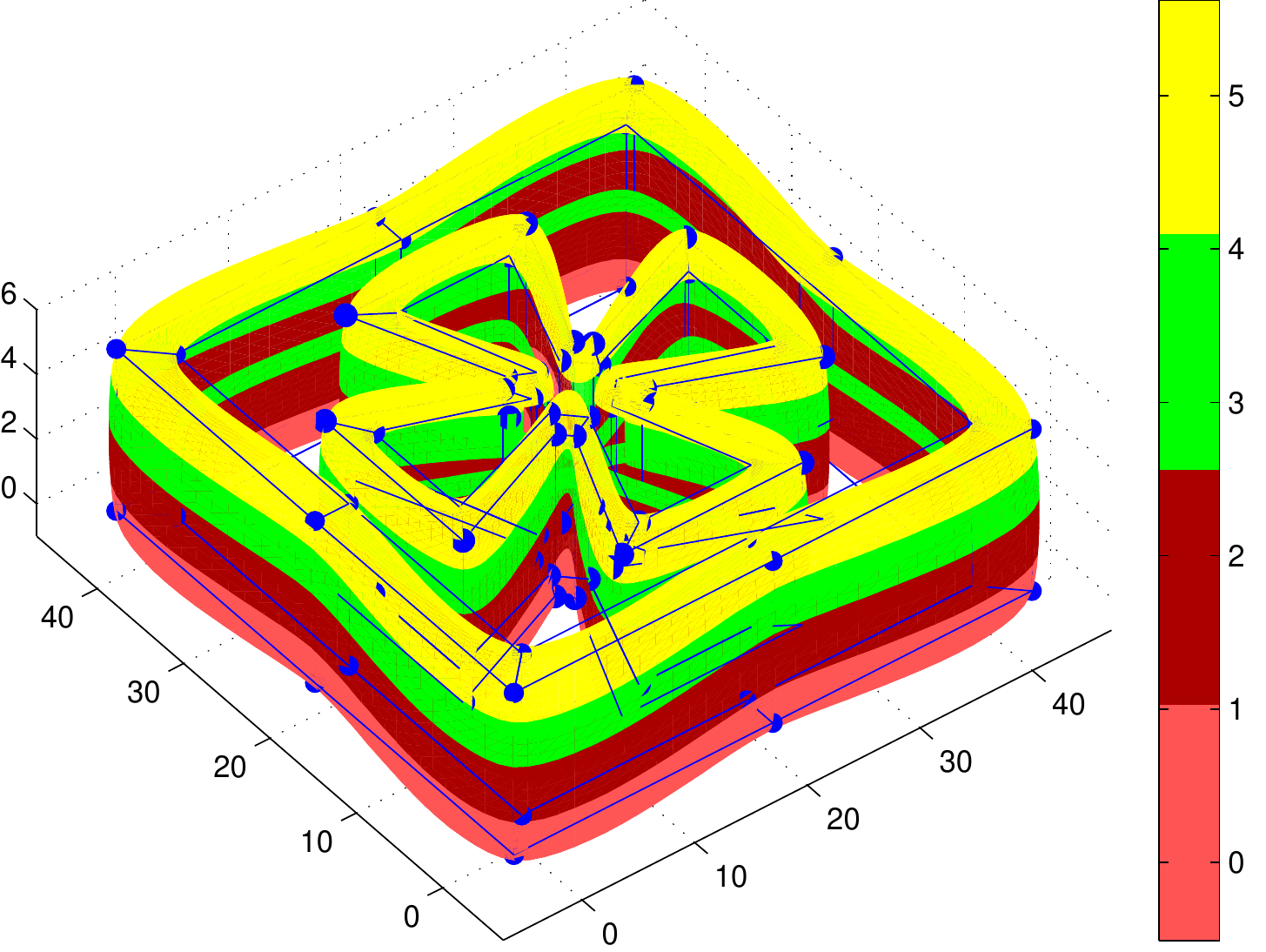, width=1.8 in} & \epsfig{file=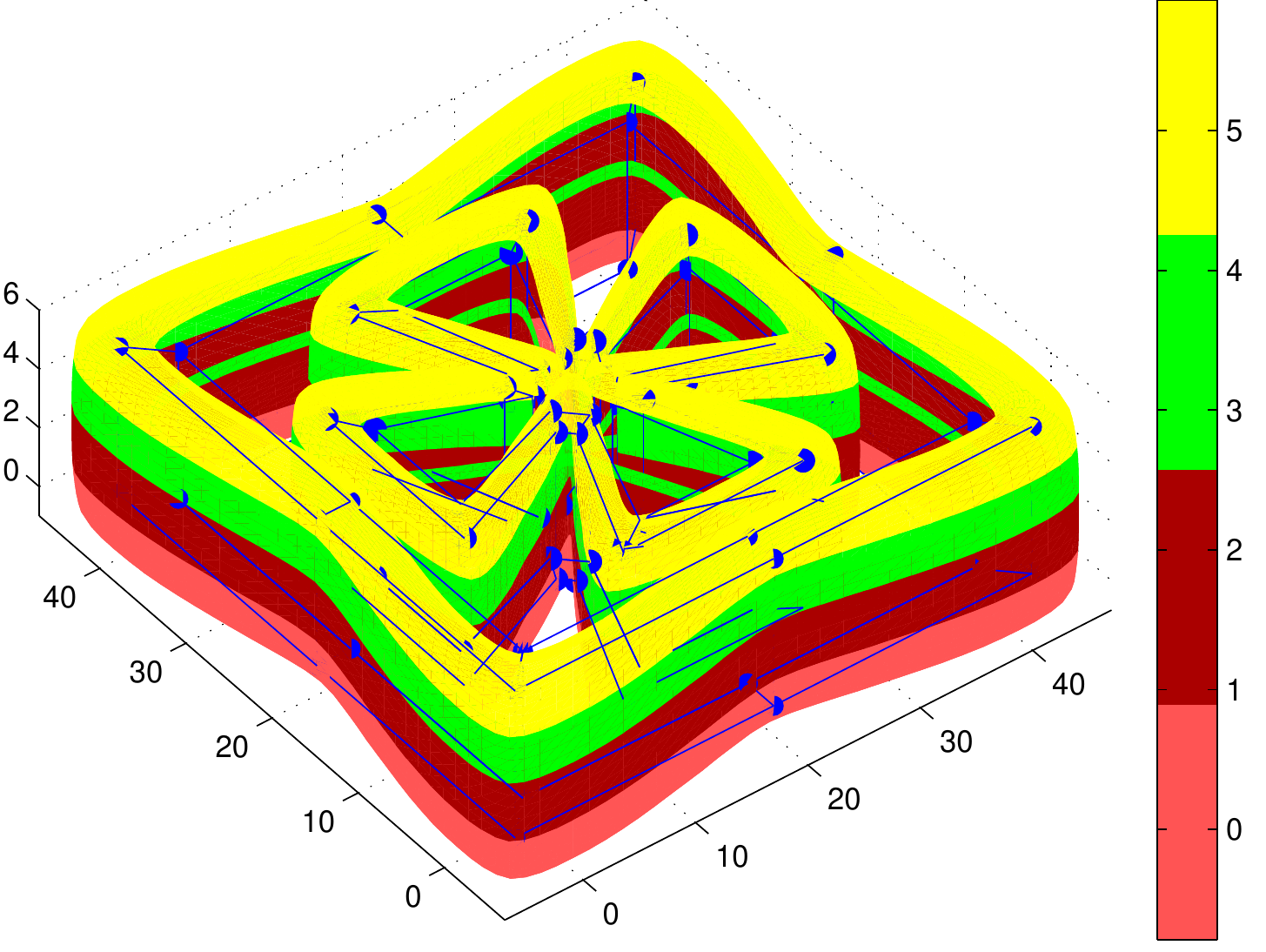, width=1.8 in}\\
(i) $\alpha=-1$ & (j) $\alpha=-0.5$ & (k) $\alpha=0$ & (l) $\alpha=0.25$ \\
\epsfig{file=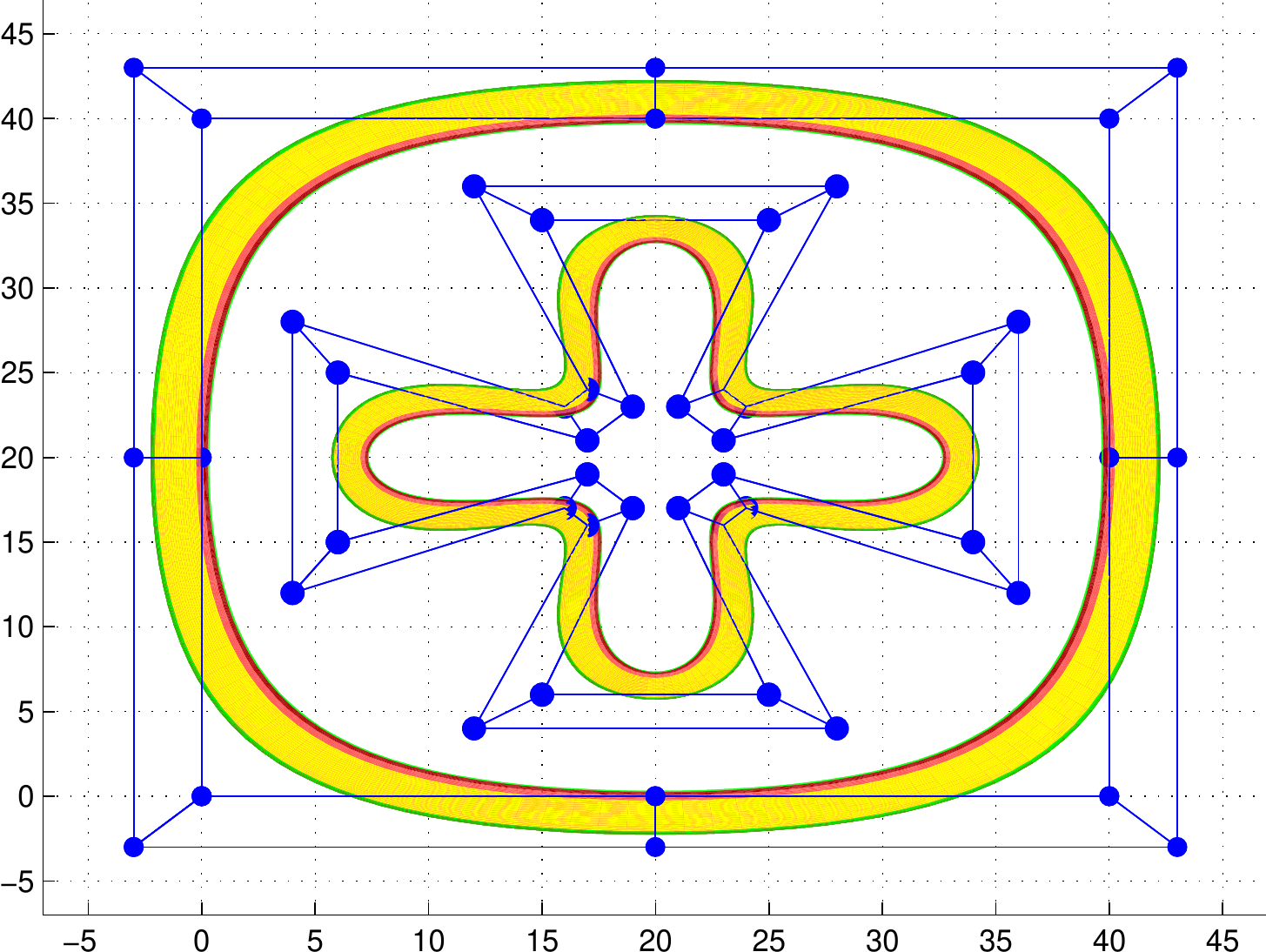, width=1.8 in} & \epsfig{file=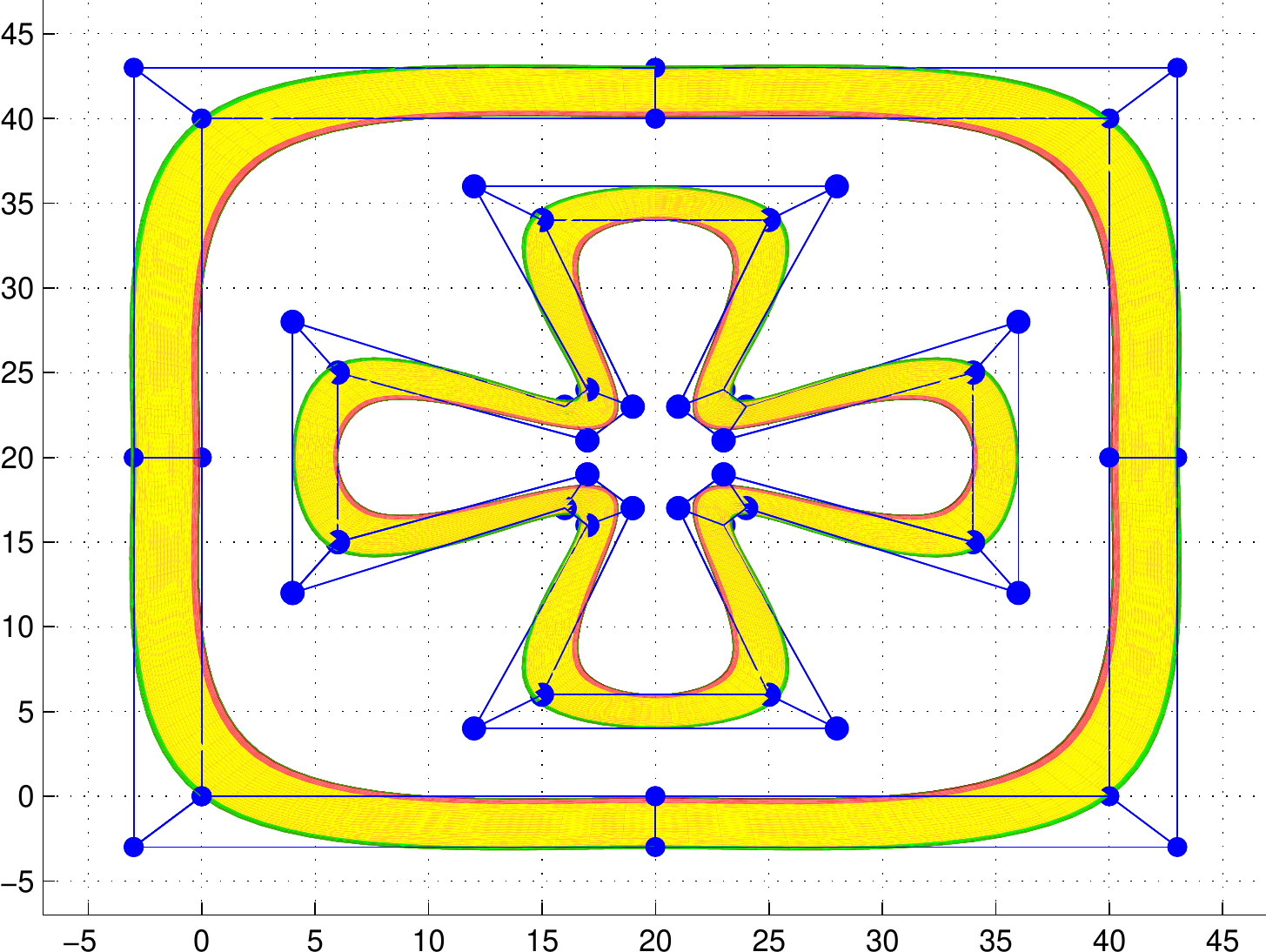, width=1.8 in} & \epsfig{file=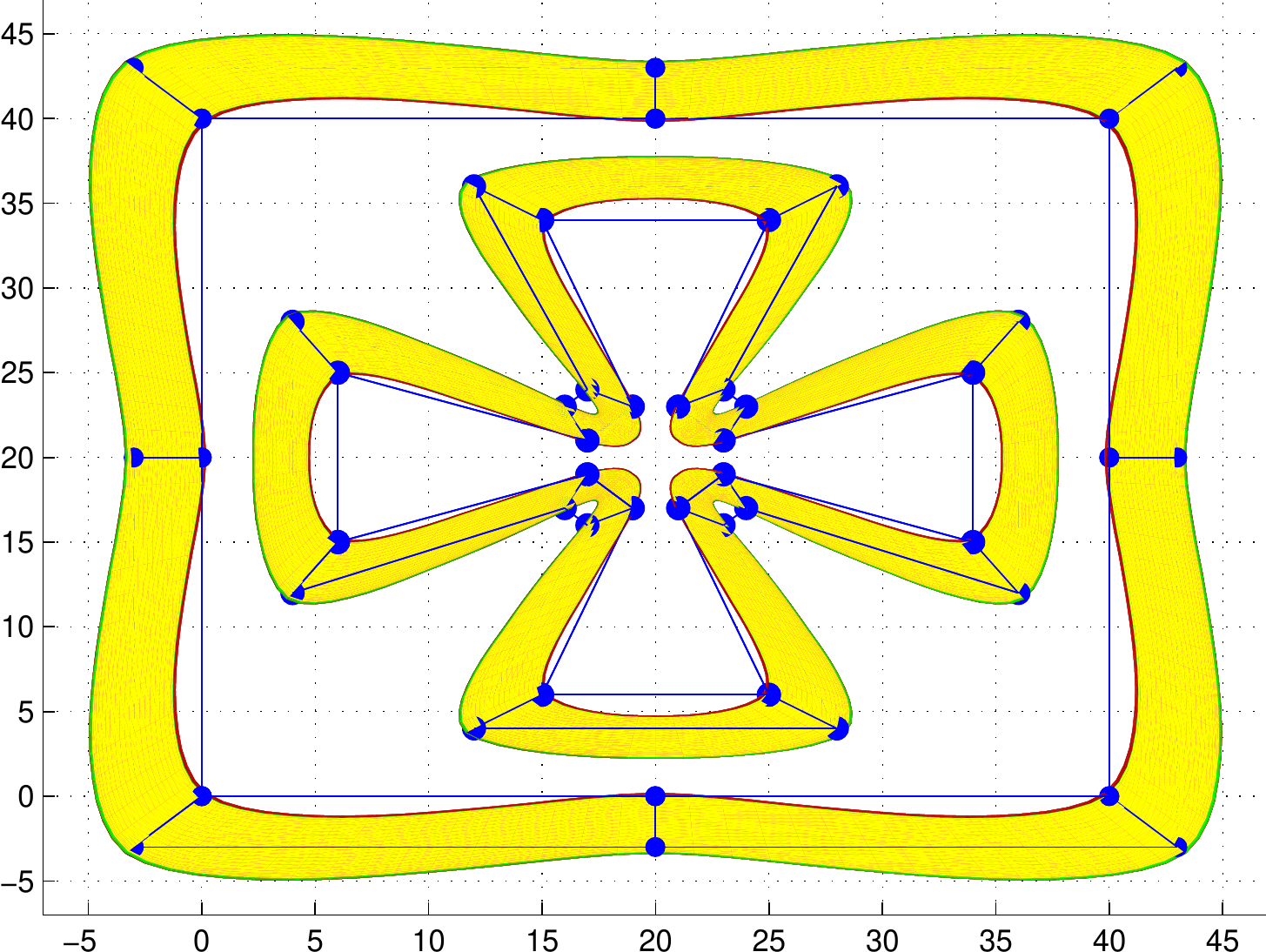, width=1.8 in} & \epsfig{file=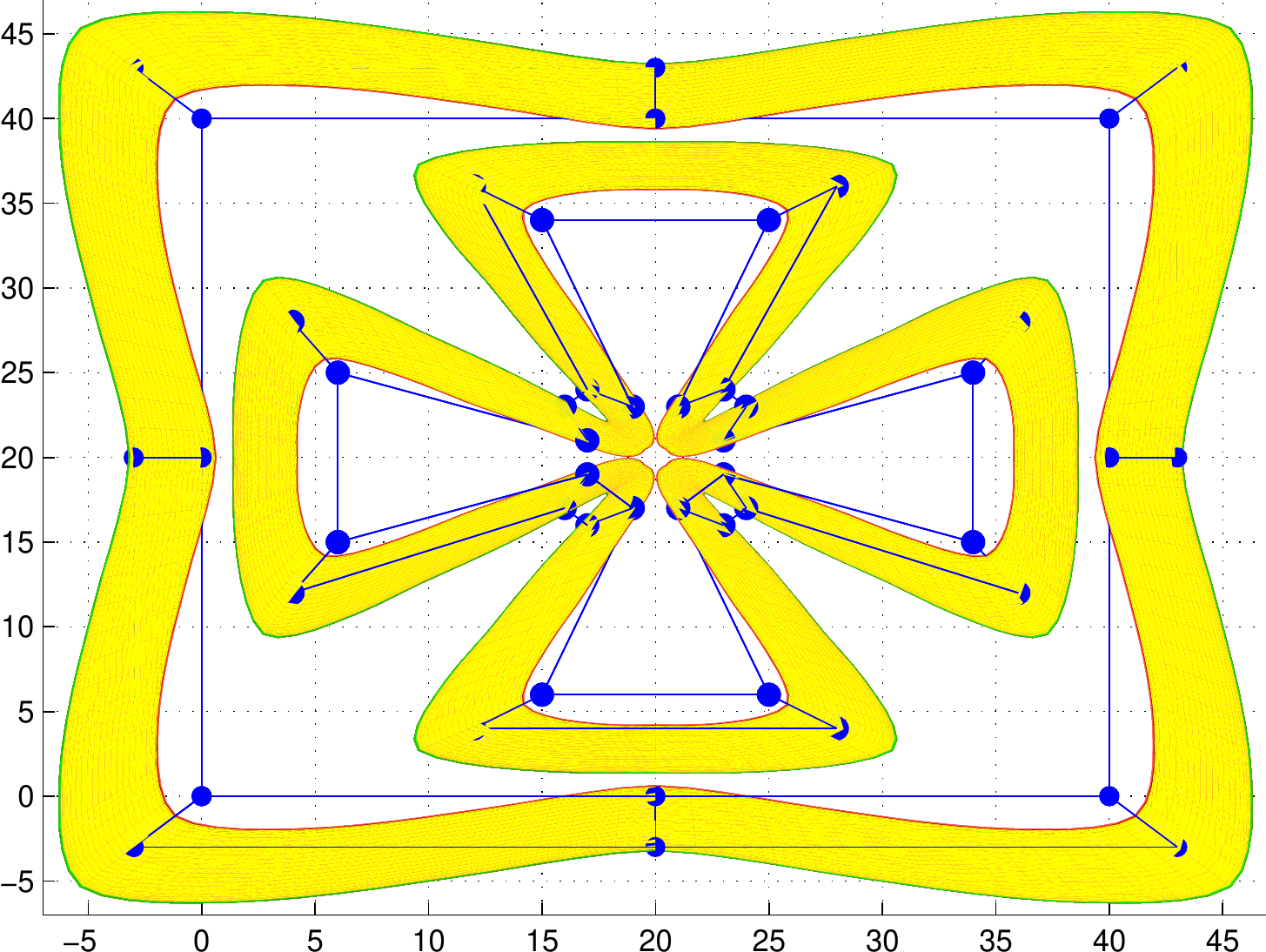, width=1.8 in}\\
(m) & (n) & (0) & (p)
\end{tabular}
\end{center}
 \caption[]{\label{surface}\emph{Surfaces produced by the tensor product of proposed $4$-point relaxed scheme after four subdivision levels.}}
\end{figure}
\end{landscape}

\section{Concluding Remarks}\label{Conclusion3}
In this article, we have derived the general refinement rules of the family of $(2n+2)$-point primal binary schemes. The proposed schemes are originated from the schemes of \cite{Dubuc}. We have also presented certain efficient algorithms with linear time complexity to check the properties of the proposed schemes. The beauty of the proposed family of schemes is that they carries the characteristics of both of its parent schemes. The proposed schemes show optimal continuity and polynomial generation. Proposed schemes reproduces polynomials up to degree one, while for specific value of tension parameters they can reproduce polynomials up to degree $2n+1$. The proposed schemes eliminate Gibbs phenomenon for negative values of tension parameter. Moreover, for a specific range of tension parameter, the proposed schemes also preserve monotonicity and convexity. The proposed schemes provide flexibility in fitting curves and surfaces, that is, using the same data points we can variate the 2D and 3D models in between or around the control polygon/mesh. Proposed schemes are a good choice for noisy data. In future, we will derive the non-uniform form of the proposed schemes.

%
%

\end{document}